\documentclass{amsart}
\usepackage{marktext}
\usepackage{gimac}
\usepackage[foot]{amsaddr}
\RequirePackage{tikz}
\usetikzlibrary{decorations.pathreplacing}
\newcommand{\GI}[2][]{\sidenote[colback=yellow!20]{\textbf{GI\xspace #1:} #2}}

\newcommand{\JN}[2][]{\sidenote[colback=orange!20]{\textbf{JN\xspace #1}: #2}}
\newcommand{\RP}[2][]{\sidenote[colback=cyan!20]{\textbf{RP\xspace #1}: #2}}

\usepackage{xr}
\newcommand{\dvx}{\nabla_x \cdot}
\newcommand{\dvy}{\nabla_y \cdot}

\newcommand{\II}{\mathit{II}}
\newcommand{\sfrac}[2]{#1 / #2}

\newdelim{\qv}{\langle}{\rangle}

\colorlet{spinefill}{cyan!15!white}

\renewcommand{\setminus}{-}

\newcommand{\no}{\nonumber}
\newcommand{\bi}{\begin{itemize}}
\newcommand{\ei}{\end{itemize}}

\newcommand{\commentout}[1]{\error}

\newcommand{\Rm}{{\mathbb R}}
\newcommand{\Zm}{{\mathbb Z}}

\newcommand{\Em}{\E}
\newcommand{\Pm}{\P}

\newcommand{\p}{\partial}
\newcommand{\ep}{\varepsilon}

\begin{document}
\title[Anomalous diffusion in combs]{Anomalous diffusion in comb-shaped domains and graphs.}

\author[Cohn]{Samuel Cohn\textsuperscript{1}}
\address{%
  \textsuperscript{1} Department of Mathematical Sciences,
  Carnegie Mellon University,
  Pittsburgh, PA 15213.}
\email{samuelcohn032@gmail.com}

\author[Iyer]{Gautam Iyer\textsuperscript{1}}
\email{gautam@math.cmu.edu}

\author[Nolen]{James Nolen\textsuperscript{2}}
\address{%
  \textsuperscript{2} Duke University,
  Department of Mathematics,
  243 Physics Building,
  Durham, NC 27708.}
\email{nolen@math.duke.edu}

\author[Pego]{Robert L. Pego\textsuperscript{1}}
\email{rpego@cmu.edu}
\thanks{%
  This material is based upon work partially supported by
  the National Science Foundation under grants
  DMS-1252912, 
  DMS-1351653, 
  DMS-1515400, 
  DMS-1814147, 
  and the Center for Nonlinear Analysis.
}
\begin{abstract}
  In this paper we study the asymptotic behavior of Brownian motion in both comb-shaped planar domains, and comb-shaped graphs.
  We show convergence to a limiting process when both the spacing between the teeth \emph{and} the width of the teeth
  vanish at the same rate. 
  The limiting process exhibits an anomalous diffusive behavior and can be described as a Brownian motion time-changed by the local time of an independent sticky Brownian motion.
  In the two dimensional setting the main technical step is an oscillation estimate for a Neumann problem, which we prove here using a probabilistic argument.
  In the one dimensional setting we provide both a direct SDE proof, and a proof using the trapped Brownian motion framework in Ben Arous \etal (Ann.\ Probab.\ '15).
\end{abstract}
\maketitle
\section{Introduction.}\label{s:intro}

Diffusion in comb-like structures arises in the study of several applications such as the study of linear porous media, microscopically disordered fluids, transport in dendrites and tissues (see for instance~\cites{Young88,ArbogastDouglasEA90,ShowalterWalkington91,BressloffEarnshaw07,DagdugBerezhkovskiiEA07}
and references therein).
Our aim in this paper is to study idealized, periodic, comb-shaped domains in $\R^2$ under scaling regimes where an anomalous diffusive behavior is observed.
We also study scaling limits of a skew Brownian motion on an infinite comb-shaped graph.
In both scenarios we show that under a certain scaling the limiting process is a Brownian motion time-changed by the local time of an independent sticky reflected Brownian motion.
We describe each of these scenarios separately in Sections~\ref{s:ifatcomb} and~\ref{s:ithincomb} below.

\subsection{Anomalous Diffusion in Comb-Shaped Domains.}\label{s:ifatcomb}

Let $h_0 \in (0, \infty]$, and $\alpha, \epsilon > 0$, and let $\Omega_\epsilon \subset \R^2$ be the fattened comb-shaped domain defined by
\begin{equation}\label{e:OmegaEp}
  \Omega_\epsilon = \set{
    (x, y) \in \R^2 \; \st -\epsilon < y < h_0 \one_{B(\epsilon \Z, \alpha \epsilon^2/2)}(x)
  }\,,
\end{equation}
where $B(\epsilon \Z, \alpha \epsilon^2/2) \subseteq \R$ denotes the $\alpha \epsilon^2/2$ neighborhood of $\epsilon \Z$, and $\one$ denotes the indicator function.
Figure~\ref{f:fatcomb} shows a picture of the domain~$\Omega_\epsilon$.
We refer to the region where $-\epsilon < y < 0$ as the spine; $\Omega_\epsilon$ also has teeth of height $h_0$ and width~$\alpha \epsilon^2$, which are spaced $\epsilon$ apart.
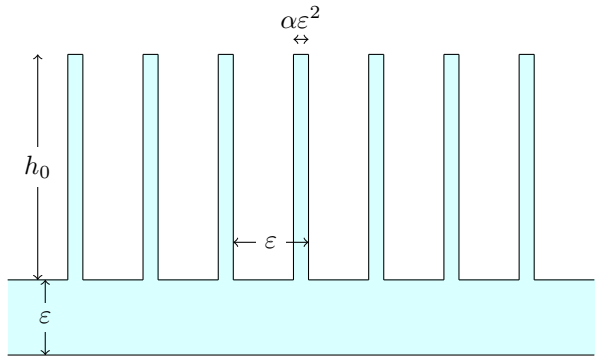
\begin{figure}[hbt]
  \begin{center}
    \begin{tikzpicture}
      \fill[spinefill] (-3.9,0) rectangle (3.9,1);
      \draw (-3.9,0) -- (3.9,0);
      \foreach \x [count=\n]in {-3.1,-2.1,-1.1,-.1,.9,1.9,2.9}{
	\fill[spinefill] (\x,1) rectangle ++(.2,3);
	\draw (\x,1) -- ++(-.8,0);
	\draw (\x,4) -- ++(.2,0);
      }
      \foreach \x [count=\n]in {-3.1,-2.9,-2.1,-1.9,-1.1,-.9,-.1,.1,.9,1.1,1.9,2.1,2.9,3.1}{
	\draw (\x,1) -- ++(0,3);
      }

      \draw (3.1,1) -- ++(.8,0);

      \draw[<->] (-.1,4.2) -- (.1,4.2);
      \draw node at (0,4.5) {$\alpha \epsilon^2$};

      \coordinate (A) at (-.9,1.5);
      \coordinate (B) at ( 0.1,1.5);
      \draw[<->] (A) -- (B) node[midway,fill=white] {$\epsilon$};

      \coordinate (A) at (-3.5,1);
      \coordinate (B) at ( -3.5,4);
      \draw[<->] (A) -- (B) node[midway,fill=white] {$h_0$};

      \coordinate (A) at (-3.4,1);
      \coordinate (B) at ( -3.4,0);
      \draw[<->] (A) -- (B) node[midway,fill=spinefill] {$\epsilon$};

    \end{tikzpicture}
    \caption{Image of the comb-shaped domain~$\Omega_\epsilon$.
      The teeth have width $\alpha \epsilon^2$ and height $h_0$.
      The spine has width $\epsilon$, and the teeth are spaced a distance of~$\epsilon$ apart.}
    \label{f:fatcomb}
  \end{center}
\end{figure}

Let $Z^\epsilon = (X^\epsilon, Y^\epsilon)$ be a Brownian motion in $\Omega_\epsilon$ that is reflected normally on the boundary $\partial \Omega_\epsilon$.
Our aim is to study the limiting behavior of~$Z^\epsilon$ as~$\epsilon \to 0$.
This is an idealized, two dimensional, version of the arterial flow models considered by Young~\cite{Young88}.
Note that the process $Z^\epsilon$ may travel large horizontal distances when it is in the spine, but travels only negligible horizontal distances when it is ``trapped'' inside the teeth.
From the shape of~$\Omega_\epsilon$, one expects that the chance $Z^\epsilon$ wanders into the teeth from the spine is of order $\alpha \epsilon$.
Since the teeth are spaced $\epsilon$ apart, the process $Z^\epsilon$ encounters $O(1/\epsilon)$ teeth after traveling an~$O(1)$ distance horizontally.
These balance, and after large horizontal distances, the process $Z^\epsilon$ spends comparable amounts of time in the spine and in the teeth.
This leads us to expect that the limiting horizontal behavior of~$Z^\epsilon$ should be described by a Brownian motion that is time-changed so that it only moves when the process is in the spine -- this is our main result.

To state the result, we let $\Omega_0 \defeq \Rm \times [0,h_0]$, and let $\pi_\epsilon\colon \Omega_\epsilon \to \Omega_0$ be defined by $\pi_\epsilon(x, y) = (x, y^+)$, where $y^+ = \max\set{y, 0}$ denotes the positive part of~$y$.
Given a probability measure~$\mu^\epsilon$ on~$\Omega_\epsilon$, let~$\pi_\epsilon^*(\mu^\epsilon)$ denote the push forward of~$\mu^\epsilon$,  under the map $\pi_\epsilon$, to a probability measure on $\Omega_0$.
We can now state the main result.



\begin{theorem}\label{t:zlimfat}  
  Let $Z^\epsilon = (X^\epsilon,Y^\epsilon)$ be a normally reflected Brownian motion in $\Omega_\epsilon$ with initial distribution $\mu^\epsilon$.
  If the sequence of measures $(\pi_\epsilon^* (\mu^\epsilon) )$ converges weakly to a probability measure $\mu$ on~$\Omega_0$, then the sequence of processes $Z^{\epsilon,+} \defeq \pi_\epsilon(Z^\epsilon)$ converges weakly as $\epsilon \to 0$.
  \GI[2019-04-18]{I changed $Z^\epsilon$ to $(Z^\epsilon)^+$, and introduced~$\pi_\epsilon^+$.
    I think it needs to be changed everywhere for the fat comb.}
  The limiting process, denoted by $Z = (X,Y)$, can be described as follows.
  The initial distribution of~$Z$ is~$\mu$.
  The process $Y$ is a Brownian motion on $(0, h_0)$, which is normally reflected at $h_0$ if $h_0 < \infty$, and is stickily reflected (with parameter $1/\alpha$) at $0$.
  The process $X$ is a time-changed Brownian motion given by
  \begin{equation}\label{e:Xtc}
    X_t = \bar W_{ \frac{2}{\alpha}L^{Y}_t(0) }\,,
  \end{equation}
  where $\bar W$ is a Brownian motion on $\R$ that is independent of $Y$, and $L^Y(0)$ is the local time of~$Y$ at $0$.
\end{theorem}

To clarify notation, we follow the normalization convention of~\cite{KaratzasShreve91}, and define local time of $Y$ at $0$ by
\[
  L_t^Y(0)
    = \lim_{\delta \to 0} \frac{1}{2 \delta} \int_0^t \one_{\{0 \leq Y_s \leq \delta \}} \,d \qv{Y}_s
    = \lim_{\delta \to 0} \frac{1}{2 \delta} \int_0^t \one_{\{0 < Y_s \leq \delta \}} \,ds \,.
\]
In the second equality above we note that the strict inequality $0 < Y_s$ in the integrand is crucial, as the process $Y$ spends a non-negligible time at $0$.
Indeed, recall that the sticky reflection of the process $Y$ at $0$ is characterized by the local time relation
\begin{equation*}
  2\, dL^Y_t(0) = \alpha \one_{\set{Y_t = 0}} \, dt\,.
\end{equation*}
Such a process can be constructed explicitly by time changing a reflected Brownian motion, or by using the Hille-Yosida theorem.
We elaborate on this in Section~\ref{s:limitprocess}, below.

We remark that while the statement of Theorem~\ref{t:zlimfat} is intuitive, the proof isn't as simple.
The broad outline of the proof follows techniques introduced by Freidlin and Wentzell (see for instance Theorem 8.2.2 in~\cite{FreidlinWentzell12}) and the structure in~\cites{HairerKoralovEA16,HairerIyerEA18}.
However, the key step in establishing the required estimates requires balancing the time spent by~$Z^\epsilon$ in the spine with the local time at the interface between the teeth and spine.
In order to prove this, we require an oscillation estimate on the solution to a certain Neumann problem (Proposition~\ref{p:uosc1}, below).

To the best of our knowledge, the oscillation estimate we require can not be obtained by standard techniques for the following reasons:
First, for the problem at hand energy methods only provide estimates with domain dependent constants.
Since~$\Omega_\epsilon$ varies with~$\epsilon$ these constants may degenerate as~$\epsilon \to 0$.
Second, since we impose Neumann boundary conditions on the entire boundary we may not easily use techniques based on the comparison principle.
We prove the oscillation estimate here directly by using a probabilistic argument, and this comprises the bulk of the proof of Theorem~\ref{t:zlimfat}.
\medskip

Notice that Theorem~\ref{t:zlimfat} immediately yields the behavior of the variance of the horizontal displacement.
This question has been studied by various authors (see for instance~\cite{BerezhkovskiiDagdugEA14} and references therein), and is of interest as it is an easily computable benchmark indicating anomalous diffusion.
\begin{corollary}\label{c:var}
  If~$h_0 < \infty$ then
  \begin{subequations}
  \begin{gather}
    \label{e:varShort}
    \lim_{t \to 0}
    \lim_{\epsilon \to 0}
    \frac{1}{t} \E^{(x,0)} \abs{X^\epsilon_t - x}^2  = 1\,,
    \\
    \label{e:varLong}
    \lim_{t \to \infty} \lim_{\epsilon \to 0}
    \frac{1}{t} \E^{(x,0)} \abs{X^\epsilon_t - x}^2  = \frac{1}{\alpha h_0 + 1}\,.
  \end{gather}
  \end{subequations}
  If~$h_0 = \infty$, then~\eqref{e:varShort} still holds.
  However, instead of~\eqref{e:varLong} we have
  \begin{equation}\label{e:varLongInf}
    \lim_{t \to \infty} \lim_{\epsilon \to 0} \frac{1}{\sqrt{t}}
      \E^{(x,0)} \abs{X^\epsilon_t - x}^2  = \frac{1}{\alpha} \paren[\Big]{\frac{8}{\pi}}^{1/2}\,.
  \end{equation}
\end{corollary}

Here we clarify that the notation $\E^{(x, 0)}$ refers to the expectation under the probability measure~$\P^{(x,0)}$ under which $(X^\epsilon_0, Y^\epsilon_0) = (x, 0)$ almost surely.
Note that when $h_0 < \infty$, the variance is asymptotically linear with slope~$1$ at short time, and asymptotically linear at long time with slope strictly smaller than~$1$.
On the other hand, when $h_0 = \infty$ the variance is asymptotically linear for short time, and asymptotically $O(\sqrt{t})$ for long time, indicating an anomalous sub-diffusive behavior on long time scales.
This was also previously observed by Young~\cite{Young88}.

In addition to the variance, another quantity of interest is the limiting behavior of the probability density function.
This is essentially a PDE homogenization result that also follows quickly from Theorem~\ref{t:zlimfat}.
Explicitly, let~$u^\epsilon$ represent the concentration density of a scalar diffusing in the region $\Omega_\epsilon$.
When the diffusivity is normalized to be~$1/2$, and the boundaries are impermeable the time evolution of~$u^\epsilon$ is governed by the heat equation with Neumann boundary conditions:
\begin{subequations}
\begin{alignat}{2}
  \label{e:heat1}
  \span
    \partial_t u^\epsilon - \frac{1}{2} \lap u^\epsilon = 0
  &\qquad&
    \text{in } \Omega_\epsilon
  \\
  \label{e:heat2}
  \span
    \partial_\nu u^\epsilon = 0
  &&
    \text{on } \partial \Omega_\epsilon\,.
\end{alignat}
\end{subequations}
Using Theorem~\ref{t:zlimfat} we can show that $u^\epsilon$ converges as $\epsilon \to 0$, and obtain effective equations for the limit.
The same equations were also obtained heuristically by Young~\cite{Young88}.
\begin{corollary}\label{c:pde}
  Let $u_0 \colon \Omega_0 \to \R$ be a bounded continuous function, and let $u^\epsilon$ be the solution to~\eqref{e:heat1}--\eqref{e:heat2} with initial data~$u_0 \circ \pi_\epsilon$.
  Let $\mu^\epsilon$ be a family of test probability measures on $\Omega_\epsilon$ such that $(\pi_\epsilon^*(\mu^\epsilon))$ converges weakly to a probability measure~$\mu$ on~$\Omega_0$.
  Then for any $t > 0$ we have
  \begin{equation}\label{e:uepConv}
    \lim_{\epsilon \to 0} \int_{\Omega_\epsilon} u^\epsilon(z, t) \, d\mu^\epsilon(z)
      = \int_{\Omega_0} u(z, t) \, d\mu(z)\,,
  \end{equation}
  where $u \colon \Omega_0 \to \R$ is the unique solution of the system
  \begin{subequations}
  \begin{alignat}{2}
    \label{e:rho1}
    \span
      \partial_t u - \frac{1}{2} \partial_y^2 u = 0\,,
    &\qquad&
      \text{for } t > 0,\ y \in (0, h_0)\,,
    \\
    \label{e:rho2}
    \span
      \alpha \partial_y u + \partial_x^2 u = \partial_y^2 u\,,
    &&
      \text{when }y = 0\,,
    \\
    \label{e:rho3}
    \span
      \partial_y u = 0 
    &&
      \text{when } y = h_0\,,
    \\
    \label{e:rho4}
    \span
      u = u_0
    &&
      \text{when } t = 0\,.
  \end{alignat}
  \end{subequations}
\end{corollary}

Since large scale transport only occurs in the $x$-direction, one is often only interested in the limiting behavior in this direction.
This can be obtained by taking the slice of $u$ at $y = 0$, leading to a self contained time fractional equation, similar to the Basset equation~\cite{Basset86}.
We remark that such time fractional PDEs associated with the time-changed diffusions have been studied in more generality in~\cites{BaeumerMeerschaertEA09} (see also~\cites{Cohn18,MagdziarzSchilling15}), and we refer the reader to these papers for the details.

\begin{proposition}\label{p:ftime}
  Let $v(x, t) = u( x, 0, t)$, where $u$ is the solution of~\eqref{e:rho1}--\eqref{e:rho4}.
  Then~$v$ satisfies
  \begin{equation}\label{e:ev}
    \partial_t v + \frac{\alpha}{2} \partial_t^w v - \frac{1}{2} \partial_x^2 v = \frac{\alpha}{2} f\,,
  \end{equation}
  with initial data $v(x, 0) = u_0(x, 0)$.
  The operator $\partial_t^w$ appearing above is a \emph{generalized Caputo derivative}
  defined by
  \begin{equation*}
    \partial_t^w v(x, t) \defeq \int_0^t w(t-s) \partial_t v(x, s) \, ds\,,
  \end{equation*}
  where $w$ is defined by
  \begin{equation*}
    w(t) \defeq
      \frac{2}{h_0} \sum_{k=0}^\infty \exp\paren[\Big]{
	-\frac{(2k + 1)^2 \pi^2 t}{8 h_0^2}
      }\,.
  \end{equation*}
  The function $f$ appearing  on the right of~\eqref{e:ev} can be explicitly determined in terms of $u_0$ by the identity $f = f(x, t) = \partial_y g(x, 0, t)$, where $g = g(x, y, t)$ solves
  \begin{alignat*}{2}
    \span
      \partial_t g - \frac{1}{2} \partial_y^2 g = 0
    &\qquad&
      \text{for } t > 0,\ y \in (0, h_0)\,,
    \\
    \span
      g(x, 0, t) = g(x, h_0, t) = 0
    &&
      \text{for } t > 0\,,
    \\
    \span
      g(x, y, 0) = u_0(x, y) - u_0(x, 0)
    &&
      \text{for } y \in (0, h_0),\ t = 0\,.
  \end{alignat*}
\end{proposition}

\begin{remark*}
  As we will see later, the Laplace transform of~$w$ is given by
  \begin{equation}\label{e:LW}
    \mathcal L w(s)
      = \int_0^\infty e^{-s t} w(t) \, dt
      = \frac{2\tanh \paren{ h_0 \sqrt{2 s} }}{\sqrt{2 s}} \,.
  \end{equation}
  For $h_0 = \infty$, 
  \begin{equation*}
    w(t) = \paren[\Big]{ \frac{2}{\pi t} }^{1/2}\,,
    \qquad\text{and}\qquad
    \mathcal Lw(s)
      = \paren[\Big]{\frac{2}{s}}^{1/2}\,.
  \end{equation*}
  In this case, $\partial_t^w$ is precisely $\sqrt{2} \partial_t^{1/2}$, the standard Caputo derivative of order $1/2$ (see for instance~\cite{Diethelm10}), and equation~\eqref{e:ev} becomes the Basset differential equation~\cite{Basset86}.
\end{remark*}

Finally we conclude this section with two remarks on generalizations of Theorem~\ref{t:zlimfat}.

\begin{remark}[Other scalings]
  The width of the spine and teeth may be scaled in different ways to obtain the same limiting process as in Theorem~\ref{t:zlimfat}.
  Explicitly, let
  \begin{equation*}\label{scale2}
    \tilde \Omega_\epsilon = \set{
      (x, y) \in \R^2 \; \st -w_S(\epsilon) < y < h_0 \one_{B(\epsilon \Z, w_T(\epsilon)/2)}(x)
    }\,,
  \end{equation*}
  where $w_S(\epsilon)$ and $w_T(\epsilon)$ denote the width of the spine and teeth respectively.
  We claim that Theorem~\ref{t:zlimfat} still holds (with the same limiting process), provided 
  \begin{equation}
    \lim_{\epsilon \to 0} \frac{w_T}{\epsilon w_S(\epsilon)} = \alpha \in (0, \infty)\,,
    \qquad\text{and}\qquad
    \lim_{\epsilon \to 0} w_S(\epsilon) = 0\,.  \label{genscale}
  \end{equation}
  The proof of Theorem~\ref{t:zlimfat} needs to be modified slightly to account for this more general statement.
  These modifications are described in Section~\ref{sec:otherscaling}, below.

  In the degenerate case when $\alpha = 0$, the process $Z^\epsilon$ rarely enters the teeth and the limiting behavior is simply that of a horizontal Brownian motion.
  On the other hand, if $\alpha = \infty$, then the process $Z^\epsilon$ enters the teeth too often, and the limiting behavior is simply that of a vertical, doubly reflected, Brownian motion.
\end{remark}

\begin{remark}[Higher dimensional models]
  Theorem \ref{t:zlimfat} can also be extended to analogous higher-dimensional models.
  For example, let $\Omega_\epsilon' \subseteq \R^3$ be a three dimensional ``brush'', defined by
  \[
    \Omega_\epsilon' \defeq \bigcup_{k \in \Z} (Q_k \cup T_k)\,.
  \]
  Here $Q_k$ and~$T_k$ are defined by
  \begin{align*}
    Q_k &\defeq \paren[\big]{\epsilon k-\frac{\epsilon}{2}, \epsilon k + \frac{\epsilon}{2}} \times \paren[\big]{-\frac{\epsilon}{2}, \frac{\epsilon}{2}} \times [-\epsilon,0), \\
    T_k &\defeq \set[\big]{ (x_1,x_2,x_3) \in \R^3 \st \paren{(x_1 - \epsilon k)^2 + x_2^2}^{1/2} \leq r  \epsilon^{3/2}, \quad x_3 \in  [0,h_0) }.
  \end{align*}
  In this case, the spine is the set $\cup_{k} \overline{Q_k}$, an infinite rectangular cylinder; the cylindrical teeth $T_k$ are spaced $O(\epsilon)$ apart and have radius $r\epsilon^{3/2} > 0$. If $Z^\epsilon$ is a Brownian motion in this domain with normal reflection at the boundary, then one obtains an analogous scaling limit as $\epsilon \to 0$.  The $O(\epsilon^{3/2})$ scaling of the radius of the teeth is chosen so that the ratio
  \[
    \frac{2 \text{Vol}(Q_k)}{\text{Area}( \overline{Q_k} \cap \overline{T_k} )} = \frac{2}{\pi r^2}
  \]
  is independent of $\epsilon$ -- this constant ratio plays the same role as the constant $\sfrac{2}{\alpha}$ in the comb-shaped domain~$\Omega_\epsilon$.
  While our proof of Theorem \ref{t:zlimfat} extends to this higher-dimensional version in a straight-forward way, the added modifications are technical.
  Thus, for simplicity and clarity of presentation, we only focus only on the comb-shaped domain as defined above for Theorem~\ref{t:zlimfat}.
\end{remark}

\subsection{Anomalous Diffusion in Comb-Shaped Graphs.}\label{s:ithincomb}

We now turn our attention to comb-shaped graphs, with the intention of studying a simpler version of the model in Section~\ref{s:ifatcomb} and of relating it to other work on trapped random walks.  Related random walk models on comb-shaped discrete graphs have been studied by several authors, including~\cites{BZ03,Ber06,CCFR09, CsakiCsorgoEA11}.  In each of these works, a limit process is obtained which involves a Brownian motion time-changed by the local time of an independent Brownian motion.  One difference between these other works and Theorem \ref{t:zlim} below is that the limiting processes in our result involves Brownian motion with sticky reflections, a consequence of the gluing condition described below.  More closely related to our model are the works~\cite{BenArousCerny07, BenArousCabezasEA15}, especially Section~3.2 of \cite{BenArousCabezasEA15}, where the trapping and drift of the random walk plays a role that is similar to our gluing condition.  In Section~\ref{s:tbm} below, we will use the framework in~\cite{BenArousCabezasEA15} for an alternate proof of our result in this simpler setting, illuminating the relationship between these models.  Nevertheless, the analyses in these other works do not apply to the comb-shaped domains considered in the previous Section~\ref{s:ifatcomb}, where the boundary local time of the diffusion process (pre-limit) plays an essential role.

We consider the infinite connected comb-shaped graph, $\mathcal C_\epsilon \subset \R^2$, be defined by
\begin{equation}\label{e:CepDef}
  \mathcal C_\epsilon = \paren[\big]{\R \times \set{0}} \cup \paren[\big]{ \epsilon \Z \times [0, h_0) }\,.
\end{equation}
We think of $\R \times \set{0}$ as the \emph{spine} of~$\mathcal C_\epsilon$, and $\epsilon \Z \times [0, h_0)$ as the infinite collection of teeth.
The teeth meet the spine at the junction points $J_\epsilon \subseteq \mathcal C_\epsilon$ defined by
\begin{equation}\label{e:JepDef}
  J_\epsilon \defeq (\epsilon \Z) \times \set{0}\,,
\end{equation}
and is depicted in Figure~\ref{f:thincomb}.
\begin{figure}[hbt]
  \begin{center}
    \begin{tikzpicture}
      \foreach \x [count=\n]in {-4,-3.5,...,4}{
	\node at (\x,0) [circle,fill=black,inner sep=0pt,minimum size=4pt] {};
	\draw [thick] (\x,0) -- ++(0,3);
      };
      \draw [thick] (-4.25,0) -- (4.25,0);
      \draw[<->] (-4.75,0) -- ++(0,3) node[midway,fill=white] {$h_0$};
      \draw[decorate,decoration={brace,mirror}] (0,-.2) -- ++(.5,0)
	node[pos=.5, anchor=north, yshift=-2pt] {$\epsilon$};

    \end{tikzpicture}
    \caption{Image of the comb-shaped graph~$\mathcal C_\epsilon$.
      The teeth are spaced $\epsilon$ apart and have height $h_0$.}\label{f:thincomb}
  \end{center}
\end{figure}
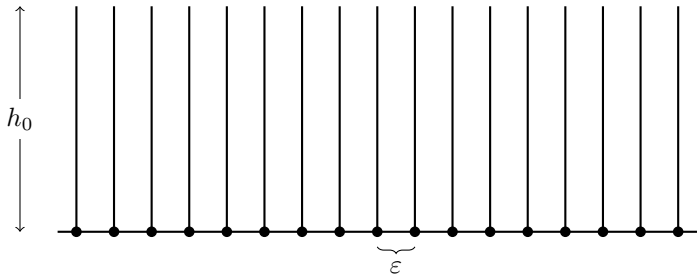

Let~$Z^\epsilon = (X^\epsilon, Y^\epsilon)$ be a diffusion on~$\mathcal C_\epsilon$ such that away from the junction points~$J_\epsilon$, the process~$Z^\epsilon$ is a standard Brownian motion.
If $h_0 < \infty$, we reflect~$Z^\epsilon$ at the ends of the teeth.
At the junction points, we specify a ``gluing condition'' that dictates~$Z^\epsilon$ enters the teeth with probability~$\alpha \epsilon / (2 + \alpha \epsilon)$, and stays in the spine with probability~$2 / (2 + \alpha \epsilon)$.
One can formulate this precisely by requiring the local time balance
\begin{equation*}
  L^{X^\epsilon}_t(J_\epsilon) = \frac{2}{2 + \alpha \epsilon} L^{Z^\epsilon}_t(J_\epsilon)\,,
  \qquad
  L^{Y^\epsilon}_t(J_\epsilon) = \frac{\epsilon}{2 + \alpha \epsilon} L^{Z^\epsilon}_t(J_\epsilon)\,,
\end{equation*}
at the junction points, and we describe this further in Section~\ref{s:thincomb}.
Alternately, one can make the gluing condition precise by using the excursion decomposition of~$Z^\epsilon$, and we do this in Section~\ref{s:excursion}.

Clearly the mechanics of the above diffusion on the comb-shaped graph~$\mathcal C_\epsilon$ shows that it is a simplified model of the diffusion on the comb-shaped domain~$\Omega_\epsilon$.
Our main result in this section shows convergence of~$Z^\epsilon$ to the same limit process as that in Theorem~\ref{t:zlimfat}.

\begin{theorem} \label{t:zlim}
Let $(\mu^\epsilon)$ be sequence of probability measures on~$\mathcal C_\epsilon$ which converge weakly to a probability measure $\mu$ on $\Omega_0 \defeq \R \times [0,h_0]$.
Let $Z^\epsilon$ be the above graph diffusion with initial distribution~$\mu^\epsilon$.
Then, as $\epsilon \to 0$, the processes~$Z^\epsilon$ converge weakly to the same limit process~$Z = (X, Y)$ defined in Theorem~\ref{t:zlimfat}.
\end{theorem}

The proof of Theorem~\ref{t:zlim} is technically and conceptually much simpler than that of Theorem~\ref{t:zlimfat}, and is presented in Section~\ref{s:thincomb}.
 Moreover, the excursion decomposition of~$Z^\epsilon$ on the comb-shaped graph~$\mathcal C_\epsilon$ allows for an elegant proof using time changes and the trapped Brownian motion framework in~\cites{BenArousCabezasEA15}.
We present this approach in Section~\ref{s:excursion}.

The process process~$Z^\epsilon$ on the comb-shaped graph~$\mathcal C_\epsilon$ is closely related to a model of fluid flow in fissured media, where trapping in microscopic regions of low permeability yields a macroscopic anomalous diffusive effect. Explicitly, consider medium composed of two materials: a set of \emph{blocks}, where the permeability is relatively low, and~\emph{fissures} where the permeability is relatively high (see for instance~\cites{ArbogastDouglasEA90,ShowalterWalkington91,BourgeatLuckhausEA96}).
Assuming that the region occupied by the fissures is connected and that the blocks are arranged periodically, the fluid flow in this situation is modeled by the equation
\begin{equation*}
  \partial_t u^\epsilon - \dv \paren[\big]{ a^\epsilon \grad u^\epsilon } = f\,,
  \qquad
  a^\epsilon(x)
    = \one_F\paren[\Big]{\frac{x}{\epsilon}} a\paren[\Big]{\frac{x}{\epsilon}} 
      + \epsilon^2 \one_B\paren[\Big]{\frac{x}{\epsilon}} A\paren[\Big]{\frac{x}{\epsilon}}\,.
\end{equation*}
Here $a, A$ are uniformly elliptic matrices representing the permeability in the fissures and blocks respectively, and $F, B$ denote the region occupied by the blocks and fissures respectively. For this linear model, Clark~\cite{Clark98} proved that as $\epsilon \to 0$, the functions $u^\epsilon$ two-scale converges to a function $U = U(x, y, t)$ that satisfies a coupled system, called the double-porosity model, in which the fluid in the fissures is driven in a non-local manner by the fluid in the blocks.

To understand this model probabilistically, one could study a diffusion~$\tilde Z^\epsilon$ whose generator is~$\dv a^\epsilon \grad$.
Inside the fissures, the process~$\tilde Z^\epsilon$ diffuses freely until it hits the boundary of a block.
Upon hitting a block boundary, the contrast between the block and fissure permeabilities dictates that~$\tilde Z^\epsilon$ enters the blocks with probability~$O(\epsilon)$, and remains in the fissures with probability $1 - O(\epsilon)$.
Since the blocks have diameter $O(\epsilon)$, and the permeability there is $O(\epsilon^2)$, the excursions of $\tilde Z^\epsilon$ into the blocks take~$O(1)$ amount of time. These characteristic features are exactly captured by the above comb model: the spine plays the role of the fissures and the teeth play the role of the blocks (rescaled to have size~$1$), and our gluing condition dictates that~$Z^\epsilon$ enters the teeth with probability~$O(\epsilon)$.

\subsection*{Plan of this paper}

The rest of the paper is organized as follows.
We begin by describing the limit process~$Z$, and study its basic properties in Section~\ref{s:limitprocess}.
Next, in Section~\ref{s:fatcomb} we prove Theorem~\ref{t:zlimfat} and all the required lemmas.
In Section~\ref{s:thincomb} we prove Theorem~\ref{t:zlim} on the comb-shaped graph~$\mathcal C_\epsilon$.
The proof is similar to that of Theorem~\ref{t:zlimfat}, but the technicalities are much simpler.
Finally, in Section~\ref{s:thincomb} we provide an alternate proof of Theorem~\ref{t:zlim} using the trapped Brownian motion framework in~\cite{BenArousCabezasEA15}.
\section{The Limit Process.}\label{s:limitprocess}

Before proving our main results in this paper, we give a more thorough description of the limit process $Z = (X,Y)$.
There are two canonical constructions of this process.
The first, relatively well-known construction involves directly writing $Y$ as a time-changed Brownian motion, and this is presented in Section~\ref{s:timechange}.
The second construction involves a characterization using the generator.
While the technicalities using this second approach are more involved, they relate to the PDE analogue and immediately yield Corollary~\ref{c:pde}.

\begin{remark}\label{r:h0eq1}
  The process~$Z$ depends on the parameters~$\alpha > 0$, and $h_0 \in (0, \infty]$.
  To simplify the presentation, we will subsequently assume~$h_0 = 1$.
  The case $h_0 = \infty$ may be handled by replacing the normal reflection at $1$ with a diffusion on the semi-infinite interval $(0, \infty)$.
\end{remark}


\subsection{Construction via Time Changes.}\label{s:timechange}

We begin by constructing the limit process~$Z$ using a time-changed Brownian motion.
To construct the process~$Y$, let $\bar{B_t}$ be a standard doubly reflected Brownian motion on the interval $(0, 1)$.
(Recall that in Remark~\ref{r:h0eq1} we assumed~$h_0 = 1$ for simplicity.)
Let $L^{\bar{B}}_s(0)$ be the local time of $\bar B$ at $0$, and define
\begin{equation*}
  \varphi(s) \defeq s + \frac{2}{\alpha}L^{\bar{B}}_s(0), \quad s \geq 0 \,.
\end{equation*}
Let $T$, defined by
\begin{equation}\label{e:Tdef}
  T(t) = T_t \defeq \varphi^{-1}(t) = \inf \set{ s \geq 0 \st \varphi(s) \geq t }\,,
\end{equation}
denote the inverse of $\varphi$.
Since $\varphi$ is strictly increasing, note that $T$ is continuous.
Thus the process~$Y$, defined by
\begin{subequations}
\begin{equation}\label{e:limitdef1}
  Y_t \defeq \bar B_{T_t}\,,
\end{equation}
is a continuous process on $[0, 1]$.
Clearly, on any interval of time where $Y$ remains inside the interval $(0, 1]$, trajectories of $Y$ and $\bar B$ are identical.
When $Y$ hits $0$, however, the trajectories are slowed down on account of the time change~$T$.
The behavior at~$0$ is known as a \emph{sticky reflection with parameter~$1/\alpha$} at~$0$, and we refer the reader to~\cite[14, \S5.7]{ItoMcKean74} or the original papers of Feller~\cites{Feller52,Feller54} for more details.

Clearly once the process~$Y$ is known, the process~$X$ can be recovered using~\eqref{e:Xtc}, reproduced here for convenience:
\begin{equation}\label{e:limitdef2}
  X_t \defeq \bar W_{ \frac{2}{\alpha}L^{Y}_t(0) }\,.
\end{equation}
\end{subequations}
Here~$\bar W$ is standard one dimensional Brownian motion that is independent of $\bar B$.
Intuitively, we think of $\R \times \{0\}$ as the spine of the limiting comb, and $\Rm \times (0,h_0]$ as the continuum of teeth.
The process $T_t$ may be interpreted as the time accumulated in the teeth, and $ \frac{2}{\alpha}L^Y_t(0)$ is the time accumulated in the spine. 

\subsection{The SDE Description.}
We now describe the process $Z = (X, Y)$ via a system of SDEs.
Let $W$ and $B$ be two independent standard one dimensional Brownian motions.
We claim that the process $Z$ can be characterized as the solution of the system of SDEs
\begin{subequations}
\begin{gather}
  \label{e:sdeX}
  dX_t = \one_{\set{Y_t = 0}} \, d W_t \,,\\
  \label{e:sdeY}
  dY_t = \one_{\set{Y_t \not= 0}} \, d B_t
  - dL^Y_t(1)
  + dL^Y_t(0)\,,\\
  \label{e:localtime}
  \alpha \one_{\set{Y_t = 0}} \, dt = 2\, dL^Y_t(0)\,,
\end{gather}
\end{subequations}
with initial distribution~$\mu$.
Existence of a process $Z$ satisfying~\eqref{e:sdeX}--\eqref{e:localtime} can be shown abstractly using the Hille-Yosida theorem, and we refer the reader to~\cite{Cohn18} for the details.
Instead, we will show existence by showing that the process~$Z$ constructed in the previous section is a solution to~\eqref{e:sdeX}--\eqref{e:localtime}.

\begin{lemma}\label{l:sdeZ}
  The process~$Z = (X, Y)$ defined by~\eqref{e:limitdef1}--\eqref{e:limitdef2} is a weak solution to the system~\eqref{e:sdeX}--\eqref{e:localtime}.
\end{lemma}

The proof of Lemma~\ref{l:sdeZ} boils down to an SDE characterization of sticky Brownian motion that was recently shown by Engelbert and Peskir~\cite{EngelbertPeskir14}.
We remark that in~\cite{EngelbertPeskir14} the authors also show weak uniqueness of the appropriate SDE.
While we present the proof of existence below, we refer the reader to~\cite{EngelbertPeskir14} for the proof of uniqueness.
%

\begin{proof}
  By the Tanaka formula we have
  \begin{equation}\label{e:Tanaka1}
    \bar B_t = \tilde{B}_t + L_t^{\bar B}(0) - L_t^{\bar B}(1) \,,
  \end{equation}
  where $\tilde B$ is a Brownian motion. Since $T_t$ is a continuous and increasing time change, $\tilde B_{T_t}$ is still a continuous martingale, $L^Y_t(0) = L^{\bar B}_{T_t}(0)$ and $L^Y_t(1) = L^{\bar B}_{T_t}(1)$.
  Note first
  \begin{equation}
  \alpha\int_0^t \one_{\set{Y_s = 0}} \, ds = \alpha\int_0^t \one_{\set{\bar B_{T_s} = 0}} \, d\varphi(T_s) = \alpha\int_0^{T_t}\one_{\set{\bar B_s = 0}} \, d\varphi(s).
  \end{equation}
  Then since $\set{t \st \bar B_t = 0}$ has Lebesgue measure 0 and $L_t^{\bar B}$ only increases on this set, we decompose $\alpha\varphi(s) = \alpha s + 2L^{\bar B}_s$ to obtain
  \begin{equation}
  \alpha\int_0^{T_t}\one_{\set{\bar B_s = 0}} \, d\varphi(s) = 2\int_0^{T_t}\one_{\set{\bar B_s = 0}} dL^{\bar B}_s(0) = 2L^{\bar B}_{T_t}(0) = 2 L^Y_t(0) \,,
  \end{equation}
  which implies \eqref{e:localtime}. Notice that since $(2/\alpha)L^Y_t(0)$ is independent of $\bar W$, $X_t$ is a martingale with quadratic variation 
  \begin{equation}\label{e:qvX}
  \qv{X}_t = \frac{2}{\alpha}L_t^Y(0) \,. 
  \end{equation}
  In addition we have 
  \begin{equation*}
  \qv{\tilde B_T}_t = T_t\,.
  \end{equation*}
  Thus, for the process~$B$ defined by
  \begin{equation}\label{e:Bdef}
  B_t \defeq \tilde B_{T_t} + \bar W_{\frac{2}{\alpha}L^Y_t(0)}\,,
  \end{equation}
  we have $\qv{B}_t = t$.
  For the filtration, we let
  \begin{equation*}
    \mathcal G_t
      = \sigma\paren[\Big]{ \mathcal N \cup \mathcal F^{\bar B}_{T_t} \cup \mathcal F^{X}_t }
  \end{equation*}
  where $\mathcal N$ denotes the collection of all $\mathcal F^{(\bar B, \bar W)}_\infty$-null sets.  Since $\bar B$ and~$\bar W$ are independent, it is easy to see that for all $s \geq 0$, $X_t - X_s$ is independent of $\mathcal G_s$, and both $\tilde B_{T_t}$ and $X_t$ are $\mathcal G$-martingales.
  Thus, $B$ is also a $\mathcal G$-martingale, and by L\'evy's criterion must be a Brownian motion.

  Now~\eqref{e:sdeX}--\eqref{e:sdeY} follow from~\eqref{e:localtime}, \eqref{e:Bdef} and the fact that
  \begin{equation*}
  \int_0^t \one_{\set{Y_s = 0}} \, d\tilde B_{T_s} = 0 \quad \text{ and } \quad \int_0^t \one_{\set{Y_s \not= 0}} \, dX_s = 0\,.
  \qedhere
  \end{equation*}
\end{proof}

\subsection{Computing the Generator (Lemma~\ref{l:Zgen}).}\label{s:zgen}

We now compute the generator of~$Z$.
In the teeth (when $y > 0$) this is a standard calculation with It\^o's formula.
In the spine (when $y = 0$), however, one needs to estimate the time spent in the spine.
We state this precisely and carry out the details here.
\begin{lemma}\label{l:Zgen}
Let $\Omega_0 = \R \times [0, 1)$, and define the operator $A$ by
\begin{equation} \label{Adef}
  A \defeq \frac{1}{2} \partial_y^2\,.
\end{equation}
Define the domain of $A$, denoted by $\mathcal{D}(A)$, to be the set of all functions $g \in C_0(\Omega_0) \cap C^2_b(\Omega_0)$ such that
\begin{equation}\label{e:DAflux}
  \partial_y g(x,1) = 0 \,,
  \qquad\text{and}\qquad
  \partial_x^2 g(x,0) + \alpha \partial_y g(x,0)
    = \partial_y^2 g(x,0)\,.
\end{equation}
The generator of the process~$Z$ (defined by~\eqref{e:limitdef1}--\eqref{e:limitdef2}) is the operator~$A$ with domain $\mathcal D(A)$.
\end{lemma}

\begin{proof}
  Choose $g \in \mathcal{D}(A)$ and apply It\^o's formula to obtain
\begin{align*}g(X_t,Y_t) = g(X_0,Y_0) &+ \int_0^t \p_xg(X_s,Y_s)dX_s + \int_0^t \p_y g(X_s,Y_s) \, dY_s\\
&+  \frac{1}{\alpha}\int_0^t \p^2_x g(X_s,Y_s) \, dL_s^Y(0) + \frac{1}{2}\int_0^t \p_y^2 g(X_s, Y_s) \, dT_s \,.
\end{align*}
Taking expectations gives
\begin{multline}\label{e:tmpg2}
  \Em^{(x,y)}\brak[\Big]{g(X_t,Y_t) - g(x,y)} = \Em^{(x,y)}\brak[\Big]{\int_0^t\p_yg(X_s,Y_s) \, dY_s}\\
 +\Em^{(x,y)}\brak[\Big]{ \frac{1}{\alpha}\int_0^t\p^2_xg(X_s,Y_s) \, dL_s^Y(0) + \frac{1}{2}\int_0^t \p_y^2 g(X_s,Y_s) \, dT_s}
 \,.
\end{multline}

Now for $y \in (0, 1)$ we know  $Y$ is a Brownian motion before it first hits $0$ or $1$, and hence $\lim_{t\to 0}\Pm^y(L_t^Y(0) \neq 0 ) = 0$.
Moreover by definition of $T$, we know $T_t = t$ when $\{L^Y_t = 0\}$.
Consequently
\begin{equation*}
\lim_{t\to 0}\Em^{(x,y)} \brak[\Big]{\frac{g(X_t,Y_t) - g(x,y)}{t}} = \frac{1}{2}\p_y^2g(x,y) \,.
\end{equation*}

For $y = 1$ we note
\begin{multline}\label{e:tmpg1}
\lim_{t\to 0}\Em^{(x,1)} \brak[\Big]{\frac{g(X_t,Y_t) - g(x,y)}{t}}
\\
  = \frac{1}{2}\p_y^2g(x,1) +  \lim_{t\to 0}\Em^{(x,1)}\brak[\Big]{\frac{1}{t}\int_0^t \p_y g(X_s,Y_s) \, dY_s } \,.
\end{multline}
By~\eqref{e:Tanaka1} we know $\E^{(x,1)} L_t^{Y}(1) = O(\sqrt{t})$, and hence the right hand side of~\eqref{e:tmpg1} is finite if and only if $\partial_y g(x,1) = 0$.

Finally, we compute the generator on the spine $y = 0$.
First we show that if we start $Y$ at $0$ then for a short time it spends ``most'' of the time at 0. More precisely we claim
\begin{equation}
	\lim_{t\to 0}\Em^0\brak[\Big]{\frac{T_t}{t}} = 0\label{e:T_estimate} \, .
\end{equation}
Here we clarify that the~$0$ superscript on~$\E$ refers to the initial distribution of the process~$\bar B$, where as the double superscript~$\E^{(x,y)}$, or measure superscript~$\E^\mu$ used earlier refers to the initial distribution of the joint process~$Z = (X, Y)$.

Let $M_t$ be the running maximum of $\tilde{B}$. Note that since $L^{\bar B} = L^{\tilde{B}}$ on $\{M_t < 1\}$, we have
 \begin{multline*}
	\Pm^0\paren[\Big]{L^{\bar B}_t(0) \leq r} \leq \Pm^0\paren[\Big]{L^{\tilde B}_t(0) \leq r} + \Pm^0\paren[\Big]{M_t > 1}\\
	= 1 - 2 \Pm^0\paren[\Big]{r < \tilde B_t < 1} \leq \sqrt{\frac{2}{\pi}}\paren[\Big]{\frac{r}{\sqrt{t}} +\sqrt{t}e^{-\frac{1}{2t}}} \,.
\end{multline*}
Thus,
\begin{align*}
  \Em^0\brak[\Big]{\frac{T_t}{t}}
    &= \int_0^1 \Pm^0 \paren[\Big]{T_t > st} \, ds
    = \int_0^1 \Pm^0 \paren[\Big]{st + 2L^{\bar B}_{st}(0) \leq t} \, ds
  \\
  &= \int_0^1 \Pm^0 \paren[\Big]{L^{\bar B}_{st}(0) \leq \frac{(1-s)t}{2} } \, ds
    \leq \int_0^1\sqrt{\frac{2}{\pi}}\paren[\Big]{\frac{2(1-s)}{\sqrt{s}}\sqrt{t} + \sqrt{st} \, e^{-\sfrac{1}{2st}}}\, ds
  \\
  &\leq C\sqrt{t} \,.
\end{align*}

With this estimate, we can now compute generator on the spine.
Using equation \eqref{e:T_estimate} we see
\begin{equation}\label{e:tmpLtYbyt}
	\Em^{0}\left[\frac{L^Y_t(0)}{t}\right] = \Em^{0}\left[\frac{L^{\bar B}_{T_t}(0)}{t}\right] = \frac{\alpha}{2}\Em^{0}\left[\frac{t - T_t}{t}\right]\xrightarrow{t\rightarrow 0}\frac{\alpha}{2} \,.
\end{equation} 
Using \eqref{e:Tanaka1} we have,
\begin{equation*}
	\Em^{0}\left[\frac{Y_t}{t}\right] = \Em^{0}\left[\frac{\bar B_{T_t}}{t}\right] = \Em^{0}\left[\frac{\tilde{B}_{T_t} + L^{\bar B}_{T_t}(0) - L^{\bar B}_{T_t}(1) }{t}\right] \,.
\end{equation*}
Since $T_t \leq t$, the third term tends to 0 and using the modulus of continuity for Brownian motion the first term does as well.
Therefore we also have 
\begin{equation}\label{e:tmpYtByt}
	\Em^{0}\left[\frac{Y_t}{t}\right] \xrightarrow{t\rightarrow 0} \frac{\alpha}{2} \, .
\end{equation}
Thus using~\eqref{e:T_estimate}, \eqref{e:tmpLtYbyt} and~\eqref{e:tmpYtByt} in equation~\eqref{e:tmpg2} gives
\begin{equation*}
  \lim_{t\to 0} \frac{1}{t} \Em^{(x,y)}\brak[\Big]{g(X_t,Y_t) - g(x,y)}
    = \frac{\alpha}{2} \partial_y g(x, 0) + \frac{1}{2} \partial_x^2g(x, 0) + 0
 \,,
\end{equation*}
finishing the proof.
\end{proof}

\subsection{PDE Homogenization (Corollaries~\ref{c:var}, \ref{c:pde}, and Proposition~\ref{p:ftime}).}

Once the generator of~$Z$ is known, the behavior of the variance (Corollary~\ref{c:var}) and PDE homogenization result (Corollary~\ref{c:pde}) can be deduced quickly.

\begin{proof}[Proof of Corollary~\ref{c:var}]
  We first assume~$h_0 = 1$ as in Remark~\ref{r:h0eq1}.
  Using Theorem~\ref{t:zlimfat} and~\eqref{e:qvX} we see
  \begin{equation}\label{e:varX}
    \lim_{\epsilon \to 0} 
      \E^{(x,0)} \abs{X^\epsilon_t - x}^2
    = \E^{(x,0)} \abs{X_t - x}^2
      = \frac{2}{\alpha} \E^{0} L^Y_t(0)\,.
  \end{equation}
  Now equation~\eqref{e:varShort} follows from~\eqref{e:tmpLtYbyt}.

  For the long time limit (when~$h_0 = 1$) we note that by ergodicity of~$\bar B$, we know that $\E^0 \abs{ L_t^{\bar B} / t - 1/2} \to 0$ as $t \to \infty$.
  Thus using~\eqref{e:Tdef} we must have
  \begin{equation*}
    \lim_{t \to \infty} \E^0 \abs[\Big]{\frac{T(t)}{t} - \frac{\alpha}{\alpha + 1} } = 0 \,.
  \end{equation*}
  Consequently,
  \begin{equation*}
    \Em^{0}\paren[\Big]{\frac{L^Y_t(0)}{t}}
      = \Em^{0}\paren[\Big]{\frac{L^{\bar B}_{T_t}}{t}}
      = \frac{\alpha}{2}\Em^{0}\paren[\Big]{\frac{t - T_t}{t}}
      \xrightarrow{t\rightarrow \infty}\frac{\alpha}{2(\alpha + 1)} \,,
  \end{equation*}
  and together with~\eqref{e:varX} this implies~\eqref{e:varLong}.
  This finishes the proof of~\eqref{e:varShort} and~\eqref{e:varLong} when $h_0 = 1$.
  The case for arbitrary finite $h_0$ is similar.

  When $h_0 = \infty$, the process $Y$ is a sticky Brownian motion on the half line, and the distribution of $L^Y_t(0)$ can be computed explicitly.
  Namely (see for instance~\cite{Howitt07}) we have
  \begin{align}
      \label{e:sbmOT}
      \frac{2}{\alpha} L^Y_t(0)
      = \int_0^t \one_{\set{Y_s = 0}} \, ds
      &\sim \frac{2\abs{N}}{\alpha} \paren[\Big]{ t + \frac{N^2}{\alpha^2}}^{1/2} - \frac{2 N^2}{\alpha^2}\,,
  \end{align}
  where~$N$ is the standard normal.
  Taking expectations and using~\eqref{e:varX} immediately yields~\eqref{e:varShort} and~\eqref{e:varLongInf}, finishing the proof.
\end{proof}

\begin{proof}[Proof of Corollary~\ref{c:pde}]
  By the Kolmogorov backward equation~\cite[\S5.6]{Friedman75} we known that the function~$u^\epsilon$ (defined by~\eqref{e:heat1}--\eqref{e:heat2}) satisfies
  \begin{equation*}
    u^\epsilon(z, t) = \E^z u_0(Z^\epsilon_t)\,.
  \end{equation*}
  Consequently
  \begin{equation*}
    \int_{\Omega_\epsilon} u^\epsilon(z, t) \, d\mu^\epsilon(z)
      = \E^{\mu^\epsilon} u_0( Z^\epsilon_t )
      \xrightarrow{\epsilon \to 0} 
      = \E^{\mu} u_0( Z_t )\,,
  \end{equation*}
  by Theorem~\ref{t:zlimfat}.
  Thus, if we set
  \begin{equation}\label{e:udef}
    u(z, t) = \E^z u_0(Z_t)\,,
  \end{equation}
  we see that~\eqref{e:uepConv} holds.

  It only remains to verify that~$u$ satisfies~\eqref{e:rho1}--\eqref{e:rho4} hold.
  To see this, recall that the function~$u$ defined by~\eqref{e:udef} belongs to $C(0, \infty; \mathcal D(A))$ and satisfies the Kolmogorov equations
  \begin{alignat*}{2}
    \span
      \partial_t u - A u = 0
      &\qquad& t > 0\,,
    \\
    \span
      u(\cdot, t) = u_0
      && \text{when } t = 0\,.
  \end{alignat*}
  The first equation above implies~\eqref{e:rho1} by definition of~$A$ (equation~\eqref{Adef}).
  Equations~\eqref{e:rho2} and~\eqref{e:rho3} follow from the fact that $u(\cdot, t) \in \mathcal D(A)$ for all $t > 0$, and equation~\eqref{e:rho4} follows from the second equation above.
\end{proof}

We now obtain evolution equations for the slice of $u$ at $y = 0$, as stated in Proposition~\ref{p:ftime}.
\begin{proof}[Proof of Proposition~\ref{p:ftime}]
  Let $u_1 = u - g$, and observe that $u_1$ satisfies~\eqref{e:rho1} with initial data $u_1(x, y, 0) = u_0(x, 0) = v_0(x)$, and boundary conditions
  \begin{equation}\label{e:u1bc}
    u_1(x, 0, t) = u(x, 0, t) = v(x, t)
    \qquad\text{and}\qquad
    \partial_y u_1(x, 1, t) = 0\,.
  \end{equation}
  (Recall that in Remark~\ref{r:h0eq1} we have already set $h_0 = 1$ for simplicity.)
  We now treat~$x$ as a parameter, and solve~\eqref{e:rho1} using separation of variables (in $y$, $t$) with boundary conditions~\eqref{e:u1bc}.
  A direct calculation shows
  \begin{equation}\label{e:pyu1}
    \partial_y u_1(x, 0, t) = - \partial_t^w v\,,
  \end{equation}
  and hence
  \begin{equation}\label{e:pyu}
    \partial_y u(x, 0, t) = -\partial_t^w v(x, t) + \partial_y g(x, 0, t)\,.
  \end{equation}

  Now for $t > 0$ using equation~\eqref{e:rho1} and~\eqref{e:rho2} and continuity of second derivatives of $u$ up to $y = 0$ we see
  \begin{equation}\label{e:ev2}
    \partial_t v(x, t)
      = \frac{\alpha}{2} \partial_y u(x, 0, t) + \frac{1}{2} \partial_x^2 v(x, t)\,.
  \end{equation}
  Using~\eqref{e:pyu} and~\eqref{e:ev2} yields~\eqref{e:ev} as claimed.
\end{proof}

\begin{remark}
  For brevity, we have suppressed the explicit separation of variables calculation deriving~\eqref{e:pyu1}.
  One can avoid this calculation by using the Laplace transform as follows.
  Following standard convention, we will denote the Laplace transform of a function using an upper case letter using the variable $s$, instead of $t$.
  Explicitly, given a function $f$, we define its Laplace transform, denoted by $F$ or $\mathcal Lf$, by
  \begin{equation*}
    F(s) \defeq \mathcal Lf(s) = \int_0^\infty e^{-s t} f(t) \, dt\,.
  \end{equation*}
  For functions that depend on both space and time variables, the Laplace transform will only be with respect to the time variable.

  Taking the Laplace transform of~$u_1$ yields the ODE in the variable $y$
  \begin{equation*}
    s U_1 - v_0 - \frac{1}{2} \partial_y^2 U_1 = 0\,,
  \end{equation*}
  with boundary conditions $U_1(x, 0, s) = V(x, s)$, and $\partial_y U_1(x, 1, s) = 0$.
  Solving this ODE yields
  \begin{equation*}
    U_1(x, y, s)
      = \frac{v_0}{s} +
	\paren[\Big]{\frac{1}{1 + e^{2 \sqrt{2s}} }}
	\paren[\Big]{V - \frac{v_0}{s}}
	\brak[\Big]{ e^{y \sqrt{2s} } + e^{\sqrt{2s}(2 - y)} }\,,
  \end{equation*}
  and hence
  \begin{equation*}
    \partial_y U_1 (x, 0, s )
      = -\sqrt{2s}
        \paren[\Big]{V - \frac{v_0}{s}}
        \tanh \sqrt{2s}
      = -\frac{2 \tanh \sqrt{2s} }{\sqrt{2s}} \paren[\Big]{ s V - v_0} \,.
  \end{equation*}
  Choosing~$w$ to be a function with Laplace transform~\eqref{e:LW}, implies~\eqref{e:pyu1} as claimed.
\end{remark}

\section{Comb-Shaped Domains (Theorem \ref{t:zlimfat}).}\label{s:fatcomb}

We now turn to the proof of Theorem~\ref{t:zlimfat}. Recall that $Z_t^{\epsilon,+} = \pi_\epsilon(Z_t^\epsilon) = (X_t^\epsilon, \max (Y_t^\epsilon,0))$. The main ingredients in the proof are the following lemmas.

\begin{lemma}\label{l:FCtightness}
  Let $Z^\epsilon = (X^\epsilon, Y^\epsilon)$ be the reflected Brownian motion on the comb-shaped domain~$\Omega_\epsilon$, as described in Theorem~\ref{t:zlimfat}.
  Then, for any $T > 0$, the family of processes $Z^\epsilon$ is tight in $C([0,T]; \R^2)$. 
\end{lemma}
\begin{lemma}\label{l:MPuniqness}
  Let $A$ be the generator defined in~\eqref{Adef}, with domain $\mathcal D(A)$. Weak uniqueness holds for the martingale problem for $A$.
\end{lemma}
\begin{lemma}\label{l:FCgenerator}
If $f \in \mathcal D(A)$, and $K \subset \Omega_0$ is compact, then
  \begin{equation}\label{e:fconv}
\lim_{\epsilon \to 0} \sup_{z \in K \cap \Omega_\epsilon}    \E^z\paren[\Big]{
      f(Z^{\epsilon,+}_t)
      - f(Z^{\epsilon,+}_0)
      - \int_0^t Af(Z^{\epsilon,+}_s) \, ds
    } = 0\,.
  \end{equation}%
\end{lemma}

 Momentarily postponing the proof of these lemmas, we prove Theorem~\ref{t:zlimfat}.
\begin{proof}[Proof of Theorem~\ref{t:zlimfat}]
  Suppose first $Z^{\epsilon,+} \to Z'$ weakly along some subsequence.
  We claim $Z'$ should be a solution of the martingale problem for $A$ with initial distribution~$\mu$.
  To see this set
  \begin{equation*}
    M^\epsilon_t = f( Z^{\epsilon,+}_t) - f( Z^{\epsilon,+}_0 )
      - \int_0^t A f( Z^{\epsilon,+}_r ) \, dr
  \end{equation*}
  and observe
  \begin{equation*}
    \E^{\mu^\epsilon} \paren[\big]{
      M^\epsilon_t
      \given \mathcal F_s
    }
    = M^\epsilon_s
      + \E^{Z^\epsilon_s} \paren{M^\epsilon_{t-s}}\,,
  \end{equation*}
  by the Markov property.
  Using Lemma~\ref{l:FCgenerator}, and taking limits along this subsequence, the last term on the right vanishes.
  Since this holds for all $f \in \mathcal D(A)$ and $\mathcal D(A)$ is dense in $C_0(\Omega_0)$, $Z'$ must be a solution of the martingale problem for~$A$.
  Since $Z^{\epsilon,+} \to Z'$ weakly and $\pi_\epsilon^*(\mu^\epsilon) \to \mu$ weakly by assumption, we have $Z(0) \sim \mu$.
  By uniqueness of solutions to the martingale problem for~$A$ (Lemma~\ref{l:MPuniqness}), the above argument shows uniqueness of subsequential limits of~$Z^{\epsilon,+}$.
  Combined with tightness (Lemma~\ref{l:FCtightness}), and the fact that $Z$ is a solution to the martingale problem for~$A$ (Lemma~\ref{l:Zgen}), this gives weak convergence as desired.
\end{proof}

It remains to prove Lemmas~\ref{l:FCtightness}--\ref{l:FCgenerator}.
We do this in Sections~\ref{s:FCtightness}, \ref{s:MPuniqness} and~\ref{s:FCgenerator}, below.

\subsection{Proof of Tightness (Lemma~\ref{l:FCtightness}).}\label{s:FCtightness}

To prove tightness, we need an auxiliary lemma comparing the oscillation of trajectories in the spine to that of Brownian motion.
This will also be used in the proof of Lemma~\ref{l:FCgenerator}.
\begin{lemma}\label{lem:BrownianComp}
Let $W'$ be a standard Brownian motion on $\R$ with $W'(0) = 0$.  For any $T > 0$, $\epsilon \in (0,1/2]$, $z \in \Omega_\epsilon$, and any $a, \delta > 0$, we have
\begin{equation} \label{XWcomp}
\P^z \paren[\Big]{ \sup_{ \substack{r,t \in [0,T] \\ |t - r| \leq \delta}} |X^\epsilon(t) - X^\epsilon(r)|  \geq a } \leq \P \paren[\Big]{ \sup_{ \substack{r,t \in [0,T] \\ |t - r| \leq \delta}}4 |W'(t) - W'(r)|  \geq a - 2\epsilon }\,.
\end{equation}
\end{lemma}
\begin{proof}
 Let
\begin{equation*}
\tau_0 = \inf \set[\big]{ t \geq 0 \st X^\epsilon(t) \in \epsilon \paren[\big]{\Z + \frac{1}{2}}}\,,
\end{equation*}
and inductively define
\begin{equation*}
\tau_{k+1} = \inf \set[\big]{ t \geq \tau_k \st \abs{X^\epsilon(t) - X^\epsilon(\tau_k)} =  \epsilon }\,,
\end{equation*}
for $k \geq 0$.
By symmetry of the domain, observe that $k \mapsto X^\epsilon(\tau_k)$ defines a simple random walk on the discrete points $\epsilon(\Z + \sfrac{1}{2})$.   Next, define
\[
\tau_k' = \inf \set { t \geq \tau_k \st |X^\epsilon(t) - X^\epsilon(\tau_k)| =  \epsilon/4 }, \quad k \geq 0.
\]
In particular, $\tau_k < \tau_k' < \tau_{k+1}$. At time $\tau_k$, $X^\epsilon(\tau_k)$ is in the spine, at the midpoint between two adjacent teeth. For $t \in [\tau_k, \tau_k']$, $X^\epsilon(t)$ is in the spine and cannot enter the teeth, because $|X^\epsilon(t) - x| \leq \epsilon/4$ where $x = X^\epsilon(\tau_k) \in \epsilon ( \Z + \frac{1}{2})$.
Define the increments $\Delta_k X^\epsilon = X^\epsilon(\tau_{k+1}) - X^\epsilon(\tau_k) \in \{ - \epsilon, + \epsilon \}$.
By the strong Markov property and symmetry of the domain, the random variables $\{ (\tau_k' - \tau_k) \}_k \cup \{  \Delta X^\epsilon_k \}_k$ are independent. 

Now, suppose that $W'(t)$ is an independent Brownian motion on $\R$, with $W'(0) = 0$. Define another set of stopping times inductively by $\sigma_0 = 0$ and
\begin{align*}
\sigma_{k+1} & = \inf \set { t \geq \sigma_k \st |W'(t) - W'(\sigma_k)| =  \epsilon/4 }, \quad k \geq 0.
\end{align*}
Let $\Delta \sigma_k = \sigma_{k+1} - \sigma_k$, and $\Delta_k W' = W'(\sigma_{k+1}) - W'(\sigma_k) \in \{ -\epsilon/4, \epsilon/4\}$.  Observe that the family of random variables
\[
\{(\sigma_{k+1} - \sigma_k) ,  4 \Delta W'_k \}_{k \geq 0}
\]
has the same law as the family 
\[
\{ (\tau_k' - \tau_k) ,  \Delta X^\epsilon_k \}_{ k \geq 0}.
\]

Next, define
\[
K(t) = \max \{ k \geq 0 \;|\; \tau_k  \leq t\},
\]
and observe that if $|t - r| \leq \delta$ and $0 \leq r \leq t \leq T$, then we must have $\tau_{K(t)} - \tau_{K(r) + 1} \leq \delta$ and thus
\[
\sum_{j = K(r) + 1}^{K(t) - 1} (\tau_j' - \tau_j) \leq \delta, \quad \quad \text{and} \quad \quad  \sum_{j = 0}^{K(t) - 1} (\tau_j' - \tau_j) \leq T.
\]
In this case,
\begin{align*}
  \MoveEqLeft[4] \abs{X^\epsilon(t) - X^\epsilon(r)}  \leq 2 \epsilon + \abs{X^\epsilon(K(t)) - X^\epsilon(K(r)+ 1)} \nonumber \\
& = 2 \epsilon + \abs[\Big]{ \sum_{j= K(r) + 1}^{K(t) - 1} \Delta X^\epsilon_j}  \\
& \leq 2 \epsilon + \sup_{0 \leq \ell \leq m} \abs[\Big]{ \sum_{j= \ell+1}^{m- 1} \Delta X^\epsilon_j} \one_{\set[\big]{  \sum_{j = \ell+1}^{m-1} (\tau_j' - \tau_j) \leq \delta }} \one_{\set[\big]{ \sum_{j = 0}^{m-1} (\tau_j' - \tau_j) \leq T }} \,.
\end{align*}
This last supremum has the same law as
\begin{multline*}
\sup_{0 \leq \ell \leq m} \abs[\Big]{ \sum_{j= \ell+1}^{m-1} 4 \Delta W'_j}
  \one_{\set[\big]{  \sum_{j = \ell+1}^{m-1} (\sigma_{j+1} - \sigma_j) \leq \delta }} \one_{\set[\big]{\sum_{j = 0}^{m-1} (\sigma_{j+1} - \sigma_j) \leq T }}
  \\
  = \sup_{0 \leq \ell \leq m} 4 \abs{ W'(\sigma_m) - W'(\sigma_{\ell+1})} \, \one_{\set{ \sigma_{m} - \sigma_{\ell+1} \leq \delta }} \, \one_{ \set{ \sigma_m - \sigma_0  \leq T }} \,.
\end{multline*}
Since the right hand side of the above is bounded by
\begin{equation*}
  \sup_{ \substack{r,t \in [0,T] \\ \abs{t - r} \leq \delta}}4 \abs{W'(t) - W'(r)} \,,
\end{equation*}
we obtain \eqref{XWcomp}.
\end{proof}

We now prove Lemma~\ref{l:FCtightness}.
\begin{proof}[Proof of Lemma~\ref{l:FCtightness}]
  Note first that Lemma~\ref{lem:BrownianComp} immediately implies that the processes~$X^\epsilon$ are tight.
  Indeed, by~\eqref{XWcomp} we see
  \begin{equation}
    \lim_{\delta \to 0} \limsup_{\epsilon \to 0} \P^{\mu^\epsilon} \paren[\Big]{ \sup_{ \substack{r,t \in [0,T] \\ |t - r| \leq \delta}} |X^\epsilon(t) - X^\epsilon(r)|  \geq a }  = 0 \label{tightX1}\,.
  \end{equation}
  Moreover, since $\mu^\epsilon$ converge weakly to the probability measure~$\mu$, the distributions of $X^\epsilon_0$ are tight.
  This implies implies tightness of the processes $X^\epsilon$.

  For tightness of~$Y^\epsilon$, we note as above that the distributions of $Y^\epsilon_0$ are already tight.
  In order to control the time oscillations, fix $T > 0$, and let
  \begin{equation*}
    d Z^\epsilon = d B_t + d L^{\partial \Omega_\epsilon}_t\,,
  \end{equation*}
  be the semi-martingale decomposition of~$Z^\epsilon$ (see for instance~\cite{StroockVaradhan71}).
  Here $B = (B_1, B_2)$ is a standard Brownian motion and $L^{\partial \Omega_\epsilon}$ is the local time of $Z^\epsilon$ on $\partial \Omega_\epsilon$.
  Let $\omega(\delta) = \omega_T(\delta)$, defined by 
  \begin{equation*}
    \omega(\delta) = \sup_{ \substack{s,t \in [0,T] \\ |t - s| \leq \delta}} |B_2(t) - B_2(s)| \,,
  \end{equation*}
  be the modulus of continuity for $B_2$ over $[0, T]$.
  Let $[s,t] \subset [0,T]$ with $|t - s| \leq \delta$.
  If  $0 < Y^\epsilon_r < 1$ for all $r \in (s,t)$, then we must have
  \[
    |Y^\epsilon(t) - Y^\epsilon(s)| = |B_2(t) - B_2(s)| \leq \omega(\delta) \,.
  \]
  Otherwise, for some $r \in (s,t)$ either $Y_r = 0$ or $Y_r = 1$.
  Let $G_\delta$ be the event that $\omega(\delta) < 1/2$; on this event $Y$ cannot hit both $0$ and $1$ on the interval $[s,t]$.
  Define
\begin{equation*}
  \eta_- = \inf \set{ r > s \st Y^\epsilon_r \in \{0,1\} }\,,
  \quad\text{and}\quad
  \eta_+ = \sup \set{ r < t \st Y^\epsilon_r \in \{0,1\} }\,. 
\end{equation*}
In this case we have
\begin{align*}
|Y^\epsilon_t - Y^\epsilon_s| & \leq \max \paren{ |Y^\epsilon(\eta_-) - Y^\epsilon(s)|\;, \; |Y^\epsilon(t) - Y^\epsilon(\eta_+)| } + \one_{G_\delta^c}  + \epsilon^2 \\
& = \max \paren{ |B({\eta_-}) - B(s)|\;, \; |B(t) - B(\eta_+)| } + \one_{G_\delta^c} + \epsilon^2 \leq \omega(\delta) +  \one_{G_\delta^c} + \epsilon^2.
\end{align*}
Combining the two cases, we see that for any $z \in \Omega_\epsilon$,
\begin{equation*}
\P^z \paren[\Big]{\sup_{ \substack{s,t \in [0,T] \\ |t - s| \leq \delta}} |Y^\epsilon(t) - Y^\epsilon(s)| > a } \leq \P(\omega(\delta) > a-\epsilon^2) + \P(G^c_\delta).
\end{equation*}
Since the right hand side is independent of $z$, integrating over $z$ with respect to $\mu^\epsilon$ implies
\begin{equation*}
\lim_{\delta \to 0} \limsup_{\epsilon \to 0} \P^{\mu^\epsilon} \paren[\Big]{\sup_{ \substack{s,t \in [0,T] \\ |t - s| \leq \delta}} |Y^\epsilon(t) - Y^\epsilon(s)| > a } = 0
\end{equation*}
holds for any $a > 0$.
This shows tightness of $Y^\epsilon$ in $C([0,T])$, finishing the proof of Lemma~\ref{l:FCtightness}.
\end{proof}
\subsection{Uniqueness for the Martingale Problem (Lemma~\ref{l:MPuniqness})}\label{s:MPuniqness}
The proof of Lemma~\ref{l:MPuniqness} relies on the existence of regular solutions to the corresponding parabolic equation.
We state this result next.
\begin{lemma}\label{l:PDEexistence}
  For all $f \in \mathcal D(A)$, there exists a solution to
  \begin{equation}\label{e:dtuEqAu}
    \partial_t u - A u = 0\,,\qquad
    u(\cdot, 0) = f\,,\qquad
    \text{with }
    u(\cdot, t) \in \mathcal D(A)\,.
  \end{equation}
\end{lemma}
Given Lemma~\ref{l:PDEexistence}, the proof of Lemma~\ref{l:MPuniqness} is standard (see for instance~\cite{RogersWilliams00a,EthierKurtz86}).
For the readers convenience, we describe it briefly here.
\begin{proof}[Proof of Lemma~\ref{l:MPuniqness}]
  Suppose $Z, Z'$ are two processes satisfying the martingale problem for $A$.
  Let $f \in \mathcal D(A)$ be any test function, and $u$ be the solution in $\mathcal D(A)$ of $\partial_t u - Au = 0$ with initial data $f$.
  Then for any $z \in \Omega_0$, and fixed $T > 0$, the processes $u(Z_t, T-t)$ and $u(Z'_t, T-t)$ are both martingales under the measure $P^z$.
  Hence
  \begin{align*}
    \E^\mu f(Z_T)
      &= \int_{\Omega_0}\E^z f(Z_T) \, \mu(dz)
      = \int_{\Omega_0}\E^z u(Z_t, T-t) \, \mu(dz)
      = \int_{\Omega_0} u(z, T) \, \mu(dz)
      \\
      &= \int_{\Omega_0}\E^z u(Z'_t, T-t) \, \mu(dz)
      = \int_{\Omega_0}\E^z f(Z'_T) \, \mu(dz)
      = \E^\mu f(Z'_T)\,.
  \end{align*}
  Since $\mathcal D(A)$ is dense in $C_0(\Omega_0)$ this implies $Z$ and $Z'$ have the same one dimensional distributions.
  By the Markov property, this in turn implies that the laws of $Z$ and $Z'$ are the same.
\end{proof}

It remains to prove Lemma~\ref{l:PDEexistence}.
\begin{proof}[Proof of Lemma~\ref{l:PDEexistence}]
  Let $v(x, t) = u(x, 0, t)$.
  Since~\eqref{e:dtuEqAu} is equivalent to~\eqref{e:rho1}--\eqref{e:rho3}, Proposition~\ref{p:ftime}%
  \footnote{
    We remark that the proof of Proposition~\ref{p:ftime} is self contained, and does not rely on Theorem~\ref{t:zlimfat}.
    Thus its use here is valid and does not lead to circular logic loop.%
  }
  implies that~$v$ satisfies the Basset type equation~\eqref{e:ev}.
  For the homogeneous equation associated with~\eqref{e:ev}, existence and uniqueness is proved in~\cite{Chen17}.
  The inhomogeneous equation can be solved using an analog of Duhamel's principle~\cites{Umarov12,UmarovSaydamatov06}.
  Explicitly, for $s \geq 0$, let $\tilde v_s$ be a solution to the equation
  \begin{subequations}
  \begin{align}
    \label{e:tvsEvol}
    \span
    \partial_t \tilde v_s(x,t) + \frac{\alpha}{2} \partial_t^w \tilde v_s(x,t)
      - \frac{1}{2}\partial_x^2 \tilde v_s(x,t) = 0 \,,
      && \text{for } t > s\,,
    \\
    \label{e:tvsId}
    \span
      \tilde v_s(x,s) = \paren[\Big]{I + \frac{\alpha}{2} \mathcal I^{w}_s}^{-1} \frac{\alpha f(x,\cdot)}{2} \,.
  \end{align}
  \end{subequations}
  Here $\mathcal I^w_\cdot$ is the integral operator with kernel $w$ defined by
  \begin{equation*}
    \mathcal I^w_t h = \int_0^t w( t - s ) h(s) \, ds\,,
  \end{equation*}
  for any function~$h \colon (0, \infty) \to \R$.
  Since $\mathcal I^w$ is a compact operator, the operator $(I + \mathcal (\alpha/2)I^w )$ is invertible, ensuring the initial condition~\eqref{e:tvsId} can be satisfied.
  For convenience, define $\tilde v_s( x, r ) = \tilde v_s(x, s )$ when $r < s$.
  Now, one can directly check that the function $v$ defined by
  \begin{equation*}
    v(x, t) \defeq \int_0^t \tilde v_s( x, t ) \, ds \,,
  \end{equation*}
  is a strong solution to the inhomogeneous equation~\eqref{e:ev}.

  Since $u$ satisfies the heat equation for $y \in (0, 1)$ we can write $u$ in terms of $v$ and $f$ using the heat kernel.
  Explicitly, we have
  \begin{equation*}
    u(x, y, t)
      = \frac{\alpha}{2} \int_0^1 K_t''(y, z) f(z) \, dz
	+ \kappa \int_0^t \partial_z K_{t-s}''(y, 0) v(x, s) \, ds\,,
  \end{equation*}
  where $K''$ is the heat kernel on $(0, 1)$ with Dirichlet boundary conditions at $y = 0$ and Neumann boundary conditions at $y = 1$.
  Since $v$ is $C^{2,1}$ this immediately implies $u \in C^{2, 1}$.
  Thus to show $u(\cdot, t) \in \mathcal D(A)$ we only need to verify the flux condition~\eqref{e:DAflux}.
  This, however, follows immediately from the fact that $\partial_y^2 u(x, 0, t)  = 2 \partial_t u(x, 0, t) = 2 \partial_t v(x, t)$ and equation~\eqref{e:ev2}.
\end{proof}

\subsection{Generator Estimate (Lemma~\ref{l:FCgenerator}).}\label{s:FCgenerator}


The main idea behind the proof of Lemma~\ref{l:FCgenerator} is to balance the local time $Z^\epsilon$ spends at the ``gate'' between the spine and teeth, and the time spent in the spine.
Explicitly, let $S \defeq \R \times (-\epsilon,0)$ denote the spine of~$\Omega_\epsilon$, and $T$, defined by
\[
T \defeq \bigcup_{k \in \epsilon \Z} \set[\big]{ (x,y) \st |x - \epsilon k| < \frac{\alpha\epsilon^2}{2}, \ y \in (0,1) }\,,
\]
denote the collection of the teeth (see~\eqref{e:OmegaEp} and Figure~\ref{f:fatcomb}).
Let the ``gate'' $G$, defined by 
\[
G \defeq \partial T \cap \partial S = \bigcup_{k \in \epsilon \Z} \set[\big]{ (x,0) \st |x - \epsilon k| \leq \frac{\alpha\epsilon^2}{2} } \,,
\]
denote the union of short segments connecting the spine and teeth.  Let $L^G_t$ denote the local time of $Z^\epsilon_t$ at the set $G$. Now the required local time balance can be stated as follows.

\begin{lemma}\label{l:FCLocalTimeG}
  For every $g \in C^1_b(\R)$ and $K \subseteq \Omega_0$ compact we have 
  \begin{equation}\label{claim2G}
    \lim_{\epsilon \to 0} \sup_{z \in K \cap \Omega_\epsilon}  \E^z\paren[\Big]{ \alpha\int_0^t  g(X_s^\epsilon)  \one_{\{Y_s^\epsilon < 0\}} \,ds  - 2 \int_0^t g(X_s^\epsilon) dL^{G}_s} = 0 \,.
  \end{equation}
\end{lemma}

Next, we will also need to show that the local times on the left edges and right edges of the teeth balance.
Explicitly, let $\partial T^-$, $\partial T^+$ defined by
\begin{gather*}
  \partial T^- \defeq \set[\big]{ (x, y) \in Z^\epsilon \st x \in \epsilon \Z - \frac{\alpha\epsilon^2}{2},\ y > 0 }\,,
  \\
  \llap{\text{and}\qquad} \partial T^+ \defeq \set[\big]{ (x, y) \in Z^\epsilon \st x \in \epsilon \Z + \frac{\alpha\epsilon^2}{2},\ y > 0 }\,.
\end{gather*}
denote the left and right edges of the teeth respectively.
Let $L^+$ and $L^-$ be the local times of $Z^\epsilon$ about $\partial T^-$ and $\partial T^+$ respectively, and let $L^{\pm}$ denote the difference
\begin{equation*}
  L^\pm = L^- - L^+\,.
\end{equation*}
The balance on the teeth boundaries we require is as follows.
\begin{lemma}\label{l:FCLocalTimeT}
  For every $f \in \mathcal D(A)$ and $K \subseteq \Omega_0$ compact, we have 
  \begin{equation}\label{claim1x}
    \lim_{\epsilon \to 0} \sup_{z \in K \cap \Omega_\epsilon}   \E^z \paren[\Big]{ \int_0^t \frac{1}{2} \partial_x^2 f(Z^{\epsilon,+}_s) \one_{\{Y_s^\epsilon > 0\}}\,ds +  \int_0^t \partial_x f(Z^{\epsilon,+}_s) \, d L^{\pm}_s  } = 0\,.
  \end{equation}
\end{lemma}

Momentarily postponing the proofs of Lemmas~\ref{l:FCLocalTimeG} and~\ref{l:FCLocalTimeT}, we prove Lemma~\ref{l:FCgenerator}.
\begin{proof}[Proof of Lemma~\ref{l:FCgenerator}]
  Given $f \in \mathcal D(A)$, we define $f^\epsilon \colon \Omega_\epsilon \to \R$ by 
\[
f^\epsilon(x,y) \defeq f(x, y^+)\,.
\]
Thus, $f(Z^{\epsilon,+}_t) =  f^\epsilon(Z^\epsilon_t)$, and~\eqref{e:fconv} reduces to showing
  \begin{equation*}
  \lim_{\epsilon \to 0} \sup_{z \in K \cap \Omega_\epsilon}    \E^z\paren[\Big]{
      f^\epsilon(Z^\epsilon_t)
      - f^\epsilon(Z^\epsilon_0)
      - \int_0^t \frac{1}{2} \partial_y^2 f^\epsilon (Z_s^\epsilon) \, ds
    } = 0 \,.
  \end{equation*}

Since $f \in \mathcal{D}(A)$, we have $\partial_x^2 f(x,0) + \alpha\partial_y f(x,0) = \partial_y^2 f(x,0)$ and $\partial_y f(x,1) = 0$.  Therefore, the extension $f^\epsilon$ satisfies $\partial_x^2 f^\epsilon(x,y) = \partial_y^2 f^\epsilon(x,0^+) - \alpha\partial_y f^\epsilon(x,0^+)$ for $(x,y) \in S$, as well as $\partial_y f^\epsilon = 0$ for $(x,y) \in S$. Notice that $\partial_{y} f^\epsilon$ may be discontinuous across $G$. Using these facts and It\^o's formula, we compute
\begin{align*}
 \E^z \paren[\Big]{ f^\epsilon(Z^\epsilon_t) - f^\epsilon(Z^\epsilon_s)} & =  \E^z \paren[\Big]{  \int_0^t \frac{1}{2} \left( \partial_y^2 f(Z^{\epsilon,+}_s) + \partial_x^2 f(Z^{\epsilon,+}_s) \right) \one_{\{Y_s^\epsilon > 0\}}  \, ds} \\
& \quad +  \E^z \paren[\Big]{  \int_0^t \frac{1}{2} \partial_x^2 f(X_s^\epsilon,0^+)  \one_{\{Y_s^\epsilon < 0\}} \, ds} \\
& \quad  +  \E^z \paren[\Big]{  \int_0^t \partial_y f(X_s^\epsilon,0^+) \, dL^{G}_s + \int_0^t \partial_x f(Z^{\epsilon,+}_s) \, d L^{\pm}_s } \\
&  =  \E^z \paren[\Big]{  \int_0^t \frac{1}{2} \left( \partial_y^2 f(Z^{\epsilon,+}_s)   + \partial_x^2 f(Z^{\epsilon,+}_s)\right) \one_{\{Y_s^\epsilon > 0\}}  \, ds }  \\
& \quad +  \E^z \paren[\Big]{ \frac{1}{2}  \int_0^t \left( \partial_y^2 f(X_s^\epsilon,0^+) - \alpha\partial_y f(X_s^\epsilon,0^+) \right) \one_{\{Y_s^\epsilon < 0\}} \, ds} \\
& \quad +   \E^z\paren[\Big]{  \int_0^t \partial_y f(X_s^\epsilon,0^+) \, dL^{G}_s + \int_0^t \partial_x f(Z^{\epsilon,+}_s) \, d L^{\pm}_s }\,,
\end{align*}
and hence
  \begin{align*}
& \E^z \paren[\Big]{ f^\epsilon(Z^\epsilon_t) - f^\epsilon(Z^\epsilon_0) - \int_0^t \frac{1}{2} \partial_y^2 f^\epsilon (Z_s^\epsilon) \, ds}\\
 & \quad \quad \quad \quad = \E^z \paren[\Big]{ \int_0^t \frac{1}{2} \partial_x^2 f(Z^{\epsilon,+}_s) \one_{\{Y_s^\epsilon > 0\}}\,ds +  \int_0^t \partial_x f(Z^{\epsilon,+}_s) \, d L^{\pm}_s  }\\
 & \quad \quad \quad \quad  \quad -  \frac{1}{2} \E^z \paren[\Big]{ \int_0^t  \alpha\partial_y f(X_s^\epsilon,0^+)  \one_{\{Y_s^\epsilon < 0\}} \,ds  - 2 \int_0^t \partial_y f(X_s^\epsilon,0^+) dL^{G}_s}\,.
  \end{align*}
  Using Lemmas~\ref{l:FCLocalTimeG} and~\ref{l:FCLocalTimeT} we see that the supremum over $z \in \Omega_\epsilon \cap K$ of the right hand side of the above vanishes as $\epsilon \to 0$.
  This proves Lemma~\ref{l:FCgenerator}.
\end{proof}

It remains to prove Lemmas~\ref{l:FCLocalTimeG} and~\ref{l:FCLocalTimeT}, and we do this in Sections~\ref{s:FCLocalTimeG} and~\ref{s:FCLocalTimeT} respectively.

\subsection{Local Time at the Gate (Lemma~\ref{l:FCLocalTimeG}).}\label{s:FCLocalTimeG}

The crux in the proof of Lemma~\ref{l:FCLocalTimeG} is an oscillation estimate on the solution to a specific Poisson equation with Neumann boundary conditions (Proposition~\ref{p:uosc1}, below).
We state this when it is first encountered, and prove it in the next subsection.

\begin{proof}[Proof of Lemma~\ref{l:FCLocalTimeG}]
The expectation in~\eqref{claim2G} can be written as
\begin{multline} \label{claim2Gsum}
\E^z \paren[\Big]{ \int_0^t  \alpha g(X_s^\epsilon)  \one_{\{Y_s^\epsilon < 0\}} \,ds  - 2 \int_0^t g(X_s^\epsilon) \, dL^{G}_s} \\
= \sum_{k \in \Z} g(\epsilon k)  \E^z \paren[\Big]{ \alpha\int_0^t \one_{\{Y_s^\epsilon < 0\}} \one_{\{ |X_s^\epsilon - \epsilon k| < \epsilon /2 \}} \,ds  - 2 \int_0^t \one_{\{ |X_s^\epsilon - \epsilon k| < \epsilon /2 \}} \, dL^{G}_s} \\
+ R^\epsilon
\end{multline}
where the remainder term $R^\epsilon$ is given by
\begin{multline*}
  R^\epsilon \defeq
	\alpha\sum_{k \in \Z}\E^z \paren[\Big]{ \int_0^t (g(X^\epsilon_s) - g(\epsilon k))\one_{\{Y_s^\epsilon < 0\}} \one_{\{ |X^\epsilon_s - \epsilon k| < \epsilon /2 \}} \,ds}\\
 -2\E^z \paren[\Big]{\int_0^t (g(X^\epsilon_s) - g(\epsilon k)) \one_{\{ |X^\epsilon_s - \epsilon k| < \epsilon /2 \}} dL^{G}_s} \defeq R^\epsilon_1 + R^\epsilon_2\,.
\end{multline*}
To estimate $R^\epsilon$, for any $\delta > 0$ we choose sufficiently large $M > 0$ such that
\begin{equation}\label{R-largek}
\sup_{(x,y) \in K}\E \paren[\Big]{\int_0^t\one_{\set{|x| + 4 |W_s| + 2 \geq M}} \, ds} < \frac{\delta}{\norm{g}_\infty} \,,
\end{equation}
where $W$ is a standard Brownian motion in $\R$. Here we write $\Pm$ and $\Em$ (without superscripts) to denote the probability measure and expected value for a standard Brownian motion. By Lemma \ref{lem:BrownianComp}, we have
\begin{equation*}
\P^z(|X^\epsilon_s| + 1 \geq M) \leq \P(x + 4 |W_s| + 2 \geq M ) \,,
\end{equation*}
where $z = (x,y)$ and so the above estimate can be applied for $X^\epsilon$ independent of $\epsilon \in (0,1/2]$. Since $g$ is continuous and hence uniformly continuous on $[-M,M]$, for any $\delta > 0$ we can choose $\epsilon > 0$ such that if $x_1,x_2 \in [-M,M]$ with $\abs{x_1 - x_2} < \epsilon$ then $\abs{g(x_1) -  g(x_2)} < \delta$. For such $\epsilon$ and for integers $k \in \epsilon^{-1}[-M,M]$ we have
\begin{multline}\label{R-smallk}
\E^{(x,y)}
\int_0^t\abs{g(\epsilon k) - g(X_s^\epsilon)}
\one_{\{Y_s^\epsilon < 0,\; \abs{X^\epsilon_s - \epsilon k} < \epsilon/2\}}\, ds\\
\leq \delta \int_0^t\P^z\paren[\Big]{\abs{X^\epsilon_s - \epsilon k} < \epsilon/2}\, ds\, .
\end{multline}
Combining the above with \eqref{R-largek}, gives the following estimate of $R^\epsilon_1$
\begin{align*}
\abs{R^\epsilon_1} &\leq \alpha\Biggl(\delta \sum_{\substack{k \in \mathbb{Z} \\ \epsilon k \in [-M,M]}}\int_0^t\P^z\paren[\Big]{\abs{X^\epsilon_s - \epsilon k} < \frac{\epsilon}{2}}\, ds\\
&\qquad\qquad+  2\norm{g}_\infty\sum_{|\epsilon k| >  M}
\int_0^t\P^z\paren[\Big]{\abs{X^\epsilon_s - \epsilon k} < \frac{\epsilon}{2}}\, ds\Biggr)
\\
&\leq \alpha(t + 2)\delta \,.
\end{align*}
Since $\delta > 0$ was arbitrary this proves $R^\epsilon_1\to 0$ as $\epsilon \to 0$. An estimate for $R^\epsilon_2$ can be obtained in the same manner.
Namely,
\begin{align*}
\abs{R^\epsilon_2} &\leq 2\paren[\bigg]{\delta\E^z\paren[\big]{L^G_t} +2\norm{g}_\infty\sum_{\substack{k \in \mathbb{Z} \\ |\epsilon k| \geq M}}
\E^z\paren[\Big]{\int_0^t\one_{\set{\abs{X^\epsilon_s - \epsilon k} < \epsilon/2}}\, dL^G_s}}\\
&\leq c(t)\delta + 2\norm{g}_\infty
\E^x\paren[\Big]{\int_0^t\one_{\set{|X^\epsilon_s|+1 \geq  M}}\, dL^G_s}\,.
\end{align*}
Let $\tau = \inf\set{t \st \abs{X_t^\epsilon} + 1 \geq M}$ and note that by the Markov property
\begin{align*}
  \nonumber
\E^z\paren[\Big]{\int_0^t\one_{\set{|X^\epsilon_s| +1 \geq M}}\, dL^G_s}
  &\leq \E^z\paren[\Big]{\E^{X^\epsilon_\tau}\paren[\Big]{L^G_{t - t\wedge\tau}}}
  \\
  &\leq \paren[\Big]{\sup_{z'}\E^{z'}\paren[\big]{L^G_t}} \P^z(\tau < t)\,.
\end{align*}
Applying It\^o's formula to $w(Z^\epsilon)$, where
\begin{equation*}
w(x,y) \defeq
\begin{dcases}
\frac{1}{2}(1 - y)^2\,, & y \in [0,1]\,,\\
0, & \text{otherwise,}
\end{dcases}
\end{equation*}
 shows
 \begin{equation}\label{e:ELG}
   \E^z(L^G_t) = O(1)
   \quad\text{as } t \to 0\,.
 \end{equation}
  By choosing $M$ larger, if necessary, we have 
 \begin{equation*}
 \sup_{z \in K} \P^z(\tau<t)<\delta
 \end{equation*} 
 for all $\epsilon \in (0,1/2]$.  Since $\delta > 0$ is arbitrary, this shows that $R^\epsilon_2\to 0$ as $\epsilon \to 0$.

Next, we need a PDE estimate to control the expression
\begin{align*}
 \E^z\paren[\Big]{ \alpha\int_0^t\one_{\{Y_s^\epsilon < 0\}} \one_{\{ |X_s - \epsilon k| < \epsilon /2 \}} \,ds  - 2 \int_0^t \one_{\{ |X_s - \epsilon k| < \epsilon /2 \}} dL^{G}_s}. 
\end{align*}
from \eqref{claim2Gsum}. To this end, let $Q$ be a region of width~$\epsilon$ directly below the tooth at~$x = 0$, and $G_0$ be the component of $G$ contained in $[-\epsilon/2,\epsilon/2] \times \R$.
Explicitly, let
\begin{equation}\label{e:QG0}
  Q \defeq
    \brak[\Big]{ -\frac{\epsilon}{2}, \frac{\epsilon}{2}}
    \times \brak[\big]{-\epsilon,0}
  \qquad\text{and}\qquad
    G_0 = \set[\Big]{ (x,0) \st
	- \frac{\alpha\epsilon^2}{2} < x < \frac{\alpha\epsilon^2}{2} }\,.
\end{equation}
Let $\mu^\epsilon$ denote the one dimensional Hausdorff measure supported on~$G_0$ (i.e. a measure supported on $G_0$).

\begin{proposition}\label{p:uosc1}
  Let the function $u^\epsilon\colon \Omega_\epsilon \to \R$ be the solution of
\begin{alignat}{2}
  \label{uoscpde}
  \span
    - \Delta u^\epsilon =  \alpha\one_{Q} - \mu^\epsilon
	&\qquad& \text{in } \Omega_\epsilon
  \\
  \label{e:uoscBC}
  \span
    \partial_\nu u^\epsilon = 0
      && \text{on } \partial \Omega_\epsilon \,,
\end{alignat}
with the normalization condition
\begin{equation} \label{uepsnorm1}
\inf_{\Omega_\epsilon} u^\epsilon = 0 \,.
\end{equation}
Then there exists a constant $C > 0$, independent of~$\epsilon$ such that
\begin{equation}\label{uosc1}
\sup_{\Omega_\epsilon} u^\epsilon(z)
  \leq C \epsilon^{2} \abs{\ln \epsilon}\,.
\end{equation}
\end{proposition}
\begin{remark*}
  Existence of a solution to~\eqref{uoscpde}--\eqref{e:uoscBC} can be proved by
  using~\cite[Thm.\ 2.2]{Droniou00} and a standard approximation argument to deal with the unbounded domain.  See also \cite[Thm.\ 2.2.1.3]{Grisvard85}.
\end{remark*}
Throughout the remainder of this proof and the section, we will use the convention that $C > 0$ is a constant that is independent of~$\epsilon$. We apply It\^o's formula to the function $u^\epsilon$ defined in Proposition \ref{p:uosc1} to obtain
  \begin{align*}
    2 \E^z \paren{ u^\epsilon(Z^\epsilon_t) - u^\epsilon(Z^\epsilon_0) }
      & = - \E^z \paren[\Big]{ \alpha\int_0^t \one_{Q}(Z_s^\epsilon) \, ds - 2 L^{G_0}_t 	}\,. \\
& =  \E^z\paren[\Big]{ \alpha\int_0^t    \one_{\{Y_s^\epsilon < 0\}} \one_{\{ |X_s| < \epsilon /2 \}} \,ds  - 2 \int_0^t \one_{\{ |X_s| < \epsilon /2 \}} dL^{G}_s}.
  \end{align*}
The oscillation bound \eqref{uosc1} now implies
\[
\left| \E^z\paren[\Big]{ \alpha\int_0^t    \one_{\{Y_s^\epsilon < 0\}} \one_{\{ |X_s - \epsilon k| < \epsilon /2 \}} \,ds  - 2 \int_0^t \one_{\{ |X_s - \epsilon k| < \epsilon /2 \}} dL^{G}_s} \right| \leq C \epsilon^2 |\log \epsilon|
\]
holds for all $k$ and $x \in \R$. Because of \eqref{R-largek}, we can restrict the sum in \eqref{claim2Gsum} to $k \in \Z$ for which $\epsilon |k| \leq M$ (i.e.\ only $O(\epsilon^{-1})$ terms in the sum).  Therefore,
\begin{multline*}
\sum_{\substack{k \in \Z \\\epsilon |k| \leq M} } \E^z\paren[\Big]{ \alpha\int_0^t    \one_{\{Y_s^\epsilon < 0\}} \one_{\{ |X_s^\epsilon - \epsilon k| < \epsilon /2 \}} \,ds  - 2 \int_0^t \one_{\{ |X_s^\epsilon - \epsilon k| < \epsilon /2 \}} dL^{G}_s} \\
 \leq O(\epsilon |\log (\epsilon)|).
\end{multline*}  
Combining this with the above estimates, we conclude that \eqref{claim2G} holds.
\end{proof}

To complete the proof of Lemma~\ref{l:FCLocalTimeG}, it remains to prove Proposition~\ref{p:uosc1}.
We do this in the next subsection.

\subsection{An Oscillation Estimate for the Neumann Problem (Proposition~\ref{p:uosc1}).}\label{s:uosc1}

The proof of Proposition~\ref{p:uosc1}  involves a ``geometric series'' argument using the probabilistic representation. 
Explicitly, we obtain the desired oscillation estimate by estimating the probabilities of successive visits of $Z^\epsilon$ between two segments.
The key step in the proof involves the so called narrow escape problem (see for instance~\cite{HolcmanSchuss14}), which guarantees that the probability that Brownian motion exists from a given interval on the boundary of a domain vanishes logarithmically with the interval size. 
In our specific scenario, however, we can not directly use the results of~\cite{HolcmanSchuss14} and we prove the required estimates here.

\begin{proof}[Proof of Proposition \ref{p:uosc1}]
  Note first that
  \begin{equation*}
    \int_{\Omega_\epsilon} \paren[\big]{ \alpha\one_Q - \mu^\epsilon } \, dz = 0\,,
  \end{equation*}
  and hence a bounded solution to~\eqref{uoscpde}--\eqref{e:uoscBC} exists.  
Moreover, because the measure $\alpha\one_{Q}(z) - \mu^\epsilon$ is supported in $\bar{Q}$, the function $u^{\epsilon}$ is harmonic in 
$
  \Omega_{\epsilon}
  \setminus \bar{Q}
$.
Thus, by the maximum principle,
\begin{equation*}
  \sup_{\Omega_\epsilon} u^\epsilon \leq \sup_{Q} u^\epsilon\,.
\end{equation*}

\begin{figure}[htb]
  \begin{center}
    \begin{tikzpicture}
      \fill[spinefill] (-5,0) rectangle (5,3);
      \fill[spinefill] (-.5,3) rectangle (.5,5);

      \fill [blue!20!white] (-2,0) rectangle (2,3);
      \draw node at (0,1.5) {Q};

      \draw (-5,0) -- (5,0);
      \draw (-3.5,3) -- (-.5,3);
      \draw (.5,3) -- (3.5,3);

      \draw (-.5,3) -- (-.5,5);
      \draw (.5,3) -- (.5,5);

      \draw[dashed](-.5,3) -- (.5,3);
      \draw[dashed](-.5,4) -- (.5,4);
      \draw node at (0,3.3) {$G_0$};
      \draw node at (0,4.3) {$A'$};

      \draw[dashed](-3,0) -- (-3,3);
      \draw[dashed](3,0) -- (3,3);
      \draw node at (-3.2,1.5) {$D'$};
      \draw node at (3.2,1.5) {$D'$};

      \draw[dashed](-2,0) -- (-2,3);
      \draw[dashed](2,0) -- (2,3);
      \draw node at (-2.2,1.5) {$D$};
      \draw node at (2.2,1.5) {$D$};

      \fill[spinefill] (-4.5,3) rectangle ++(1,2);
      \draw (-3.5,3) -- (-3.5,5);
      \draw (-4.5,3) -- (-4.5,5);

      \fill[spinefill] (3.5,3) rectangle ++(1,2);
      \draw (3.5,3) -- (3.5,5);
      \draw (4.5,3) -- (4.5,5);

      \draw (-5,3) -- (-4.5,3);
      \draw (5,3) -- (4.5,3);

    \end{tikzpicture}
    \caption{Image of one period of $\Omega_\epsilon$.}
  \end{center}
\end{figure}
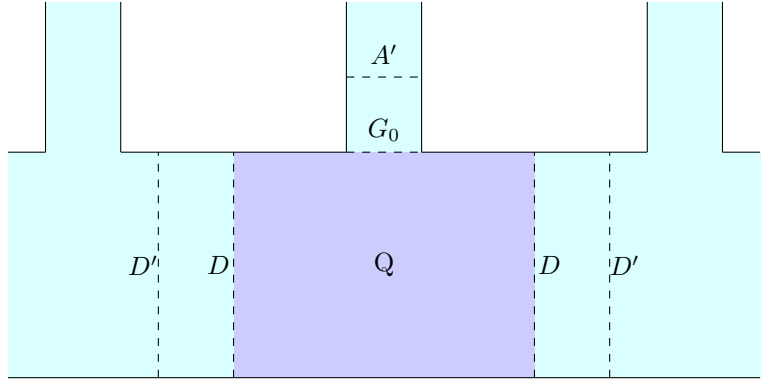
Define $Q' \supseteq Q$ to be the region that enlarges~$Q$ by $\epsilon^2$ on the top, and $\epsilon/4$ on the sides.
Precisely, let
\begin{equation*}
  Q' \defeq
    \Omega_{\epsilon} \bigcap \paren[\Big]{ \brak[\big]{-\frac{3\epsilon}{4}, \frac{3\epsilon}{4}} \times \brak[\big]{-\epsilon,\epsilon^2} }\,.
\end{equation*}
The first step is to estimate the oscillation of $u^\epsilon$ on the top and side portion of $Q'$.
Let~$A'$ and $D'$, defined by
\begin{equation}\label{e:ad}
  A' \defeq
    \brak[\Big]{-\frac{\alpha\epsilon^2}{2}, \frac{\alpha\epsilon^2}{2}}
    \times \set[\big]{ \alpha \epsilon^2 }
  \qquad\text{and}\qquad
  D' \defeq \set[\Big]{\pm \frac{3\epsilon}{4}}
    \times \brak[\big]{-\epsilon,0}
\end{equation}
denotes the top and sides of~$Q'$ respectively.
We aim to show
\begin{align}\label{uosctopbottom}
  \sup_{a,d \in A' \cup D'} \abs{u^\epsilon(a) - u^\epsilon(d)} \leq C \epsilon^2 \abs{\ln \epsilon} \,.
\end{align}


Let $\tau_0$ be the first time at which the process $Z^\epsilon_t$ hits the gate~$G_0$ (defined in~\eqref{e:QG0}).
The stopping time $\tau_0$ is finite almost surely, but has infinite expectation.
We claim that the distribution of~$Z^\epsilon_{\tau_0}$ on~$G$ is bounded below by a constant multiple of the Hausdorff measure, uniformly over all initial points in $A' \cup D'$.

\begin{lemma}\label{lem:rholower}
  For any~$z \in A' \cup D'$, let $\rho(z, \cdot)$, defined by
  \begin{equation*}
    \rho(z, r) = \P^z (Z^\epsilon_{\tau_0} \in dr )\,,
  \end{equation*}
  denote the density of the random variable $Z^\epsilon_{\tau_0}$ on~$G_0$.
  Then, there exists $\delta > 0$ such that
  \begin{equation}\label{rholowerS}
    \rho(z,r) \geq \frac{\delta}{ \alpha \epsilon^2}\,,
  \end{equation}
  for all $z \in A'\cup D'$  and $r \in G_0$.
\end{lemma}

Momentarily postponing the proof of this lemma, we note that for any $a, d \in A' \cup D'$, we have
\begin{align*}
  \MoveEqLeft
  \E^{a} u^\epsilon(Z^\epsilon_{\tau_0}) - \E^{d} u^\epsilon(Z^\epsilon_{\tau_0})  = \int_{G_0} \rho(a,r) u^\epsilon(r) \,dr - \int_{G_0} \rho(d,r) u^\epsilon(r) \,dr  \\
  & =  \int_{G_0} \paren[\Big]{ \rho(a,r) - \frac{\delta}{\alpha \epsilon^{2}} } u^\epsilon(r) \,dr
  - \int_{G_0} \paren[\Big]{ \rho(d,r) - \frac{\delta}{\alpha \epsilon^{2}} } u^\epsilon(r) \,dr  \\
  & \leq \paren{1 - \delta} \paren[\Big]{\sup_{G_0} u^\epsilon - \inf_{G_0} u^\epsilon }
    \leq \paren{1 - \delta} \paren[\Big]{\sup_{r_1, r_2 \in G_0} \abs{u^\epsilon(r_1) - u^\epsilon(r_2)} }\,.
\end{align*}
To obtain the second last inequality above we used the fact that
\begin{equation}
  \rho(z,r) - \frac{\delta}{ \alpha \epsilon^{2}} \geq 0\,,
\end{equation}
which is guaranteed by Lemma~\ref{lem:rholower}.

Now by It\^o's formula,
\begin{align}\label{uaubbound}
  \nonumber
  u^\epsilon(a) - u^\epsilon(d) & = \E^{a} u^\epsilon(Z^\epsilon_{\tau_0}) - \E^{d} u^\epsilon(Z^\epsilon_{\tau_0})  - \frac{1}{2} \E^a \paren[\Big]{
    2 L^{G_0^+}_{\tau_0}
    - \alpha\int_0^{\tau_0} \one_{Q}(Z^\epsilon_s) \, ds } \\
  \nonumber
  & \qquad + \frac{1}{2} \E^d \paren[\Big]{
    2 L^{G_0^+}_{\tau_0}
    - \alpha\int_0^{\tau_0} \one_{Q}(Z^\epsilon_s) \, ds } \\
  \nonumber
   \leq ( 1 - \delta)& \sup_{r_1,r_2 \in G_0} |u^\epsilon(r_1) - u^\epsilon(r_2)|   - \frac{1}{2} \E^a \paren[\Big]{
    2 L^{G_0^+}_{\tau_0}
    -\alpha \int_0^{\tau_0} \one_{Q}(Z^\epsilon_s) \, ds } \\
  & \qquad + \frac{1}{2} \E^d \paren[\Big]{
    2 L^{G_0^+}_{\tau_0}
    - \alpha\int_0^{\tau_0} \one_{Q}(Z^\epsilon_s) \, ds } 
\end{align}
Note that by definition of~$\tau_0$ we have we have $L_{\tau_0}^{G_0^+} = 0$ for all $a,d \in A' \cup D'$.
Also, if $a \in A'$, then $Y^\epsilon_s > 0$ for all $s \in [0,\tau_0]$ with probability one.
Hence
\begin{multline}\label{e:uaubbound2}
  \sup_{a, d \in A' \cup D'} \Bigl\lvert
    - \frac{1}{2} \E^a \paren[\Big]{
      2 L^{G_0^+}_{\tau_0}
      - \alpha\int_0^{\tau_0} \one_{Q}(Z^\epsilon_s) \, ds
    }
    \\
    + \frac{1}{2} \E^d \paren[\Big]{
      2 L^{G_0^+}_{\tau_0}
      - \alpha\int_0^{\tau_0} \one_{ Q }(Z^\epsilon_s) \, ds
    }
  \Bigr\rvert
  \leq \alpha\sup_{d \in D'} \E^d \int_0^{\tau_0} \one_Q(Z^\epsilon_s) \,ds\,.
\end{multline}

We claim that the term on the right is bounded by~$C \epsilon^2 \abs{\ln \epsilon}$.
To avoid distracting from the main proof, we single this out as a lemma and postpone the proof.

\begin{lemma}\label{lem:ABtoG}
  With the above notation,
  \begin{equation*}
    \sup_{d \in  D'}  \E^d \int_0^{\tau_{0}} \one_{Q}(Z^\epsilon_s) \, ds \leq C \epsilon^2 \abs{\ln \epsilon} \,.
  \end{equation*}
\end{lemma}

Using Lemma~\ref{lem:ABtoG} and~\eqref{e:uaubbound2} in~\eqref{uaubbound} we conclude
\begin{align} \label{uaubbound2}
\sup_{a,d \in  A' \cup D'} \abs{ u^\epsilon(a) - u^\epsilon(d)} \leq ( 1 - \delta) \sup_{r_1,r_2 \in G_0} \abs{u^\epsilon(r_1) - u^\epsilon(r_2)}   + C \epsilon^2 \abs{\ln \epsilon} \,.
\end{align}
To finish proving~\eqref{uosctopbottom}, we will now have to control the oscillation of~$u^\epsilon$ on $G_0$ in terms of the oscillation of~$u^\epsilon$ on $A' \cup D'$.

For this, given $Z^\epsilon_0 \in G_0$, let $\tau'_0$ be the first time that $Z^\epsilon_t$ hits $A' \cup D'$.
By It\^o's formula again, we have for all $r_1,r_2 \in G_0$:
\begin{multline} \label{uxuySbound}
u^\epsilon(r_1) - u^\epsilon(r_2)
  \leq \sup_{a',d' \in A' \cup D'} (u^\epsilon(a') - u^\epsilon(d'))
  \\
  -\frac{1}{2} \E^{r_1} \paren[\Big]{
	  2 L^{G_0}_{\tau'_0}
	  - \alpha\int_0^{\tau'_0} \one_{Q} \, ds }
  + \frac{1}{2} \E^{r_2} \paren[\Big]{
	  2 L^{G_0}_{\tau'_0}
	  - \alpha\int_0^{\tau'_0} \one_{Q } \, ds }\,.
\end{multline}
We claim that the last two terms above are $O(\epsilon^2)$.
For clarity of presentation we single this out as a Lemma and postpone the proof.

\begin{lemma}\label{lem:GtoAB}
  With the above notation
  \begin{equation*}
    \sup_{r \in G_0} \abs[\Big]{
      \E^r \paren[\Big]{
	2 L^{G_0}_{\tau'_0} - \alpha\int_0^{\tau'_0} \one_{Q}(Z^\epsilon_s) \, ds }
    }
    \leq C \epsilon^2\,.
  \end{equation*}
\end{lemma}

Using~\eqref{uxuySbound} and Lemma \ref{lem:GtoAB}, we see
\begin{align}
\sup_{r_1,r_2 \in G_0} \abs{u^\epsilon(r_1) - u^\epsilon(r_2)}
  \leq \sup_{a,d \in A' \cup D'} \abs{u^\epsilon(a) - u^\epsilon(d)}
    + C \epsilon^2\,.
\end{align}
Combining this with \eqref{uaubbound2}, we obtain
\begin{equation*}
  \sup_{a,d  \in A' \cup D'} \abs{u^\epsilon(a) - u^\epsilon(d)}
    \leq (1 - \delta) \paren[\Big]{ \sup_{a,d \in A' \cup D'} \abs{u^\epsilon(a) - u^\epsilon(d)} +  C \epsilon^2 \abs{\ln \epsilon}}  + C \epsilon^2\,.
\end{equation*}
and hence
\begin{align}
  \label{oscboundary}
  \sup_{a,d  \in A' \cup D'}
    \abs{u^\epsilon(a) - u^\epsilon(d)}
    \leq C \paren[\Big]{ \frac{1 - \delta}{\delta} } \epsilon^2 \abs{\ln \epsilon} + \frac{C}{\delta} \epsilon^2\,.
\end{align}
This proves~\eqref{uosctopbottom} as desired.

Now we turn this into an oscillation bound on $u^\epsilon$ over the interior.
Observe that for any $z \in \Omega_\epsilon$, 
\begin{align}
u^\epsilon(z) = \E^z[ u^\epsilon(Z^\epsilon_{\tau'_0})] + \frac{1}{2} \E^z \paren[\Big]{
	  2 L^{Y^\epsilon}_{\tau'_0} (0^+)
	  - \alpha\int_0^{\tau'_0} \one_{\set{Y^\epsilon_s \leq 0} } \, ds }
\end{align}
These last terms can be estimated with the same argument used in Lemma~\ref{lem:GtoAB}, leading to
\[
\sup_{z \in \Omega_\epsilon} \abs{u^\epsilon(z) - \E^z u^\epsilon(Z^\epsilon_{\tau'_0})} \leq C \epsilon^2 \,.
\]
The combination of this and \eqref{oscboundary} implies that 
\begin{equation*}
  \sup_{z_1,z_2 \in \Omega_\epsilon} \abs{u^\epsilon(z_1) - u^\epsilon(z_2)}
    \leq \sup_{z_1,z_2 \in \Omega_\epsilon}
      \abs{\E^{z_1}  u^\epsilon(Z^\epsilon_{\tau'_0})
	- \E^{z_2} u^\epsilon(Z^\epsilon_{\tau'_0})}
      + C \epsilon^2
    \leq C \epsilon^2 (\abs{\ln \epsilon} + 1)\,.
\end{equation*}
This implies \eqref{uosc1}, concluding the proof.
\end{proof}

For the proof of Lemma~\ref{lem:rholower} we will use a standard large deviation estimate for Brownian motion.
We state the result we need below.

\begin{lemma}\label{lem:steering}
Let $W_t$ be a standard Brownian motion in $\R^d$. Let $\gamma \in C([0,T];\R^d)$ be absolutely continuous with $S(\gamma) = \int_0^T |\gamma'(s)|^2 \,ds < \infty$. Then
\[
\P\left( \sup_{t \in [0,T]} |W(t) - \gamma(t)| \leq \delta \right)  \geq \frac{\P(K)}{2}  e^{- \frac{1}{2} S(\gamma) -  \sqrt{2S(\gamma)/\P(K)} } 
\]
where $K$ is the event $\{ \sup_{t \in [0,T]} |W(t)| \leq \delta \}$. 
\end{lemma}

The proof of Lemma~\ref{lem:steering} is standard -- it follows from a change of measure, as in the proof of Theorem 3.2.1 of \cite{FreidlinWentzell12}, for example. For convenience we provide a proof at the end of this section, and prove Lemmas~\ref{lem:rholower}, \ref{lem:ABtoG} and~\ref{lem:GtoAB} next.
\begin{proof}[Proof of Lemma \ref{lem:rholower}]
We need to show that for an interval $[r_1,r_2] \subset [-\alpha\epsilon^2/2, \alpha\epsilon^2/2]$,
\[
\inf_{z \in A' \cup D'} \P^z\left( Z^\epsilon_{\tau_0} \in [r_1,r_2] \times \{0\} \right) \geq C \frac{|r_2 - r_1|}{\alpha \epsilon^2}\,.
\]
Suppose $z \in D'$ (the case $z \in A'$ is similar but less complicated by the domain geometry).  In order to hit $G_0$, the process must first hit the boundary of $B(0,\alpha\epsilon^2)$ which is a ball of radius $\alpha\epsilon^2$, centered at the origin $(0,0)$, since $G_0 \subset B(0,\alpha\epsilon^2)$. So, by the strong  Markov property, it suffices to show that
\[
\inf_{z \in B(0,\epsilon^2) } \P^z\left( Z^\epsilon_{\tau_0} \in [r_1,r_2] \times \{0\} \right) \geq C \frac{|r_2 - r_1|}{\alpha \epsilon^2}.
\]

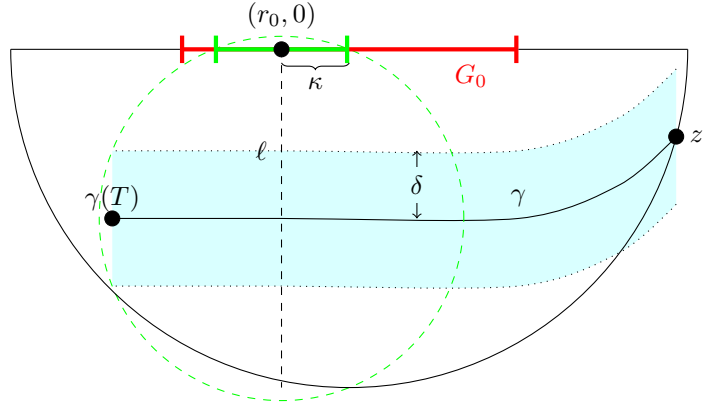
\begin{figure}[hbt]
  \begin{center}
    \begin{tikzpicture}[scale=.9]
      \fill[spinefill]
        plot [smooth, tension=.5]
          coordinates { (4.83, 2.71)  (4, 2) (2.5,1.5) (-1,1.5) (-3.5,1.5) }
        -- (-3.5,-.5)
        -- plot [smooth, tension=.5]
          coordinates {(-1,-.5) (2.5,-.5) (4, 0) (4.83, .71) }
        -- cycle;
      \draw  plot [smooth, tension=.5] coordinates { (4.83, 1.71)  (4, 1) (2.5,.5) (-1,.5) (-3.5,.5)};
      \draw[dotted]  plot [smooth, tension=.5] coordinates { (4.83, 2.71)  (4, 2) (2.5,1.5) (-1,1.5) (-3.5,1.5)};
      \draw[dotted]  plot [smooth, tension=.5] coordinates { (4.83, .71)  (4, 0) (2.5,-.5) (-1,-.5) (-3.5,-.5)};

      \draw[dashed, green] (-1, .5) circle (2.692582);
      \node[draw,circle,inner sep=2pt,fill] at (-3.5, .5) {};
      \draw node at (-3.5, .8) {$\gamma(T)$};
      \draw node at (2.5,.8) {$\gamma$};

      \draw[ultra thick, red, |-|](-2.5,3) -- (2.5,3);
      \draw[ultra thick, green, |-|](-2,3) -- ++(2,0);
      \draw (-5,3) -- (-2.5,3);
      \draw (2.5,3) -- (5,3);
      \draw node at (1.8,2.6) {$\textcolor{red}{G_0}$};
      \draw (-5,3) arc (180:360:5cm);

      \node[draw,circle,inner sep=2pt,fill] at (-1, 3) {};
      \draw node at (-1, 3.5) {$(r_0,0)$};
      \draw[decorate,decoration={brace,mirror}] (-1, 2.8) -- ++(1,0)
	node[pos=.5, anchor=north, yshift=-2pt] {$\kappa$};
      \draw [dashed](-1,3) -- (-1,-2);
      \draw node at (-1.3,1.5) {$\ell$};
      \node[draw,circle,inner sep=2pt,fill] at (4.83, 1.71) {};
      \draw node at (5.13,1.71) {$z$};

      \coordinate (A) at (1,1.5);
      \coordinate (B) at ( 1,.5);
      \draw[<->] (A) -- (B) node[midway,fill=spinefill] {$\delta$};


    \end{tikzpicture}
    \caption{The curve $\gamma$ starts on $\partial B(0,\epsilon^2)$, goes through the line $\ell$ while keeping a distance $\delta$ from the gate $G_0$. }
    \label{tikz:slit}
  \end{center}
\end{figure}
Suppose that $[r_1, r_2] = [r_0 - \kappa, r_0 + \kappa]$.
Let $\ell = \{r_0\} \times [-\epsilon^2,0)$ be the vertical line segment of length~$\epsilon^2$ below the desired exit interval.
Let  $T = \epsilon^4$, $\delta = \epsilon^2/4$, and let $\gamma$ be a curve parametrized by arc-length such that $\gamma(0) = z$ and the event $\sup_{t \in [0,T]} |Z^\epsilon(t) - \gamma(t)| \leq \delta$ implies that $Z^\epsilon$ hits~$\ell$ before $G_0$ (one example of such a curve is shown in Figure~\ref{tikz:slit}). We can choose such a curve $\gamma$ for which $|\gamma'| \leq O(\epsilon^{-2})$, so that the quantity $S(\gamma)$ in Lemma~\ref{lem:steering} is bounded independent of $\epsilon$ and of $z = \gamma(0) \in B(0,\epsilon^2)$. Notice also that the set $K$ from Lemma~\ref{lem:steering} satisfies
\begin{equation*}
	\P(K) = \P\paren[\big]{ \sup_{t \in [0,T]} |W(t)| \leq \delta} = \P\paren[\Big]{ \sup_{t \in [0,1]} |W(t)| \leq \frac{\delta}{\sqrt{T}}}
\end{equation*}
by Brownian scaling. Then since $\delta/\sqrt{T}$ is constant, this probability is bounded below and Lemma~\ref{lem:steering} states the probability that $Z^\epsilon_t$ hits~$\ell$ before $G_0$ is bounded below (away from zero), independent of $\epsilon$.
By the Markov property it now suffices to finish the proof assuming $z_0 \in \ell$.  Then consider the unique circle with center at $z_0 \in \ell$ such that the circle intersects $G_0$ at the points $(r_0 - \kappa, 0)$ and $(r_0 + \kappa, 0)$.  By symmetry of Brownian motion, the exit distribution on the circle is uniform. The probability that $Z^\epsilon_{\tau_0} \in [r_0 - \kappa, r_0 + \kappa]$ is at least the probability of exiting this circle along the arc above $G_0$, which is the ratio of the arc length to the circumference.  This probability is bounded below by $2 \kappa /(\alpha \epsilon^{2}) \gtrsim |r_1 - r_2|/(\alpha \epsilon^{2})$. 
\end{proof}

\begin{proof}[Proof of Lemma \ref{lem:ABtoG}]
  By the Markov property, Lemma~\ref{lem:ABtoG} will follow from the estimate
\begin{equation}\label{e:exit1}
\sup_{z\in Q}\E^z\int_0^{\tau_0} \one_Q(Z^\epsilon_s) \,ds \leq C \epsilon^2 \abs{\ln \epsilon}\,.
\end{equation}
Let $D = \set{\pm \epsilon/2}\times[-\epsilon,0]$ be the sides of~$Q$, and recall $D'$ (defined in~\eqref{e:ad}) denotes the sides of~$Q'$.
We consider two sequences of stopping times, ~$\zeta_i$, $\eta_i$, denoting successive visits of of~$Z^\epsilon$ to $G_0 \cup D'$ and $D$ respectively.
Precisely, let $\eta_0 = 0$, inductively define
\begin{align*}
  \zeta_i &= \inf\set{s > \eta_{i-1} \st Z^\epsilon_s \in G_0 \cup D'}\\
  \eta_i &= \inf\set{s > \zeta_i \st Z^\epsilon_s \in D}\,,
\shortintertext{for $i \in \set{1, 2, \dots}$, and let}
  M &= \min\set{n\in \mathbb{N} \st Z^\epsilon_{\zeta_n}\in G_0}\,.
\end{align*}
Notice that $\zeta_M = \tau_0$.
Using the strong Markov property, and the fact that $Z^\epsilon_s\notin Q$ for $s\in (\zeta_i,\eta_i)$ for all $i < M$, we obtain
\begin{align}\label{e:stopseq1}
  \nonumber
\E^z\int_0^{\tau_0}\one_Q(Z^\epsilon_s)\,ds
  &= \E^z\sum_{i=1}^M\int_{\eta_{i-1}}^{\zeta_i}\one_Q(Z^\epsilon_s)\,ds
  = \E^z\sum_{i=1}^M \E^{Z^\epsilon_{\eta_{i-1}}}\int_{0}^{\zeta_1}\one_Q(Z^\epsilon_s)\,ds
  \\
  &\leq (\E^z M) \paren[\Big]{
    \sup_{d \in D} \E^d \int_0^{\zeta_1} \one_Q (Z^\epsilon_s) \, ds
  }\,.
\end{align}
 Since $\zeta_1$ is bounded by the exit time of a one dimensional Brownian motion (the first coordinate of $Z^\epsilon$) from an interval of length $3\epsilon/2$, we know
\begin{equation*}
  \sup_{d \in D} \E^d \zeta_1 \leq C \epsilon^2\,.
\end{equation*}
Using this in~\eqref{e:stopseq1} shows
\begin{equation}\label{e:exit2}
  \E^z\int_0^{\tau_0}\one_Q(Z^\epsilon_s)\,ds
    \leq C \epsilon^2 \E^z M\,.
\end{equation}

We now estimate $\E^z M$. Notice that
\begin{align*}
  \P^z( M \geq n )
    &= \P^z( Z^\epsilon_{\zeta_1} \not\in G_0,\ Z^\epsilon_{\zeta_2} \not\in G_0,\ \dots,\ Z^\epsilon_{\zeta_n} \not\in G_0)\\
    &= \E^z \paren[\Big]{
	\one_{\set{Z^\epsilon_{\zeta_1} \not\in G_0,\ Z^\epsilon_{\zeta_2} \not\in G_0,\ \dots,\ Z^\epsilon_{\zeta_{n-1}} \not\in G_0} }
	\P^{Z^\epsilon_{\eta_{n-1}}} (Z^\epsilon_{\zeta_{1}} \not\in G_0)
    }
    \\
    &\leq \P^z \paren{
	Z^\epsilon_{\zeta_1} \not\in G_0,\ Z^\epsilon_{\zeta_2} \not\in G_0,\ \dots,\ Z^\epsilon_{\zeta_{n-1}} \not\in G_0}
      \paren[\Big]{ \sup_{d \in D} \P^{d} (Z^\epsilon_{\zeta_{1}} \not\in G_0) }
    \\
    &= \P^z( M \geq n-1 )
      \paren[\Big]{ \sup_{d \in D} \P^{d} (Z^\epsilon_{\zeta_{1}} \not\in G_0) }\,.
\end{align*}
Thus, by induction
\begin{equation*}
  \P^z ( M \geq n ) \leq
      \paren[\Big]{ \sup_{d \in D} \P^{d} (Z^\epsilon_{\zeta_{1}} \not\in G_0) }^n\,.
\end{equation*}

Now we claim that there exist a constant $c_0>0$, independent of $\epsilon$, such that
\begin{equation}\label{e:exitp1}
\sup_{d\in D} \P^d\paren[\big]{ Z^\epsilon_{\zeta_1} \not\in G_0} < 1 - \frac{c_0}{\abs{\ln \epsilon}}\,.
\end{equation}
This is the key step in the proof.
Once established, it implies
\begin{equation*}
  \E^z M  = \sum_{n = 1}^\infty \P^z( M \geq n )
    \leq \sum_{n = 1}^\infty \paren[\Big]{1 - \frac{c_0}{\abs{\ln \epsilon }}}^{n-1} = \frac{\abs{\ln \epsilon}}{c_0}\,,
\end{equation*}
which when combined with with~\eqref{e:exit2} yields
\begin{equation}\label{e:exit3}
\sup_{z\in D}\E^z\int_0^{\tau_1}\one_Q(Z^\epsilon_s)\,ds \leq \frac{C \epsilon^2 \abs{\ln \epsilon}}{c_0}\,.
\end{equation}
This proves~\eqref{e:exit1} and finishes the proof of Lemma~\ref{lem:ABtoG}.
\medskip

Thus it only remains to prove~\eqref{e:exitp1}.
We will prove it by showing
\begin{equation}\label{e:exitp0}
\inf_{z\in D} \P^z\paren[\big]{ Z^\epsilon_{\zeta_1} \in G_0} > \frac{c_0}{\abs{\ln \epsilon}}\,.
\end{equation}
We will prove this in three stages.
First, by scaling, it is easy to see that the probability that starting from~$D$ the process $Z^\epsilon$ hits $B(0, \epsilon / 4)$ before $D'$ with probability~$c_0 > 0$.
Next, using the explicit Greens function in an annulus we show that the probability that starting from $B(0, \epsilon/4)$, the process $Z^\epsilon$ hits $B(0, \alpha\epsilon^2)$ before exiting $B(0, \epsilon/2)$ with probability $c_0 / \abs{\ln \epsilon}$.
Finally, by scaling, it again follows that that starting from $B(0, \alpha\epsilon^2)$ the process $Z^\epsilon$ hits $G_0$ before exiting $B(0, 2\alpha\epsilon^2)$ with probability $c_0 > 0$.

For the first stage, consider the stopping times 
\begin{align*}
\sigma_{\epsilon/4} & = \inf \set[\Big]{ t > 0 \st Z^\epsilon_t \in B\paren[\Big]{0, \frac{\epsilon}{4}} } \,,\\
\sigma_{D'} & = \inf \set{ t > 0 \st  Z^\epsilon_t \in D' } \,.
\end{align*}
By rescaling, it immediately follows that
\begin{equation}\label{p0lower}
\inf_{z \in D} \P( \sigma_{\epsilon/4} < \sigma_{D'} \st Z^\epsilon_0 = z) \geq p_1\,,
\end{equation}
for some $p_1 > 0$, independent of $\epsilon$.

For the second stage suppose for $Z^\epsilon_0 \in \partial B(0,\epsilon/4)$.
Consider the stopping times~$\sigma_{\alpha\epsilon^2}$ and $\sigma_{\epsilon/2}$ defined by
\begin{align*}
\sigma_{\alpha\epsilon^2} & = \inf \{ t > 0 \st Z^\epsilon_t \in \partial B(0,\epsilon^2) \}\,,\\
\sigma_{\epsilon/2} & = \inf \{ t > 0 \st Z^\epsilon_t \in \partial B(0,\epsilon/2) \}\,.
\end{align*}
The function
\[
f(z) = \frac{\ln (2|z|/\epsilon)}{\ln(2 \alpha\epsilon)}
\]
is harmonic in $B(0,\epsilon/2) \setminus B(0,\alpha\epsilon^2)$ and satisfies $f = 1$ on $\partial B(0,\alpha\epsilon^2)$, and $f = 0$ on $\partial B(0,\epsilon/2)$. This implies that for all $z \in B(0, \epsilon/4)$ we have
\begin{equation}\label{loglower}
\P^z( \sigma_{\alpha\epsilon^2} < \sigma_{\epsilon/2} )
  = f(z)
  = \frac{\ln (1/2)}{\ln(2 \epsilon)} \,.
\end{equation} 

Finally, for the last stage, let $\sigma_{2\alpha\epsilon^2}$ be the first time $Z^\epsilon$ exits $B(0, 2\alpha\epsilon^2)$.
By scaling, it immediately follows that for all $z \in \partial B(0, \alpha\epsilon^2)$
\begin{equation}\label{peps2G0}
  \P^z( \tau_0 < \sigma_{2\alpha\epsilon^2} ) \geq p_2\,,
\end{equation}
for some constant $p_2 > 0$, independent of~$\epsilon$.

The strong Markov property and~\eqref{p0lower}, \eqref{loglower}, and~\eqref{peps2G0} imply
\[
\inf_{z\in D} \P(Z^\epsilon_{\zeta_1} \in G_0\st Z^\epsilon_{0}= z) \geq p_1 \cdot \frac{\ln (1/2)}{\ln(2 \alpha\epsilon)} \cdot p_2 \,.
\]
By the time-homogeneity of the Markov process $Z^\epsilon$, this establishes \eqref{e:exitp0}, finishing the proof.
\end{proof}
\begin{proof}[Proof of Lemma \ref{lem:GtoAB}]
To estimate the local time term, consider the function
\begin{equation*}
w(x,y) =
  \begin{dcases}
    \alpha \epsilon^2 - y\,, & y \in [0,\alpha \epsilon^2]\,,\\
    \alpha \epsilon^2, & \text{otherwise,}
  \end{dcases}
\end{equation*}
which satisfies $\partial_y w(x,0^+) - \partial_y w(x,0^-) = - 1$ for $x \in [-\alpha\epsilon^2/2, \alpha\epsilon^2/2]$. Let $\tau_{A'}$ be the first hitting time to the set $A'$, where we know $w = 0$.
Using It\^o's formula we obtain
\[
\E^z L^{G_0}_{\tau_{A'}}  = w(z), \quad z \in G_0.
\]
Clearly $\tau_{A'} \geq \tau'_0$, and so
\[
\sup_{z \in G_0} \E^z L^{G_0}_{\tau'_0}  \leq \sup_{z \in G_0} \E^z L^{G_0}_{\tau_{A'}} = \alpha \epsilon^2 \,.
\]

Next, we estimate the term
\begin{equation}\label{e:intZtau0prime}
\sup_{z \in G_0} \E^z \int_0^{\tau'_0} \one_{Q}(Z^\epsilon_s)  \, ds \,.
\end{equation}
Let $\tau_{D'} = \inf \{ t > 0 \st Z^\epsilon_t \in D'\}$, so that $\tau_{D'} \geq \tau'_0$.  Let $H = \{ (x,y) \in \R^2 \st y = -\epsilon \}$ denote the bottom boundary of $\Omega_0$, and let $H' = [-3\epsilon/4,3\epsilon/4] \times \{ -\epsilon \} = \bar{Q'} \cap H$. 
We now consider repeated visits to $H'$ before hitting $D'$.
For this, define the stopping times $\{\zeta_k \}_{k=0}^\infty$ inductively by
\begin{align*}
\zeta_0 & = \inf \set{ t > 0 \st Z^\epsilon_t \in H }\,,\\
\zeta_k & = \inf \set{ t \geq \zeta_{k-1} + \epsilon^2 \st Z^\epsilon_t \in H }\,,
  \quad \text{for } k = 1,2,3,\dots\,,\\
\shortintertext{and define}
M &= \min \{ k \in \mathbb{N} \st Z^\epsilon_{\zeta_k} \in H \setminus H' \}\,.
\end{align*}
Observe that if $Z^\epsilon_0 \in G_0$, then $\tau_{D'} \leq \zeta_M$.
Indeed, since $Z^\epsilon_{\zeta_M} \in H \setminus H'$ and trajectories of process $Z^\epsilon$ is continuous, they must must have passed through the set $D'$ at some time before~$\zeta_M$.

Now, to bound~\eqref{e:intZtau0prime} we observe
\begin{align}
 \int_0^{\tau'_0} \one_{Q}(Z^\epsilon_s)  \, ds \leq \int_0^{\zeta_0} \one_{Q}(Z^\epsilon_s)  \, ds + \sum_{k=1}^M \int_{\zeta_{k-1}}^{\zeta_k} \one_{Q}(Z^\epsilon_s)  \, ds. \label{tausum}
\end{align}
On the event $\{M > k-1\}$ we must have $Z^\epsilon_{\zeta_{k-1}} \in H'$. Using this observation, the strong Markov property, and the time-homogeneity of the process, we see that for any $z \in G_0$ we have
\begin{align} 
\E^z \int_0^{\tau'_0} \one_{Q}(Z^\epsilon_s)  \, ds & \leq \E^z \int_0^{\zeta_0} \one_{Q}(Z^\epsilon_s)  \, ds  + \E^z  \sum_{k=1}^M \int_{\zeta_{k-1}}^{\zeta_k} \one_{Q}(Z^\epsilon_s)  \, ds  \nonumber \\
& = \E^z  \int_0^{\zeta_0} \one_{Q}(Z^\epsilon_s)  \, ds + \E^z \sum_{k=1}^M \E^{Z^\epsilon_{\zeta_{k-1}}} \int_{\zeta_{0}}^{\zeta_1} \one_{Q}(Z^\epsilon_s)  \, ds  \nonumber \\
& \leq \E^z  \int_0^{\zeta_0} \one_{Q}(Z^\epsilon_s)  \, ds + \E^z \sum_{k=1}^M \sup_{z' \in H'} \E^{z'} \int_{\zeta_{0}}^{\zeta_1} \one_{Q}(Z^\epsilon_s)  \, ds  \nonumber \\
& = \E^z  \int_0^{\zeta_0} \one_{Q}(Z^\epsilon_s)  \, ds + (\E^z M) \sup_{z' \in H'} \E^{z'} \int_{\zeta_{0}}^{\zeta_1} \one_{Q}(Z^\epsilon_s)  \, ds\,.  \label{tausum2}
\end{align}

We now bound the right hand side of~\eqref{tausum2}.
Note
\begin{equation}\label{e:EM1}
  \E^z M
    = \sum_{j = 1}^\infty \P^z(M \geq j)
    = \sum_{j = 1}^\infty \P^z( Z^\epsilon_{\zeta_{0}} \in H',\ Z^\epsilon_{\zeta_1} \in H',\ \dots,\ Z^\epsilon_{\zeta_{j-1}} \in H' )\,.
\end{equation}
By the Markov property
\begin{align}
  \nonumber
  \P^z\paren[\big]{ Z^\epsilon_{\zeta_{i+1}} \in H',\ Z^\epsilon_{\zeta_i} \in H'} &= \E^z \paren[\big]{ \one_{Z^\epsilon_{\zeta_i} \in H'} \P^{Z^\epsilon_{\zeta_i}} ( Z^\epsilon_{\zeta_1} \in H' ) }
  \\
  \label{e:zetai1}
  &\leq \paren[\Big]{ \sup_{z' \in H'} \P^{z'}( Z^\epsilon_{\zeta_1} \in H' ) } \P^z( Z^\epsilon_{\zeta_i} \in H' )
\end{align}
Now using Lemma~\ref{lem:steering} and the fact that $\zeta_1 \geq \epsilon^2$,  one can show that
\begin{equation*}
   \sup_{z' \in H'} \P^{z'}( Z^\epsilon_{\zeta_1} \in H' ) \leq 1 - c_0\,,
\end{equation*}
for some constant $c_0 > 0$, independent of~$\epsilon$.
Combining this with~\eqref{e:zetai1} and using induction we obtain
\begin{equation*}
    \sum_{j = 1}^\infty \P^z( Z^\epsilon_{\zeta_{0}} \in H',\ Z^\epsilon_{\zeta_1} \in H',\ \dots,\ Z^\epsilon_{\zeta_{j-1}} \in H' )
      \leq \sum_{j=1}^\infty (1 - c_0)^{j-1}\,.
\end{equation*}
Thus, using~\eqref{e:EM1} we see 
\begin{equation*}
  \E^z M \leq \frac{1}{c_0}\,.
\end{equation*}

Using this in~\eqref{tausum2} we have
\begin{align}
\E^z \int_0^{\tau'_0} \one_{Q}(Z^\epsilon_s)  \, ds & \leq  \E^z \int_0^{\zeta_0} \one_{Q}(Z^\epsilon_s)  \, ds + \frac{1}{c_0} \sup_{z' \in H'} \E^{z'} \int_{\zeta_{0}}^{\zeta_1} \one_{Q}(Z^\epsilon_s)  \, ds   \nonumber \\
\label{Qtime1}
& \leq \E^z \int_0^{\zeta_0} \one_{Q}(Z^\epsilon_s)  \, ds + \frac{1}{c_0} \paren[\Big]{ \epsilon^2 + \sup_{z' \in \Omega} \E^{z'} \int_0^{\zeta_0} \one_{Q}(Z^\epsilon_s)  \, ds  }\,.
\end{align}
To bound this, consider the function
\begin{equation*}
  v(x,y)
    = \begin{dcases}
	\tfrac{1}{2} \paren{\epsilon^2 - y^2} \,, & y \in [-\epsilon,0]\,, \\
	\tfrac{1}{2} \epsilon^2 \,, & y > 0\,.
    \end{dcases}
\end{equation*}
and observe that for any $z \in \Omega_\epsilon$,
\begin{equation*}
\E^z \int_0^{\zeta_0} \one_{Q}(Z^\epsilon_s)  \, ds
  \leq \E^z \zeta_0
   =  v(z) \leq \frac{\epsilon^2}{2} \,.
\end{equation*}
Substituting this in~\eqref{Qtime1} shows
\[
\E^z \int_0^{\tau'_0} \one_{Q}(Z^\epsilon_s)  \, ds  \leq \paren[\Big]{\frac{1}{2} + \frac{3}{2c_0} } \epsilon^2\,,
\]
completing the proof.
\end{proof}

Finally, for completeness we prove Lemma~\ref{lem:steering}.
The proof is a standard argument using the Girsanov theorem, and can for instance be found in~\cite{FreidlinWentzell12} (see Theorem 3.2.1, therein).
\begin{proof}[Proof of Lemma~\ref{lem:steering}]
Define $Y(t) = W(t)  - \gamma(t)$.  Let $B(t)$ be an independent Brownian motion in $\R$ with respect to measure $\P$.  Let define a new measure $\bm{Q}$ by
\[
\frac{d \bm{Q}}{d\P} = e^{- \int_0^T \gamma'(s) \,dB(s) - \frac{1}{2} \int_0^T |\gamma'(s)|^2 \,ds}
\]
Let $\tilde K$ be the event $\tilde K = \tilde K_{T,\delta} = \{ \sup_{t \in [0,T]} |B(t)| \leq \delta \}$. Let $S(\gamma) = \int_0^T |\gamma'(s)|^2 \,ds$. According to the Girsanov theorem, 
\begin{align*}
\P( \sup_{ t \in [0,T]} |Y(t)| \leq \delta) & =  \bm{Q}( \tilde K ) \nonumber \\
& = \E_{\P}\left[ \one_{\tilde K} e^{- \int_0^T \gamma'(s) \,dB(s) - \frac{1}{2} \int_0^T |\gamma'(s)|^2 \,ds} \right]\\
& = e^{- \frac{1}{2} S(\gamma)} \E_{\P}\left[ \one_{\tilde K} e^{- \int_0^T \gamma'(s) \,dB(s)} \right]
\end{align*}
Now, by Chebychev and the It\^o isometry, 
\[
\P\left( \int_0^T \gamma'(s) \,dB(s) \geq \alpha \sqrt{S(\gamma)} \right) \leq \frac{1}{\alpha^2}
\]
So, if$\frac{1}{\alpha^2} \leq \frac{1}{2} \P(\tilde K)$, we have 
\[
\P \left( \sup_{ t \in [0,T]} |Y(t)| \leq \delta\right) \geq e^{- \frac{1}{2} S(\gamma) - \alpha \sqrt{S(\gamma)} }) \frac{1}{2} \P(\tilde K)
\]
In particular, by choosing $\alpha = \sqrt{2/\P(\tilde K)} > 0$, we have
\[
\P \left( \sup_{ t \in [0,T]} |Y(t)| \leq \delta\right) \geq e^{- \frac{1}{2} S(\gamma) -  \sqrt{2S(\gamma)/\P(\tilde K)} } \frac{1}{2} \P(\tilde K)
\]
Note that $\P(\tilde K) = \P(K)$ since $B$ and $W$ have the same law under $\P$. 
\end{proof}

\subsection{Local Time on Teeth Boundaries (Lemma~\ref{l:FCLocalTimeT}).}\label{s:FCLocalTimeT}
The last remaining lemma to prove is Lemma~\ref{l:FCLocalTimeT} which is the local time balance within the teeth. We again use the symmetry and geometric series arguments as in the proof of Proposition~\ref{p:uosc1}.
\begin{proof}[Proof of Lemma~\ref{l:FCLocalTimeT}]
As with \eqref{claim2G}, we will estimate 
\begin{multline} \label{claim1xk}
 I_k \defeq \E^z \Bigl( \int_0^t \frac{1}{2} \partial_x^2 f(Z^\epsilon_s) \one_{\{Y_s^\epsilon > 0\}} \one_{\{|X_s^\epsilon - \epsilon k| < \epsilon/2\}}\,ds
 \\
  +  \int_0^t \partial_x f(Z^\epsilon_s) \one_{\{|X_s^\epsilon - \epsilon k| < \epsilon/2\}} \, d L^{\pm}_s \Bigr)
\end{multline}
for any $z \in K \cap \Omega_\epsilon$.
As before, Lemma~\ref{l:FCLocalTimeT} will follow if we can show that for any finite $M$, $\sum_{\epsilon \abs{k} < M} I_k$ vanishes as $\epsilon \to 0$.
Since there are $O(1 / \epsilon)$ terms in the sum, it suffices to bound each $I_k$ by $o(\epsilon)$.
Without loss of generality, assume $k = 0$ and let $T_0 = [-\alpha\epsilon^2/2,\alpha\epsilon^2/2] \times [0,1]$ denote the tooth centered at $k = 0$.
Define the function~$\tilde f \colon T_0 \to \R$ by
\begin{equation*}
\tilde f(x,y) \defeq f(x,y) - f(0,y) - x \partial_x f(0,y) \,,
\end{equation*}
Note that for all $(x,y) \in T_0$ we have
\begin{equation*}
\tilde f(0,y) = 0\,,
\qquad
\partial_x \tilde f(0,y) = 0\,,
\qquad\text{and}\qquad \partial_x^2\tilde f(x,y) = \partial_x^2 f(x,y)\,.
\end{equation*}
and hence $\norm{\tilde f}_\infty = O(\epsilon^4)$.
Moreover,
\[
\partial_y^2\tilde f(x,y) = \partial_y^2 f(x,y) - \partial_y^2 f(0,y) - x \partial_x \partial_y^2 f(0,y) = O(\epsilon^4 ),
\]
assuming $\partial_y^2 f \in C^1$, and $\partial_y \tilde f(x,0) = O(\epsilon^4)$ for $x \in [-\alpha\epsilon^2/2,\alpha\epsilon^2/2]$.

We now extend the definition of $\tilde f$ continuously outside of $T_0$ (into the spine) to a $O(\epsilon^2)$ neighborhood of $G$ as follows.
Let $\eta(x,y)$ be a smooth, radially-symmetric cutoff function, vanishing outside of $B_{2}(0,0)$ and such that $\eta(z) = 1$ for $\abs{z} \leq 1$.
Then, for $y \leq 0$ (i.e.\ outside the tooth $T_0$), define
\begin{equation*}
  \tilde f(x,y) \defeq
    \eta\paren[\Big]{ \frac{x}{\alpha\epsilon^2}, \frac{y}{\alpha\epsilon^2}}
    \paren[\Big]{ f(x,0) - f(0,0) - x \partial_x f(0,0)  }\,.
\end{equation*}
In this way, $\tilde f$ has the additional properties that 
\begin{enumerate}
  \item $\tilde f$ vanishes outside of $T_0 \cup B_{2 \alpha\epsilon^2}(0,0)$,
  \item $\partial_y \tilde f = 0$ on $(\partial Q) \setminus G$, 
  \item The jump in $\partial_y \tilde f$ across $G$ is $O(\epsilon^4)$,
  \item $\Delta \tilde f= O(1)$ in the region $B^-_{2\alpha \epsilon^2} = \{y \leq 0\} \cap B_{2 \alpha\epsilon^2}(0,0)$.
\end{enumerate}
This last point stems from the fact that $|f(x,0) - f(0,0) - x \partial_x f(0,0) | = O(\epsilon^4)$. In view of this construction, we see that
\begin{align*}
  I_0 = \E^z \Bigl( \int_0^t \frac{1}{2} (\partial_x^2\tilde f &+ \partial_y^2\tilde f) (Z^\epsilon_s) \one_{\{Z_s^\epsilon \in T_0 \}} \,ds -  \int_0^t \partial_x \tilde f(Z^\epsilon_s) \one_{\{Z^\epsilon_s \in T_0\}}d L^{+}_s  \Bigr)   \\
  &\qquad\qquad + \E^z \paren[\Big]{ \int_0^t \partial_x f(0,Y_s^\epsilon)d(L^-_s - L^+_s) }   + O(\epsilon^2) t \\
  & = R_1 + R_2 + O(\epsilon^2) t\,.
\end{align*}
Notice how we have introduced the $\partial_y^2\tilde f$ term for the price of $O(\epsilon^2)t$.   We also still have $\partial_y \tilde f(x,1) = 0$ on the top boundary of the tooth.  By It\^o's formula applied to $\tilde f$, we have
\begin{align*}
R_1 & = \E^z [\tilde f(Z^\epsilon_t) - \tilde f(Z^\epsilon_0)]  + \E^z \paren[\Big]{ \int_0^t \partial_y \tilde f(X^\epsilon_s,0) dL^G } + \E^z  \paren[\Big]{\int_0^t O(1)\one_{B_{2\alpha\epsilon^2}}(Z_s) \,ds}   \\
& = O(\epsilon^4) + O(\epsilon^2) \E^z \paren[\Big]{ L^G_t }  +  O(1)  \E^z  \paren[\Big]{\int_0^t \one_{B^-_{2 \alpha\epsilon^2}}(Z_s) \,ds} \\
& = O(\epsilon^4) +  O(\epsilon^2) + O(1) R_3,
\end{align*}
by since $\E^z L^G_t = O(1)$ by~\eqref{e:ELG}.

We now estimate the term $R_2$.
By symmetry with respect to reflection in the $y$ coordinate, we note that
\[
\E^{z'} \paren[\Big]{ \int_0^t \partial_x f(0, Y_s^\epsilon)d(L^-_s - L^+_s) } = 0 
\]
for any $z' = (0,y)$ on the axis of the tooth $T_0$.
Thus by symmetry and the Markov property, it suffices to estimate
\begin{equation*}
\E^z \paren[\Big]{ \int_0^\tau \partial_x f(0, Y_s^\epsilon) \, dL^+_s }\,,
\end{equation*}
where $\tau = \inf\set{t \st X^\epsilon_t  = 0}$ is the first time that $Z^\epsilon_t$ reaches this $x$-axis $\{0\} \times \R$, and $z$ is to the right of the $y$-axis.
Clearly this is bounded by $\norm{\partial_x f}_\infty \E^z L^+_\tau$.
Moreover,  using $x \varmin \alpha \epsilon^2 / 2$ as a test function, we immediately see $\E^z L^+_\tau \leq \alpha \epsilon^2 / 2$.
This shows $R_2 = O(\epsilon^2)$ as desired.

Finally, we estimate the term
\[
R_3 = \E^z \paren[\Big]{\int_0^t \one_{B^-_{2 \alpha\epsilon^2}}(Z_s) \,ds},
\]
where $B^-_{2 \alpha\epsilon^2} = \{y \leq 0\} \cap B_{2 \epsilon^2}(0,0)$.
The geometry of the domain $\Omega_\epsilon$ makes this estimate a little tedious.
Since the proof is very similar to the arguments used in the proof of Proposition~\ref{p:uosc1}, we do not spell out all the details here.

We will show that $R_3 \leq O(\epsilon^3 |\log(\epsilon)|)$.
For this, we first claim
\[
\sup_{z \in \Omega_\epsilon \cap K} \E^z \paren[\Big]{\int_0^{\tau_{4\alpha\epsilon^2}} \one_{B^-_{2\alpha \epsilon^2}}(Z_s) \,ds} \leq O(\epsilon^4)
\]
where $\tau_{4 \alpha\epsilon^2} = \inf \{ t \;|\; Z^\epsilon_t \in D^-_{4 \epsilon^2} \}$, and $D^-_{4\alpha \epsilon^2} = \{ y \leq 0\} \cap \partial B_{4 \alpha\epsilon^2}(0,0)$.
This follows by directly applying It\^o's formula with a  function $f$ satisfying $\Delta f \leq 0$ in $\{ y \leq 0\} \cap B_{4 \alpha\epsilon^2}(0,0) \}$, with $\Delta f \leq - c < 0$ in $B^-_{2\alpha \epsilon^2}$.

Next, we claim that there is $C > 0$ such that
\[
  \inf_{z \in D^-_{4 \alpha\epsilon^2} } \P^z \paren[\Big]{ \sigma_{\epsilon/2} \leq \tau_{2\alpha \epsilon^2} } \geq \frac{C}{\abs{\log(\epsilon)}} \,,
\]
where $\sigma_{\epsilon/2} = \inf \{ t \st |X^\epsilon_t| = \epsilon/2 \}$ and $\tau_{2 \alpha\epsilon^2} = \inf \{ t \st Z^\epsilon_t \in B^-_{2 \alpha\epsilon^2} \}$.
This is the narrow escape asymptotics~\cites{HolcmanSchuss14}, and follows from a direct calculation with the Greens function in a manner similar to the proof of~\eqref{e:exitp1}.
Finally, we claim that for any $t > 0$, there is $C > 0$ such that
\[
\inf_{\{ |x| = \epsilon/2 \}} \Pm^z(\tau_{2 \alpha\epsilon^2} \geq t) \geq C \epsilon.
\]
This follows from comparison between $X^\epsilon_t$ and a standard Brownian motion on $\R$, via Lemma \ref{lem:BrownianComp}.  Thus, starting from $z \in D^-_{4 \alpha\epsilon^2}$, with probability at least $C \epsilon/|\log(\epsilon)|$ the process $Z_t$ will make a long excursion such that it doesn't return to $B^-_{2\alpha \epsilon^2}$ before time $t$. Using the same geometric series argument as in the proof of Lemma~\ref{lem:ABtoG}, we have
\[
R_3 \leq C (\log(\epsilon)/\epsilon) \sup_{z} \E^z \paren[\Big]{\int_0^{\tau_{4\alpha\epsilon^2}} \one_{B^-_{2\alpha \epsilon^2}}(Z_s) \,ds}  = O(\epsilon^3 |\log(\epsilon)|) \,,
\]
as claimed.

Finally, combining all these estimates we conclude that for any $k$, $I_k$ (defined in~\eqref{claim1xk}) is at most $O(\epsilon^2)$.
Consequently $\sum_{\epsilon \abs{k} < M} I_k \to 0$ as $\epsilon \to 0$, concluding the proof.
\end{proof}

\subsection{Remarks About Other Scalings.}\label{sec:otherscaling}
Consider a comb-shaped domain with the general scaling described in Remark \ref{scale2}.
For clarity, let us suppose that 
\[
w_S(\epsilon) = \epsilon^{\sigma} \,,
\qquad\text{and}\qquad
w_T(\epsilon) = \frac{\alpha \epsilon^{1 + \sigma}}{2}\,,
\]
for some~$\sigma > 0$.  Theorem \ref{t:zlimfat}, which we have proved already, pertains to the case $\sigma = 1$.   In the cases $\sigma <1$ and $\sigma > 1$, the same arguments may be applied, showing that the limit process is the same as with $\sigma = 1$.
Only a minor modification of Proposition~\ref{p:uosc1} and its supporting lemmas are required, and we sketch those modifications here. 

Analogous to the previous definition \eqref{e:QG0}, we define the sets
\begin{equation}
Q = \brak[\big]{ -\frac{\epsilon}{2}, \frac{\epsilon}{2}} \times \brak[\big]{ -\epsilon^\sigma,0} \quad \quad \text{and}\quad \quad G_0 = \set[\Big]{ (x,0) \st - \alpha \frac{ \epsilon^{1 + \sigma}}{2} < x < \alpha \frac{ \epsilon^{1 + \sigma}}{2} }. \label{Qnew}
\end{equation}
Notice that $Q$ is no longer a square if $\sigma \neq 1$.  In the case $\sigma > 1$, the bound $0 \leq u^\epsilon \leq C \epsilon^{2} |\ln \epsilon|$ in Proposition \ref{p:uosc1} remains unchanged.  The proofs of Lemma~\ref{lem:rholower}, Lemma~\ref{lem:ABtoG}, and Lemma~\ref{lem:GtoAB}, extend in a straightforward way.  In particular, the lower bound in Lemma \ref{lem:rholower} becomes $\rho(z,r) \geq \delta / (\alpha \epsilon^{1 + \sigma})$.  In the proof of \eqref{e:exitp0} within Lemma \ref{lem:ABtoG}, the balls $B(0,\epsilon^\sigma/4)$ and $B(0,\alpha \epsilon^{1 + \sigma})$ fill the roles of $B(0,\epsilon/4)$ and $B(0,\alpha \epsilon^2)$ in the previous proof. 

In the case $\sigma \in (0, 1)$, the bound on $u^\epsilon$ in Proposition \ref{p:uosc1} becomes $0 \leq u^\epsilon \leq C \epsilon^{1 +\sigma} |\ln \epsilon|$.  Nevertheless, this bound is still $o(\epsilon)$, so that the rest of the argument for the proof of Lemma~\ref{l:FCgenerator} proceeds as before.  To prove this modification of Proposition \ref{p:uosc1}, we can modify Lemma \ref{lem:rholower}, Lemma \ref{lem:ABtoG}, and Lemma \ref{lem:GtoAB}, as follows. First, $A'$ and $D'$ are defined to be the sets
\[
A' \defeq \brak[\Big]{ - \alpha \frac{\epsilon^{1 + \sigma}}{2} , \alpha \frac{\epsilon^{1 + \sigma}}{2}  } \times \{ \alpha \epsilon^{1 + \sigma} \} \quad \quad \text{and} \quad \quad D' \defeq \{ \pm \epsilon^{\sigma} \} \times [-\epsilon^\sigma,0].
\]
With these definitions, the lower bound of Lemma \ref{lem:rholower} becomes $\rho(z,r) \geq \frac{\delta}{\alpha \epsilon^{1 + \sigma}}$.    In Lemma \ref{lem:ABtoG}, the analogous bound becomes $O(\epsilon^{1 + \sigma}| \ln(\epsilon|)$.  Here, the logarithmic factor arises in the same way as before.  The $\epsilon^{1 + \sigma}$ factor comes from the fact that for a Brownian motion on $\Rm$, the expected time spent in $[-\epsilon,\epsilon]$ before hitting $\pm \epsilon^\sigma$ is $O(\epsilon^{1 + \sigma})$.  Similarly, the bound in Lemma \ref{lem:GtoAB} is $O(\epsilon^{1 + \sigma})$.  Together these imply the $O(\epsilon^{1 +\sigma} |\ln \epsilon|)$ upper bound in Proposition \ref{p:uosc1}.

\section{Comb-Shaped Graphs (Theorem~\ref{t:zlim}).}\label{s:thincomb}

\subsection{An SDE Description of \texorpdfstring{$Z^\epsilon$}{Z-epsilon}.}

We begin by constructing the graph diffusion $Z^\epsilon$ on the comb $\mathcal C_\epsilon$.
Following the approach of Freidlin and Sheu~\cite{FreidlinSheu00}, let $\mathcal{L}^\ep$ be the linear operator defined by
\begin{equation}\label{e:Lep}
  \mathcal{L}^\ep f =
    \begin{dcases}
	\frac{1}{2} \p_y^2f & \text{ if } (x,y) \in \ep\Z \times(0,1) \,,\\
	\frac{1}{2} \p_x^2f & \text{ if } (x,y) \in \R \times \{0\} \,.
    \end{dcases}
\end{equation}
Let the domain, denoted by $\mathcal{D}(\mathcal{L}^\ep)$, be the set of all functions
\begin{equation*}
  f\in C_0(\Omega_\ep) \cap  C^2_b(\Omega_\ep - J_\epsilon)
\end{equation*}
such that $\mathcal{L}^\ep f \in C_0(\Omega_\ep)$ and
\begin{subequations}
\begin{alignat}{2}
  \label{e:flux1}
  \span \alpha \ep \p_y f(x,0)
    + \p_x^+f(x,0) -\p_x^-f(x,0) = 0
    &\quad&\text{for } x \in \ep\Z \,,
  \\
  \label{e:flux2}
  \span\p_y f(x,1) = 0
    &&\text{for } x \in \ep\Z \,
\end{alignat}
\end{subequations}
The general theory in~\cite[\S4.1--4.2]{EthierKurtz86} (see also~\cite[Theorem~3.1]{FreidlinWentzell93}) can be used to show the existence of a continuous Fellerian Markov process $Z^\epsilon = (X^\epsilon, Y^\epsilon)$ that has generator $\mathcal{L}^\epsilon$.

In the teeth, and in between the nodes, it is clear that $Z^\epsilon$ is simply a Brownian motion.
The flux conditions~\eqref{e:flux1}--\eqref{e:flux2} introduce local time terms at junction points and ends of the teeth.
This can be stated precisely in terms of an It\^o formula as in the following Lemma.

\begin{lemma}\label{l:graphIto}
  Let $F$ be the set of all functions $f \in C(\mathcal C_\epsilon)$ such that $f$ is smooth on $\mathcal C_\epsilon - J_\epsilon$ and all one sided derivatives exist at the junction points $J_\epsilon$.
  There is a Brownian motion~$W$ such that for any for any $f \in F$ we have
  \begin{align*}
    df(Z^\epsilon_t)
      &= \one_{\set{Y^\epsilon_t = 0}} \partial_x f(Z^\epsilon_t) \, dW_t
	+ \frac{1}{2} \one_{\set{Y^\epsilon_t = 0}} \partial_x^2 f(Z^\epsilon_t) \, dt
    \\
      &\quad
	+ \one_{\set{Y^\epsilon_t > 0}} \partial_y f(Z^\epsilon_t) \, dW_t
	+ \frac{1}{2} \one_{\set{Y^\epsilon_t > 0}} \partial_y^2 f(Z^\epsilon_t) \, dt
    \\
      &\quad
      \frac{1}{2 + \alpha \epsilon}
      \paren[\Big]{
	 \partial_x^+ f(Z^\epsilon_t)
	- \partial_x^- f(Z^\epsilon_t)
	+ \alpha \epsilon \partial_y f(Z^\epsilon_t)
      } \, d\ell_t\,.
  \end{align*}
  Here $\ell$ defined by
  \begin{equation}\label{e:ellDef}
	  \ell_t = L_t^{Z^\epsilon}(J_\epsilon)
  \end{equation}
  is the local time of the joint process $Z^\epsilon_t = (X^\epsilon_t, Y^\epsilon_t)$ about the junction points $\epsilon \Z \times \set{0}$.
\end{lemma}
\begin{remark}
  The coefficients of each of~$\partial_x^-$, $\partial_x^+$ and $\partial_y$ in the local time term above can heuristically be interpreted the chance that~$Z^\epsilon$ enters the teeth.
\end{remark}
\begin{proof}
  We refer the reader to Section 2 (and specifically Lemma 2.3) in Freidlin and Sheu~\cite{FreidlinSheu00} where stochastic calculus for graph diffusions is developed in a general setting.
\end{proof}

Notice that choosing~$f(x,y) =x$ and~$f(x,y) = y$ in Lemma~\ref{l:graphIto} yields the following SDEs:
\begin{subequations}
\begin{align}
	\label{e:sdeXep}
	dX^\ep_t &= \one_{\{Y^\ep_t = 0\}} \, dW_t\,,
	\\
	\label{e:sdeYep}
	dY^\ep_t &= \one_{\{Y^\ep_t > 0\}} \, dW_t
	+ \frac{\alpha\epsilon}{2+ \alpha \epsilon}\, d\ell_t - dL^{Y^\epsilon}_t(1)\,
\end{align}
\end{subequations}
Note that~\eqref{e:sdeXep} and~\eqref{e:sdeYep} are coupled through the local time term $d\ell$, which is the local time of the joint process $Z^\epsilon = (X^\epsilon, Y^\epsilon)$ at the junction points~$J_\epsilon$.
We claim that with the additional assumption that the process spends~$0$ time in junctions, weak uniqueness holds for~\eqref{e:sdeXep}--\eqref{e:sdeYep}, and thus this system  can in fact be used to characterize the process $Z^\epsilon$.
Since this will not be used in this paper, we refer the reader to Engelbert and Peskir~\cite{EngelbertPeskir14} for the proof of similar results.

\subsection{Proof of Convergence (Theorem \ref{t:zlim}).}\label{sec:thinconv}

We now prove~Theorem~\ref{t:zlim}.
As with the proof of Theorem~\ref{t:zlimfat}, we need to prove tightness and a ``generator estimate''.
We state the results we require as the following two lemmas.

\begin{lemma}\label{l:tightness}
Let $Z^\epsilon = (X^\epsilon, Y^\epsilon)$ be the process on the comb-shaped graph~$\mathcal C_\epsilon$, as defined above. Then for any $T > 0$, the family of processes $Z^\epsilon$ is tight on $C([0,T]; \R^2)$.
\end{lemma}
\begin{lemma}\label{l:freidlin}
 Let $A$ be the generator \eqref{Adef}.  If $f \in \mathcal D(A)$, and $K \subseteq \Omega_0$ is compact as a subset of $\Rm^2$, then
  \begin{equation*}
    \lim_{\epsilon\to 0} \sup_{z \in K \cap \mathcal C_\epsilon} \E^{z} \paren[\Big]{
      f(Z^\epsilon_t) - f(Z_0) - \int_0^t Af(Z^\epsilon_s) \, ds
    }
    = 0
  \end{equation*}
\end{lemma}

\begin{proof}[Proof of Theorem~\ref{t:zlim}]
  Using Lemmas~\ref{l:tightness} and~\ref{l:freidlin} as replacements for Lemmas~\ref{l:FCtightness} and~\ref{l:FCgenerator} respectively, the proof of Theorem~\ref{t:zlim} is identical to that of Theorem~\ref{t:zlimfat}.
\end{proof}

The remainder of this section is devoted to proving Lemmas~\ref{l:tightness} and~\ref{l:freidlin}.

\begin{proof}[Proof of Lemma~\ref{l:tightness}]
  We write both $X^\epsilon$ and $Y^\epsilon$ as time-changed Brownian motions as follows. Let $S(t) = \int_0^t \one_{\set{Y^\epsilon_s = 0}} \,ds$. Then letting $S^{-1}(t)$ be the right-continuous inverse, by the Dambis-Dubins-Schwartz time change theorem (see for instance~\cite[Section 3.4.B]{KaratzasShreve91}), $\bar W_t = X^\epsilon_{S^{-1}(t)}$ is a Brownian motion and $X^\epsilon_t = \bar W_{S(t)}$.
Similarly we can time change $Y^\epsilon$ using $R(t) = \int_0^t \one_{\set{Y_t^\epsilon > 0}} \,ds$. Equation~\eqref{e:sdeYep} tells us that $\bar B_t = Y^\epsilon_{R^{-1}(t)}$ satisfies
\begin{equation*}
d\bar B_t = d \tilde B_t + dL^{\bar B}_t(0) - dL^{\bar B}_t(1) \, .
\end{equation*}
where $\tilde B_t$ is a Brownian motion and hence $\bar B_t$ is a doubly-reflected Brownian motion on $[0,1]$ such that $Y^\epsilon_t = \bar B_{R(t)}$.
Since $S(t) - S(s) \leq t - s$ and $R(t) - R(s) \leq t - s$ holds with probability one,  the moduli of continuity of $X^\epsilon$ and $Y^\epsilon$ over $[0,T]$ are no more than those of $\bar W$ and $\bar B$ over $[0,T]$, respectively.
This implies tightness.
\end{proof}
\begin{proof}[Proof of Lemma~\ref{l:freidlin}]
  We claim for any $k \in \N$ we have
\begin{equation*}
L^{Z^\epsilon}(\epsilon k, 0) =
 L^{X^\epsilon}(\epsilon k,0) + L^{Y^\epsilon}(\epsilon k, 0) \,,
 \quad\text{and}\quad
 L^{Y^\epsilon}(\epsilon k, 0) = \frac{\alpha\epsilon}{2} L^{X^\epsilon}(\epsilon k, 0) \,.
\end{equation*}
The first equality is immediate from the definition, and the second equality is proved in~\cite{FreidlinSheu00}.
(The second equality can also be deduced the independent excursion construction in Section~\ref{s:excursion}, below).
Consequently
\begin{equation}\label{e:LtZ}
  L^{Z^\epsilon}(\epsilon k, 0)
    = \frac{2 + \alpha\epsilon}{2} L^{X^\epsilon}(\epsilon k, 0)
    = \frac{2 + \alpha\epsilon}{\alpha\epsilon} L^{Y^\epsilon}(\epsilon k, 0) \,.
\end{equation}
For any $f\in \mathcal{D}(A)$, Lemma~\ref{l:graphIto} gives
\begin{multline}\label{e:fZt1}
f(Z^\epsilon_t) - f(Z^\epsilon_0) = \int_0^t \p_yf(Z^\epsilon_s)\one_{\{Y_s^\epsilon >0\}} \, dY^\epsilon_s
+ \int_0^t \p_xf(Z^\epsilon_s)\one_{\{Y_s^\epsilon = 0\}} \, dX^\epsilon_s \\+ \int_0^t \frac{1}{2}\p^2_yf(Z^\epsilon_s)\one_{\{Y_s^\epsilon >0\}} + \frac{1}{2}\p^2_xf(Z^\epsilon_s)\one_{\{Y_s^\epsilon = 0\}}\, ds \\
+ \sum_{k \in \Z} \paren[\Big]{
  \frac{\alpha\epsilon}{2 + \alpha\epsilon}\p_yf(\epsilon k,0)
  + \frac{1}{2 + \alpha\epsilon} \paren[\big]{
    \p_x^+f(\epsilon k, 0) - \p_x^-f(\epsilon k, 0)
  }}
L^{Z^\epsilon}_t(\epsilon k, 0)\,.
\end{multline}
The first integral on the right of equation~\eqref{e:fZt1} can be rewritten as
\begin{equation*}
  \begin{split}
    \int_0^t \p_yf(Z^\epsilon_s)\one_{\{Y_s^\epsilon >0\}} \, dY^\epsilon_s
     &=
     \int_0^t \p_yf(Z^\epsilon_s)\one_{\{Y_s^\epsilon >0\}} \, dW_s
     - \int_0^t \p_yf(X^\epsilon_s,1) \, dL^{Y^\epsilon}_s(1)
    \\
    &= \int_0^t\p_yf(Z^\epsilon_s)\one_{\{Y_s^\epsilon >0\}} \, dW_s \,.
 \end{split}
\end{equation*}
Here we used the fact that $\partial_y f(x, 1) = 0$ for any $f \in \mathcal{D}(A)$.

Returning to~\eqref{e:fZt1}, we note that $f\in C^2(\R \times\{0\})$ implies $\p_x^+f(\epsilon k, 0)=\p_x^-f(\epsilon k, 0)$.
Thus for $(x,y) \in K\cap \mathcal C_\epsilon $, taking expectations on both sides and using~\eqref{e:LtZ} gives
\begin{align*}
  \MoveEqLeft
  \E^{(x,y)} \paren[\Big]{ f(Z^\epsilon_t) - f(Z^\epsilon_0) -\int_0^t Af(Z^\epsilon_s) \, ds }\\
  &=\frac{1}{2} \E^{(x,y)} \Bigl( \int_0^t \p^2_yf(Z^\epsilon_s)\one_{\{Y_s^\epsilon >0\}} + \p^2_xf(Z^\epsilon_s)\one_{\{Y_s^\epsilon = 0\}} - \p^2_yf(Z^\epsilon_s)\, ds\\
  &\qquad\qquad+ \epsilon \sum_{k \in \Z} \p_yf(\epsilon k,0) L^{X^\epsilon}_t(\epsilon k, 0) \Bigr)\\
  &=\frac{\alpha}{2} \E^{(x,y)} \paren[\Big]{-\int_0^t \p_yf(X_s^\epsilon,0)\one_{\{Y^\epsilon_s = 0\}} \, ds + \epsilon \sum_{k \in \Z} \p_yf(\epsilon k,0) L^{X^\epsilon}_t(\epsilon k, 0)}
  \\
  &= I + \II\,,
\end{align*}
where
\begin{gather*}
  I \defeq \frac{\alpha}{2} \sum_{k \in \Z}\E^{(x,y)}
    \int_0^t\paren[\big]{
      \p_yf(\epsilon k, 0)
      - \p_yf(X_s^\epsilon, 0)
    }
    \one_{\{Y_s^\epsilon = 0,\; \abs{X^\epsilon_s - \epsilon k} < \frac{\epsilon}{2}\}}\, ds \,,
  \\
  \II \defeq \frac{\alpha}{2} \sum_{k \in \Z}
    \p_yf(\epsilon k, 0)
    \E^{(x,y)}\paren[\Big]{
      \epsilon L^{X^\epsilon}_t
      - \int_0^t\one_{\{Y_s^\epsilon = 0,\;\abs{X^\epsilon_s - \epsilon k} < \frac{\epsilon}{2}\}}\, ds
    }\,.
\end{gather*}
Note that there exists Brownian motion $W$ such that $X^\epsilon_t = W_{S(t)}$ 
where $S(t)$, defined by
\begin{equation*}
	S(t)
	\defeq \int_0^t \one_{\set{Y^\epsilon(s) = 0}} \, ds\,,
\end{equation*}
is the amount of time the joint process spends on the spine of the comb up to time~$t$.
To estimate $I$, for any $\delta > 0$ we choose sufficiently large compact set $C \subset \mathbb{R}$ such that
\begin{equation*}
\sup_{(x,y) \in K}\E^x\paren[\Big]{\int_0^t\one_{\set{W_s \notin C}} \, ds} < \frac{\delta}{\norm{\partial_y f}_\infty} \, .
\end{equation*}
Then since $S(s) \leq s$, it follows that
\begin{equation*}
	\Pm^x(X^\epsilon_s \notin C) \leq \Pm^x(W_s \notin C)
\end{equation*}
and so the above estimate can be applied for $X^\epsilon$ independent of $\epsilon$.
Then use uniform continuity of $\partial_y f$ in $C$ along with the above estimate.

In order to estimate $\II$, we again use the above representation to see
\begin{multline}\label{e:II2}
  \E^{(x,y)}\abs[\Big]{
    \epsilon L^{X^\epsilon}_t(\epsilon k, 0) - \int_0^t\one_{\{Y_s^\epsilon = 0,\;\abs{X^\epsilon_s - \epsilon k} < \frac{\epsilon}{2}\}}\, ds
  }
  \\
  =\E^{x} \abs[\Big]{
    \epsilon L^{W}_{S(t)}(\epsilon k) - \int_0^{S(t)}\one_{\{\abs{W_s - \epsilon k} < \frac{\epsilon}{2}\}}\, ds
  }
  \,,
\end{multline}
where $S(t)$, defined by
\begin{equation*}
  S(t)
    \defeq \int_0^t \one_{\set{Y^\epsilon(s) = 0}} \, ds\,,
\end{equation*}
is the amount of time the joint process spends on the spine of the comb up to time~$t$.
Thus to show $\II \to 0$, it suffices to estimate the right hand side of~\eqref{e:II2} as $\epsilon \to 0$. Also, by shifting the indices of the sum to compensate, we can assume that $x = 0$.

To this end, let $f_\epsilon$ be defined by
\begin{equation*}
 f_\epsilon(x) \defeq
 \begin{dcases}
 \epsilon(\epsilon k - x) -\frac{\epsilon^2}{4} & \text{ if } x < \epsilon k - \frac{\epsilon}{2}  \;,\\
 (x - \epsilon k)^2 & \text{ if } \epsilon k - \frac{\epsilon}{2}\leq x \leq \epsilon k + \frac{\epsilon}{2}  \;,\\
 \epsilon(x -\epsilon k ) -\frac{\epsilon^2}{4} & \text{ if } x > \epsilon k + \frac{\epsilon}{2}  \;.
 \end{dcases}
\end{equation*}
By Ito's formula we have,
\begin{multline*}
f_\epsilon(W_t) - \epsilon\abs{W_t - \epsilon k} - (f_\epsilon(W_0) - \epsilon\abs{W_0 - \epsilon k}) \\
= \int_0^t (f_\epsilon'(W_s) -\epsilon\, \text{sign}(W_s - \epsilon k)) \, dW_s + \int_0^t\one_{\{\abs{W_s - \epsilon k} < \frac{\epsilon}{2}\}}\, ds - \epsilon L^{W}_t(\epsilon k) \;.
\end{multline*}
Using the It\^o isometry and the inequalities
\begin{gather*}
  \abs[\big]{f_\epsilon(x) - \epsilon \abs{x - \epsilon k}}
    \leq  \frac{\epsilon^2}{4}\,,
  \\
  \abs{f'_\epsilon(x) - \epsilon\, \text{sign}(x - \epsilon k)} \leq  \epsilon \one_{[\epsilon k -\frac{\epsilon}{2},\epsilon k + \frac{\epsilon}{2}]} \,,
\end{gather*}
we obtain

\begin{multline*}
  \E^{0}\abs[\Big]{\epsilon L^{W}_t(\epsilon k) - \int_0^t\one_{\{\abs{W_s - \epsilon k} < \frac{\epsilon}{2}\}}\, ds }
    \leq \frac{\epsilon^2}{4} + \epsilon\left(\E^0 \int_0^t\one_{\{\abs{W_s - \epsilon k} < \frac{\epsilon}{2}\}} \,ds\right)^{\frac{1}{2}}
    \\
    \leq c(t)\epsilon^{\frac{3}{2}} \,,
\end{multline*}
since
\begin{equation*}
\E^0 \int_0^t\one_{\{\abs{W_s - \epsilon k} < \frac{\epsilon}{2}\}} \,ds = \int_0^t \P^0 \paren[\Big]{\abs{W_s - \epsilon k} < \frac{\epsilon}{2}} \, ds \leq c\int_0^t\frac{\epsilon}{\sqrt{s}} \, ds = 2c \epsilon \sqrt{t}\,.
\end{equation*}
We break up the sum in $\II$ and estimate as follows,
\begin{equation*}
  \II \leq \norm{\partial_y f}_\infty\paren[\Big]{\sum_{\abs{k} > N / \epsilon}\Em^0[\epsilon L_t^{X^\epsilon}(\epsilon k,0)] + \int_0^t \Pm^0\paren[\big]{\abs{X^\epsilon_s} > N - \frac{\epsilon}{2}} \,ds  + \frac{2 N }{\epsilon}c(t)\epsilon^{\frac{3}{2}}} \,.
\end{equation*}
We can again use that $X^\epsilon$ has the same distribution as a Brownian motion with a time change $S(t) \leq t$ to replace $X^\epsilon$ with $W$, i.e.
\begin{equation*}
  \II \leq \norm{\partial_y f}_\infty\paren[\Big]{\sum_{\abs{k} > N / \epsilon}\Em^0[\epsilon L_t^{W}(\epsilon k)] + \int_0^t \Pm^0\paren[\big]{\abs{W_s} > N - \frac{\epsilon}{2}} \,ds  + Nc(t)\epsilon^{\frac{1}{2}}} \,.
\end{equation*}
Setting $N$ sufficiently large and then sending $\epsilon \to 0$ gives us $\II\to 0$ as $\epsilon \to 0$.
This completes the proof.
\end{proof}

\section{Excursion Description on the Comb Graph.}\label{s:excursion}

In this section we describe the how diffusion~$Z^\epsilon$ on the comb-shaped graph~$\mathcal C_\epsilon$ (defined in Section~\ref{s:ithincomb}) can be constructed from the point of view of It\^o's excursion theory (c.f.~\cite{Ito72, PitmanYor07}).
We identify the components of~$Z^\epsilon$ as a trapped Brownian motion in the framework of Ben Arous \etal~\cite{BenArousCabezasEA15}, and use this to provide an alternate description of the limiting behavior as~$\epsilon \to 0$.

\subsection{The Excursion Decomposition of~\texorpdfstring{$Z^\epsilon$}{Z-epsilon}.}

The trajectories of $Z^\epsilon$ can be decomposed as a sequence of excursions where each excursion starts and ends at the junction points $J_\epsilon = \epsilon \Z \times \{0\}$, and travels entirely in the teeth, or entirely in the spine.
The excursions into the teeth of the comb (excursions of $Y^\epsilon$ into $(0,1]$ while $X^\epsilon \in \epsilon \Z$) should be those of a reflected Brownian motion on $[0,1]$.
The excursions into the spine (excursions of $X^\epsilon$ into $\Rm \setminus \epsilon \Z$ with $Y^\epsilon = 0$) should be those of a standard Brownian motion on $\Rm$ between the points $\epsilon \Z$.
Thus one expects that that by starting with a standard Brownian motion $\bar X$ on $\Rm$ and an independent reflected Brownian motion $\bar Y$ on $[0,1]$, we can glue excursions of $\bar X$ and $\bar Y$ appropriately and obtain the diffusion $Z^\epsilon$ on the comb-shaped graph~$\mathcal C_\epsilon$.
We describe this precisely as follows.

Let $\bar X$ be a standard Brownian motion on $\Rm$ and let $L^{\bar X}_t(x)$ denote its local time at~$x \in \Rm$.
Let $L_t^{\bar X}(\epsilon \Z)$, defined by
\begin{equation*}
  L^{\bar X}_t(\epsilon \Z) \defeq \sum_{k \in \Z} L^{\bar X}_t(\epsilon k) =  
\lim_{\delta \to 0} \frac{1}{2 \delta} \int_0^t \sum_{k \in \Z} \one_{(\epsilon k-\delta, \epsilon k+\delta)}(\bar X_s) \,ds \,,
\end{equation*}
denote the local time of $\bar X$ at the junction points $\epsilon \Z$.
Let $\tau^{\bar X, \epsilon}$ be the right-continuous inverse of $L^{\bar X}_t(\epsilon \Z)$ defined by
\begin{equation*}
\tau^{\bar X,\epsilon}(\ell) = \inf \set[\big]{ t > 0 \st L^{\bar X}_t(\epsilon \Z) > \ell }, \quad \ell \geq 0.
\end{equation*}
Notice that the functions $t \mapsto L^{\bar X}_t$ and $\ell \mapsto \tau^{\bar X, \epsilon}(\ell)$ are both non-decreasing.

Let $\bar Y$ be a reflected Brownian motion on $[0,1]$ which is independent of $\bar X$.  As above, let $L^{\bar Y}(0)$ be the local time of $\bar Y$ about $0$, and let $\tau^{\bar Y}$, defined by
\[
\tau^{\bar Y}(\ell) = \inf \set[\big]{ t > 0 \st L^{\bar Y}_t(0) > \ell }\,,
\]
be its right-continuous inverse.
Given $\alpha \in (0,1)$, we define the random time-changes $\psi^{\bar X, \epsilon}$ and $\psi^{\bar Y, \epsilon}$ by
\begin{equation}\label{e:psiXdef}
  \psi^{\bar X,\epsilon}(t)
    = \inf \set[\big]{
	s > 0 \st
	s + \tau^{\bar Y}\paren[\Big]{
	  \frac{\alpha \epsilon}{2} L^{\bar X}_s(\epsilon \Z)
	}
	> t
      }\,,
\end{equation}
and
\begin{equation}\label{e:epsiYdef}
  \psi^{\bar Y, \epsilon}(t)
    = \inf \set[\big]{
	s > 0 \st
	s + \tau^{\bar X,\epsilon}\paren[\Big]{
	  \frac{2}{\alpha \epsilon} L^{\bar Y}_s(0)
	}
	> t
      } \,.
\end{equation}
Note both $\psi^{\bar X, \epsilon}$ and~$\psi^{\bar Y, \epsilon}$ are continuous and non-decreasing functions of time.

\begin{proposition}\label{p:Ztimechange}
The time-changed process $Z^\epsilon$ defined by
\[
Z^\epsilon(t) \defeq \paren[\big]{ \bar X(\psi^{\bar X,\epsilon}(t)), \bar Y( \psi^{\bar Y,\epsilon}(t)) }
\]
is the same process~$Z^\epsilon$ in Theorem~\ref{t:zlim}.
Namely it is a Markov process with generator $\mathcal{L}^\epsilon$ (defined in equation~\eqref{e:Lep}), and is a weak solution of the system~\eqref{e:sdeXep}--\eqref{e:sdeYep}.
\end{proposition}

This gives an alternate and natural representation of $Z^\epsilon = (X^\epsilon,Y^\epsilon)$.  One can view this time-change representation as the pre-limit analogue of the representation~\eqref{e:limitdef1} for the limit system~\eqref{e:sdeX} -- \eqref{e:localtime}.
For clarity of presentation, we postpone the proof of Proposition~\ref{p:Ztimechange} to Section~\ref{s:Ztimechange}.

\begin{remark}
  For simplicity, throughout this section we assume the initial distribution of~$Z^\epsilon$ is~$\delta_{(0,0)}$, and denote expectations using the symbol $\E$ without any superscript.
  The main results here (in particular Theorem~\ref{t:XYconv}, below) can directly be adapted to the situation for more general initial distributions as in Theorem~\ref{t:zlim}.
\end{remark}

\subsection{Description as a Trapped Brownian Motion.}\label{s:tbm}
We now show how this representation can be explained in the framework of trapped Brownian motions as defined by Ben Arous, \etal~\cite{BenArousCabezasEA15} (see Definition 4.11 therein).  Recall that a trapped Brownian motion, denoted by $B[\mu]$, is a process of the form $B(\psi(t))$ where $B(t)$ is a standard Brownian motion and the time-change $\psi$ has the form
\[
\psi(t) = \inf \set[\big]{ s > 0 \st \phi[\mu,B]_s > t }\,,
\]
where
\[
\phi[\mu,B]_s = \mu \left(\{ (x,\ell) \in \Rm \times [0,\infty) \;|\; L^B(x,s) \geq \ell \} \right) \,,
\]
and $\mu$ is a (random) measure on $\Rm \times [0,\infty)$ called the trap measure.
For example, when $\mu$ is the Lebesgue measure on $\Rm \times [0,\infty)$, then $\phi[\mu,B] = t$, and $\psi(t) = t$.
Alternately, if $\mu$ has an atom at $(x,\ell)$ of mass $r > 0$, then $B(\psi(t))$ is trapped at $x$ for a time $r$ at the moment its local time at $x$ exceeds $\ell$.

To use this framework in our scenario, we need to identify a trap measure under which $X^\epsilon$ is a trapped Brownian motion.
We do this as follows.
First note that the process $\tau^{\bar Y}_\ell$, appearing in the time change~\eqref{e:psiXdef}, is a L\'evy subordinator.
Thus,
there exists a function $\eta^{\bar Y}(s):(0,\infty) \to (0,\infty)$, and a Poisson random measure $N^{\bar Y}$ on $[0,\infty) \times [0,\infty)$ with intensity measure $d\ell \times \eta^{\bar Y}(s) \, ds$,
such that
\begin{align}\label{tauPoisson}
\tau^{\bar Y}_\ell = \int_{[0,\ell]} \int_{[0,\infty)}  s N^{\bar Y} (d\ell \times ds)\,.
\end{align}
In the definition of $\psi^{\bar X,\epsilon}(t)$ above, we have 
\[
\tau^{\bar Y}\paren[\Big]{ \frac{\alpha \epsilon}{2} L^{\bar X}_s( \epsilon \Z) } = \tau^{\bar Y}\paren[\Big]{ \sum_{k \in \Z} \frac{\alpha \epsilon}{2} L^{\bar X}_s(\epsilon k) }.
\]
Because $\tau_\ell^{\bar Y}$ has stationary, independent increments, this is equal in law to
\begin{align*}
\tau^{\bar Y}\paren[\Big]{ \frac{\alpha \epsilon}{2} L^{\bar X}_s( \epsilon \Z) } \stackrel{d}{=} \sum_{k \in \Z} \tau^{\bar Y_k}\paren[\Big]{  \frac{\alpha \epsilon}{2} L^{\bar X}_s(\epsilon k) },
\end{align*}
where $\{ \bar Y_k\}_{k \in \Z}$ are a family of independent reflected Brownian motions on $[0,1]$.  That is, the time change $\psi^{\bar X,\epsilon}(t)$ has the same law as
\begin{align}\label{psihatdef}
\tilde \psi^{\bar X,\epsilon}(t) = \inf \set[\big]{ s > 0 \st s + \sum_{k \in \Z} \tau^{\bar Y_k}\paren[\Big]{  \frac{\alpha  \epsilon}{2} L^{\bar X}_s(\epsilon k) }  > t }\,.
\end{align}
Each of the processes $\tau^{\bar Y_k}$ can be represented as in~\eqref{tauPoisson} with independent Poisson random measures $N^{\bar Y_k}$:
\begin{align}
\tau^{\bar Y_k}_\ell = \int_{[0,\ell]} \int_{[0,\infty)}  s N^{\bar Y_k} (d\ell \times ds). \label{tauPoisson2}
\end{align}
Since each of the random measures $N^{\bar Y_k}$ is atomic, we may define $\{(\ell_{j,k}, s_{j,k})\}_{j=1}^\infty$ to be the random atoms of $N^{\bar Y_k}$ by
\begin{align}
N^{\bar Y_k} = \sum_{j = 1}^\infty \delta_{(\ell_{j,k}, s_{j,k})}. \label{tauPoisson3}
\end{align}
Then define a random measure on $\Rm \times [0,\infty)$:
\begin{align}
\mu^{\bar X,\epsilon} = dx \times d\ell + \sum_{k \in \Z} \sum_{j = 1}^\infty s_{j,k} \delta_{( \epsilon k, (2/ (\alpha \epsilon)) \ell_{j,k})} \label{mutrapx}
\end{align}
Returning to~\eqref{psihatdef}, we now have the representation
\[
s + \sum_{k \in \Z} \tau^{\bar Y_k}\left(  \frac{\alpha \epsilon}{2} L^{\bar X}_s(\epsilon k) \right) = \mu^{\bar X,\epsilon} \left( \{ (x,\ell) \in \Rm \times [0,\infty) \;|\; \ell \leq L^{\bar X}_s(x) \}\right).
\]
It is easy to check that $\mu^{\bar X}$ defines a L\'evy trap measure, in the sense of \cite{BenArousCabezasEA15}, Definition 4.10. This proves the following:
\begin{proposition}
Let $\bar X$ be a standard Brownian motion on $\Rm$ and let $\bar X[\mu^{\bar X,\epsilon}]$ be the trapped Brownian motion (see Definition 4.11 of \cite{BenArousCabezasEA15}) with trap measure $\mu^{\bar X,\epsilon}$ defined by~\eqref{mutrapx}.  Then the law of $X^\epsilon$ coincides with the law of $\bar X[\mu^{\bar X,\epsilon}]$.
\end{proposition}

The process $Y^{\epsilon}$ admits a similar representation as a trapped (reflected) Brownian motion.  To this end, we first note that $\tau^{\bar X,\epsilon}_\ell$ is also a L\'evy subordinator which and can be written as
\begin{align}
\tau^{\bar X,\epsilon}_\ell = \int_{[0,\ell]} \int_{[0,\infty)}  s N^{\bar X,\epsilon} (d\ell \times ds), \label{tauPoissonX}
\end{align}
where $N^{\bar X,\epsilon}$ is a Poisson random measure on $[0,\infty) \times [0,\infty)$ with intensity measure $d\ell \times \eta^{\bar X,\epsilon}(s)ds$.
\begin{lemma}\label{l:X-scaling}
	The excursion length measure $\eta^{\bar X,\epsilon}$ satisfies the scaling relation,
	\begin{equation*}
	\eta^{\bar X,\epsilon}(s) = \epsilon^{-3}\eta^{\bar X,1}(\epsilon^{-2} s), \quad s > 0.
	\end{equation*}
\end{lemma}

\begin{proof}
This follows in directly from the standard scaling properties of Brownian motion and its local time, and we omit the details.
\end{proof}

Letting $\{ (s_j,\ell_j) \}_{j=1}^\infty$ denote the atoms of $N^{\bar X,\epsilon}$ we then define a random measure on $[0,1] \times [0,\infty)$ by
\begin{align}\label{mutrapy}
\mu^{\bar Y,\epsilon}  = dy \times d\ell + \sum_{j=1}^\infty s_j \delta_{(0,(\alpha \epsilon/2) \ell_j)} \,.
\end{align}
This also is a L\'evy Trap Measure in the sense of \cite{BenArousCabezasEA15} (replacing $\Rm$ by $[0,1]$), and one can easily see that the associated trapped Brownian motion is precisely the process~$Y^\epsilon$.

\begin{proposition}
Let $\bar Y$ be a reflected Brownian motion on $[0,1]$, and let $\bar Y[\mu^{\bar Y,\epsilon}]$ be the trapped Brownian motion with trap measure $\mu^{\bar Y,\epsilon}$ defined by~\eqref{mutrapy}.
Then the law of $Y^\epsilon$ coincides with the law of $\bar Y[\mu^{\bar Y,\epsilon}]$.
\end{proposition}

\subsection{Convergence as \texorpdfstring{$\epsilon \to 0$}{epsilon to 0}.}

We now use Theorem 6.2 of \cite{BenArousCabezasEA15} to study convergence of $X^\epsilon$ and $Y^\epsilon$ as $\epsilon \to 0$.
The key step is to establish convergence of the trap measures, as in the following lemma.


\begin{lemma}\label{lem:trapmeasurelimits}
  Let $N_*^{\bar Y}$ be a Poisson random measure on $\Rm \times [0,\infty) \times [0,\infty)$ with intensity measure $dx \times d\ell \times \frac{1}{2} \eta^{\bar Y}(s) \, ds$.
  As $\epsilon \to 0$, the random measures $\mu^{\bar X,\epsilon}$ on $\Rm \times [0,\infty)$, defined in~\eqref{mutrapx}, converge vaguely in distribution to the random measure $\mu^{X}_*$ defined by
  \begin{equation*}
    \mu^{X}_*(A) = \int_\Rm \!\int_0^\infty \one_A(x,\ell) dx \, d\ell + \frac{\alpha}{2} \int_{\Rm} \int_0^\infty \int_0^\infty  \one_A(x,\ell) s N_*^{\bar Y}\left(dx \times d\ell \times ds \right) \,,
  \end{equation*}
  for all $A \in \mathcal{B}(\Rm \times [0,\infty))$.  The random measures $\mu^{\bar Y,\epsilon}$ on $[0,1] \times [0,\infty)$, defined in~\eqref{mutrapy}, converge vaguely in distribution to the measure $\mu^{Y}_*$ defined by
  \begin{equation*}
    \mu^{Y}_*(A) = \int_0^1 \! \int_0^\infty \one_A(y,\ell) dy \, d\ell  + \frac{2}{\alpha} \int_0^\infty \one_{A}(0,\ell) \,d\ell \quad \quad A \in \mathcal{B}([0,1] \times [0,\infty)) \,.
  \end{equation*}
\end{lemma}

Momentarily postponing the proof of Lemma~\ref{lem:trapmeasurelimits}, we state the main convergence result in this section.

\begin{theorem}\label{t:XYconv}
  Let $R(t)$ be a Brownian motion on $[0,1]$ reflected at both endpoints $x = 0,1$, and $B$ be a standard Brownian motion on $\Rm$.  
  \begin{enumerate}
    \item
      As $\epsilon \to 0$, we have $Y^\epsilon \to Y$ vaguely in distribution on $D([0,\infty))$.
      Here $Y = R[\mu_*^{\bar Y}]$ is a reflected Brownian motion that is sticky at $0$.
    \item 
      As $\epsilon \to 0$, we have $X^\epsilon \to B[\mu_*^{\bar X}]$ vaguely in distribution on $D([0,\infty))$.
      The limit process here may also be written as $B((2/\alpha) L^Y_t(0))$.
  \end{enumerate}
\end{theorem}

\begin{remark}
  Using the SDE methods in Section~\ref{s:thincomb} we are able to obtain joint convergence of the pair $(X^\epsilon, Y^\epsilon)$ (Theorem~\ref{t:zlim}).
  The trapped Brownian motion framework here, however, only provides convergence of the processes $X^\epsilon$ and $Y^\epsilon$ individually.
\end{remark}

\begin{proof}[Proof of Theorem~\ref{t:XYconv}]
  The convergence of $Y^\epsilon$ to $R[\mu_*^{\bar Y}]$ is an immediate consequence of Theorem 6.2 of \cite{BenArousCabezasEA15}, Lemma \ref{lem:trapmeasurelimits} above, and the properties of Poisson random measures.
  To identify the limiting process $R[\mu_*^{\bar Y}]$ as a sticky Brownian motion, observe that the time change has the form
  \begin{equation*}
    \mu_*^{\bar Y} \paren[\big]{ \set[\big]{ (y,\ell) \in [0,1] \times [0,\infty) \st L^{R}(y,s) \geq \ell } } = s + \frac{2}{\alpha} L^{R}(0,s)\,.
  \end{equation*}
  Thus, the limit process is $Y(t) = R(\psi(t))$ where
  \begin{equation*}
    \psi(t) = \inf \{ s > 0 \;|\; s + \frac{2}{\alpha} L^{R}(0,s) > t \}\,.
  \end{equation*}
  This is precisely a sticky Brownian motion (see Lemma~\ref{l:sdeZ}).

  For the second assertion of the Theorem, the convergence of $X^\epsilon$ to $B[\mu_*^{\bar X}]$ is again an immediate consequence of Theorem 6.2 of \cite{BenArousCabezasEA15} and Lemma \ref{lem:trapmeasurelimits} above. 
  Thus we only need to show that the trapped Brownian motion $B[\mu_*^{\bar X}]$ has the same law as the process $X_t$ from Theorem \ref{t:zlim}.
  To compare the two processes, we first write them in a similar form.
  Let $L^{\bar B}_t(0)$ is the local time of $\bar B$ at $0$, and let $\tau^{\bar B}_\ell$ be the inverse
\[
\tau^{\bar B}_\ell = \inf \{ t > 0 \;|\; L^{\bar B}_t(0) > \ell \}.
\]
Then, we have
\[
X_t = \bar W_{\frac{2}{\alpha} L^{\bar B}_{T(t)}} = \bar W(h^{-1}(t))
\]
where 
\[
h^{-1}(t) = \inf \{ r > 0 \;|\; r + \tau^{\bar B}_{r\alpha/2} > t \}
\]
The fact that $(2/\alpha) L^{\bar B}_{T(t)} = h^{-1}(t)$ follows from the definition of $T(t)$, which implies $(2/\alpha) L^{\bar B}_{T(t)} + T(t) = t$.

Therefore, the two processes are
\[
B[\mu_*^{\bar X}] = B (\phi^{-1}(t)) \quad \quad \quad X_t = \bar W(h^{-1}(t))
\]
where $\phi$ is:
\[
\phi(r) = \phi[\mu_*,B]_r =  \mu_* \left( \{ (x,\ell) \in \Rm \times [0,\infty) \;|\; L^B(x,r) \geq \ell \} \right)
\]
If $A_r^B = \{ (x,\ell) \in \Rm \times [0,\infty) \;|\; L^B(x,r) \geq \ell \}$, then by definition of the trap measure $\mu_*$,
\begin{align}
\phi(r) = r +  \frac{\alpha}{2} \int_{A_r^B \times [0,\infty)}   s N_*^{\bar Y}\left(dx \times d\ell \times ds \right) \label{timechangecomp}
\end{align}

The last integral has the same law as $\tau^{\bar B}_{r\alpha/2}$.  Hence, $h$ and $\phi$ have the same law.

Notice that $h$ is independent of $\bar W$.
We claim that $\phi$ is also independent of $B$.
To see this observe that the distribution of~$\phi(r)$ only depends on~$B$ through the volume of~$A_r^B$, which equals $r$ almost surely.
This shows~$\phi$ is independent of~$B$, and thus $B(\phi^{-1}(t))$ and $\bar W(h^{-1}(t))$ have the same law.
\end{proof}

It remains to prove Lemma~\ref{lem:trapmeasurelimits}.

\begin{proof}[Proof of Lemma \ref{lem:trapmeasurelimits}]
  It suffices to show for rectangles $A = [x_0, x_1]\times[\ell_0,\ell_1]$ that
  \begin{equation*}
  \mu^{\bar X,\epsilon}(A) \to \mu^{X}_*(A)
  \end{equation*}
  in distribution. We calculate the characteristic function using \cite[Thm~2.7]{Kyprianou06},
  \begin{align*}
  \Em[e^{i\beta\mu^{\bar X,\epsilon}(A)}] &= \exp\paren[\Big]{i\beta\abs{A} + \sum_{\epsilon k \in [x_0,x_1]}\int_{\frac{\epsilon}{2}\ell_0}^{\frac{\epsilon}{2}\ell_1}\int_0^\infty (1 - e^{i\beta s}) \eta^{\bar Y}(s)\, ds }\\
  &= \exp\paren[\Big]{i\beta\abs{A} + \paren[\Big]{\floor[\Big]{\frac{x_1}{\epsilon}} - \ceil[\Big]{\frac{x_0}{\epsilon}}}\frac{\epsilon(\ell_1 - \ell_0)}{2}\int_0^\infty (1 - e^{i\beta s}) \eta^{\bar Y}(s)\, ds }\\
  &\to \exp\paren[\Big]{i\beta\abs{A} + \frac{\abs{A}}{2}\int_0^\infty (1 - e^{i\beta s}) \eta^{\bar Y}(s)\, ds }\,
  \end{align*}
  as $\epsilon \to 0$. We note that this last formula is the characteristic function for $\mu^X_\star(A)$. The calculation for $\mu^{\bar Y,\epsilon}(A)$  uses Lemma~\ref{l:X-scaling} and a change of variables as follows
  \begin{align*}
  \Em[e^{i\beta\mu^{\bar Y,\epsilon}(A)}] &= \exp\paren[\Big]{i\beta\abs{A} + \one_{[y_0,y_1]}(0)\int_{\frac{2}{\epsilon}\ell_0}^{\frac{2}{\epsilon}\ell_1}\int_0^\infty (1 - e^{i\beta s}) \eta^{\bar X, \epsilon}(s)\, ds }\\
  &= \exp\paren[\Big]{i\beta\abs{A} +  \one_{[y_0,y_1]}(0)\frac{2(\ell_1 - \ell_0)}{\epsilon^4}\int_0^\infty (1 - e^{i\epsilon^2\beta s}) \eta^{\bar X,1}(\epsilon^{-2} s)\, ds }\\
  &= \exp\paren[\Big]{i\beta\abs{A} +  \one_{[y_0,y_1]}(0)\frac{2(\ell_1 - \ell_0)}{\epsilon^2}\int_0^\infty (1 - e^{i\epsilon^2\beta s}) \eta^{\bar X, 1}(s)\, ds }\,.
  \end{align*}
  Notice that by switching the integrals, we find 
  \begin{align*}
  \frac{1}{\epsilon^2}\int_0^\infty(1-e^{i\beta\epsilon^2 s})\eta^{\bar X, 1}(s)\, ds &= \frac{1}{\epsilon^2}\int_0^\infty(-\beta i\epsilon^2\int_0^s e^{i\beta\epsilon^2 r}\, dr)\eta^{\bar X, 1}(s)\, ds\\
  &= \int_0^\infty e^{i\beta\epsilon^2 r} \int_r^\infty \eta^{\bar X, 1}(s)\, ds \, dr \, .
  \end{align*}
  Since $\eta^{\bar X, 1}$ has exponential tails, we can send $\epsilon \to 0$, use dominated convergence and switch the integrals again to find
  \begin{align*}
  \lim_{\epsilon\to 0}\frac{1}{\epsilon^2}\int_0^\infty(1-e^{i\beta\epsilon^2 s})\eta^{\bar X, 1}(s)\, ds =\int_0^\infty s \eta^{\bar X, 1}(s)\, ds = 1 \, 
  \end{align*}
  and hence 
  \begin{equation*}
  \Em[e^{i\beta\mu^{\bar Y,\epsilon}(A)}] \to \Em[e^{i\beta\mu^{Y}_*(A)}] \, .
  \qedhere
  \end{equation*}
\end{proof}

\subsection{Proof of the Excursion Decomposition (Proposition~\ref{p:Ztimechange}).}\label{s:Ztimechange}
To abbreviate the notation, we will now write $L^{\bar X}_t$ and $L^{\bar Y}_t$ for $L^{\bar X}_t(\epsilon \Z)$ and $L^{\bar Y}_t(0)$, respectively.
Notice that $L^{\bar X}_t$ depends on $\epsilon$ while $L^{\bar Y}_t$ does not.
Let $X^\epsilon(t) =  \bar X(\psi^{\bar X,\epsilon}(t))$ and $Y^{\epsilon}(t) = \bar Y( \psi^{\bar Y,\epsilon}(t))$.
The proof of Proposition~\ref{p:Ztimechange} follows quickly from It\^o's formula, and the following two lemmas:
\begin{lemma}\label{l:claimLtratio}
  For every $t\geq 0$, we have
  \begin{equation}\label{e:claimLtratio}
    L^{X^\epsilon}_t = \frac{2}{\alpha \epsilon} L^{Y^\epsilon}_t \,.
  \end{equation}
\end{lemma}
\begin{lemma}\label{l:jqvXY}
  The joint quadratic variation of $X^\epsilon$ and $Y^\epsilon$ is $0$.
\end{lemma}

Momentarily postponing the proof of these lemmas, we prove Proposition~\ref{p:Ztimechange}.

\begin{proof}[Proof of Proposition~\ref{p:Ztimechange}]
  For any $f \in \mathcal D(\mathcal L^\epsilon)$, It\^o's formula gives
  \begin{align*}
    \MoveEqLeft
      \E f(Z^\epsilon_t) - f(Z^\epsilon_0)
      = \frac{1}{2}\E \int_0^{\psi^{\bar X, \epsilon}(t)}
	  \partial_x^2 f( \bar X_s, \bar Y_s)
	  \one_{\bar X_s \not\in \epsilon \Z} \, ds
    \\
      &+
	\frac{1}{2} \E \int_0^t
	  \paren[\Big]{
	    \partial_x f( (X^\epsilon_s)^+, Y^\epsilon_s )
	      - \partial_x f( (X^\epsilon_s)^-, Y^\epsilon_s )
	    }
	    \, dL^{X^\epsilon}_s(\epsilon \Z)
      \\
      &+ \frac{1}{2}\E \int_0^{\psi^{\bar Y, \epsilon}(t)}
	  \partial_y^2 f( \bar X_s, \bar Y_s)
	  \one_{\bar Y_s \in (0, 1)} \, ds
	+ \E \int_0^t
	  \partial_y f( X^\epsilon_s, (Y^\epsilon_s)^+ )
	  \, dL^{Y^\epsilon}_s(0)\,.
  \end{align*}
  Here we used the fact that $\qv{X^\epsilon, Y^\epsilon} = 0$ (Lemma~\ref{l:jqvXY}) and $\partial_y f(x, 1) = 0$ (which is guaranteed by the assumption $f \in \mathcal D( \mathcal L^\epsilon)$).
  Using~\eqref{e:claimLtratio} this simplifies to
  \begin{multline*}
      \E f(Z^\epsilon_t) - f(Z^\epsilon_0)
      = \E \int_0^{\psi^{\bar X, \epsilon}(t)}
	  \partial_x^2 f( \bar X_s, \bar Y_s)
	  \one_{\bar X_s \not\in \epsilon \Z} \, ds
	  \\
	+ \E \int_0^{\psi^{\bar Y, \epsilon}(t)}
	  \partial_y^2 f( \bar X_s, \bar Y_s)
	  \one_{\bar Y_s \in (0, 1)} \, ds
      \\
	+ \frac{1}{2}\E \int_0^t
	  \paren[\Big]{
	    \partial_x f( (X^\epsilon_s)^+, Y^\epsilon_s )
	      - \partial_x f( (X^\epsilon_s)^-, Y^\epsilon_s )
	      + \alpha \epsilon
		\partial_y f( X^\epsilon_s, (Y^\epsilon_s)^+ )
	    }
	    \, dL^{X^\epsilon}_s(\epsilon \Z) \,.
  \end{multline*}
  Since $f \in \mathcal D( \mathcal L^\epsilon)$ and $L^{X^\epsilon}$ only increases when $Y^\epsilon = 0$ and $X^\epsilon \in \epsilon \Z$, the last integral above vanishes.
  Consequently,
  \begin{equation*}
    \lim_{t \to 0} \frac{1}{t} \E \paren[\big]{ f(Z^\epsilon_t) - f(Z^\epsilon_0) } = \mathcal L^\epsilon f(0, 0)\,
  \end{equation*}
  showing that the generator of~$Z^\epsilon$ is $\mathcal L^\epsilon$ as claimed.
  The fact that $Z^\epsilon$ satisfies~\eqref{e:sdeXep} and~\eqref{e:sdeYep} follows immediately by choosing $f(x, y) = x$ and $f(x, y) = y$ respectively.
\end{proof}

It remains to prove Lemmas~\ref{l:claimLtratio} and~\ref{l:jqvXY}.
\begin{proof}[Proof of Lemma~\ref{l:claimLtratio}]
  We first claim that for any $t \geq 0$, we have
  \begin{equation}\label{e:psixPlusPsiy}
    \psi^{\bar X, \epsilon}(t) + \psi^{\bar Y, \epsilon}(t) = t\,.
  \end{equation}
  To see this, define the non-decreasing, right continuous function 
  \begin{equation*}
    H(t) \defeq \tau^{\bar Y} \paren[\Big]{ \frac{\alpha \epsilon}{2} L^{\bar X}_t( \epsilon \Z ) }\,.
  \end{equation*}
Using the properties of $\tau^{\bar Y}$, $L^{\bar X}$, $\tau^{\bar X,\epsilon}$, and $L^{\bar Y}$, it is easy to check that the right continuous inverse of $H$ is
\[
H^{-1}(t) = \inf \{ s > 0 \;|\;\; H(s) > t \} = \tau^{\bar X,\epsilon} \left( \frac{2}{\alpha \epsilon} L^{ \bar Y}_s(0) \right).
\]
Therefore, $\psi^{\bar X, \epsilon}$ and $\psi^{\bar Y, \epsilon}$ are the right continuous inverse functions of $t \mapsto t + H(t)$ and $t \mapsto t + H^{-1}(t)$, respectively, meaning that
\begin{align}
\psi^{\bar X, \epsilon}(t) & =  \inf \left \{ s \;|\;\; s + H(s) > t \right\}, \no \\
\psi^{\bar Y, \epsilon}(t) & = \inf \left \{ r \;|\;\; r + H^{-1}(r) > t \right\}. \no
\end{align}
In general, $H(H^{-1}(r)) \geq r$ and $H^{-1}(H(s)) \geq s$ must hold, but equality may not hold due to possible discontinuities in $H$ and $H^{-1}$.

Fix $t > 0$, and let $[t_0,t_1]$ be the maximal interval such that $t \in [t_0,t_1]$ and $\psi^{\bar X, \epsilon}$ is constant on the interval $[t_0,t_1]$.  Possibly $t_0 = t_1 = t$, but let us first suppose that the interval has non-empty interior, $t_0 < t_1$. This implies that $H(s)$ has a jump discontinuity at a point $s = \psi^{\bar X, \epsilon}(t_1)$ such that $s + H(s^-) = t_0$ and $s + H(s^+) = s + H(s) = t_1$.  Also, $H^{-1}(H(s)) = s$ must hold for such a value of $s$. So, for $\ell = H(s) = H(\psi^{\bar X, \epsilon}(t_1))$ we have
\[
\ell + H^{-1}(\ell) = H(s) + s = t_1.
\]
Therefore, $\psi^{\bar Y, \epsilon}(t_1) = \ell$, since
\[
\psi^{\bar Y, \epsilon}(t_1) = \inf \left \{ r \;|\;\; r + H^{-1}(r) > t_1 \right\}.
\]
This means that $\psi^{\bar Y, \epsilon}(t_1) = H(s)$.  Therefore, 
\[
\psi^{\bar Y, \epsilon}(t_1) + \psi^{\bar X, \epsilon}(t_1)  = H(s) + s = t_1
\]
must hold.  Now let extend the equality to the rest of the interval $[t_0,t_1]$. By assumption, $\psi^{\bar X, \epsilon}(t) = \psi^{\bar X, \epsilon}(t_1)$ for all $t \in [t_0,t_1]$.  Since $H$ has a jump discontinuity at $s$, this means $H^{-1}(r)$ is constant on the interval $[H(s^-),H(s)]$.  Hence, the function $r + H^{-1}(r)$ is affine with slope 1 on the interval $[H(s^-),H(s)] = [\psi^{\bar Y, \epsilon}(t_1) - (t_1 - t_0), \psi^{\bar Y, \epsilon}(t_1)]$. Therefore, for all $t \in [t_0,t_1]$, we must have
\[
\psi^{\bar Y, \epsilon}(t) = \psi^{\bar Y, \epsilon}(t_1) + t - t_1.
\]
This shows that for all $t \in [t_0,t_1]$, we have
\begin{align*}
\psi^{\bar X, \epsilon}(t) + \psi^{\bar Y, \epsilon}(t) = \psi^{\bar X, \epsilon}(t_1) + \psi^{\bar Y, \epsilon}(t_1) + t - t_1 = t.
\end{align*}

Applying the same argument with the roles of $\psi^{\bar X, \epsilon}$, $\psi^{\bar Y, \epsilon}$, $H$ and $H^{-1}$ reversed, we conclude that $\psi^{\bar X, \epsilon}(t) + \psi^{\bar Y, \epsilon}(t) = t$ must hold if either $\psi^{\bar X, \epsilon}$ or $\psi^{\bar Y, \epsilon}$ is constant on an interval containing $t$ which has non-empty interior.  The only other possibility is that both $\psi^{\bar X, \epsilon}$ and $\psi^{\bar Y, \epsilon}$ are strictly increasing through $t$. In this case, $H$ must be continuous at $\psi^{\bar X, \epsilon}(t)$ and $H^{-1}$ must be continuous at $\psi^{\bar Y, \epsilon}(t)$. Thus, $H^{-1}(H(\psi^{\bar X, \epsilon}(t)))= \psi^{\bar X, \epsilon}(t)$ and $H(H^{-1}(\psi^{\bar Y, \epsilon}(t))) = \psi^{\bar Y, \epsilon}(t)$ holds.  The rest of the argument is the same as in the previous case.
This proves~\eqref{e:psixPlusPsiy}.

  Now, since $X^\epsilon$ and $Y^\epsilon$ are time changes of $\bar X$ and $\bar Y$ respectively, we know that the local times are given by
  \begin{equation*}
    L^{X^\epsilon}_t
    \defeq L^{X^\epsilon}( \epsilon \Z)
    =  L^{\bar X}_{\psi^{\bar X,\epsilon}(t)}\,,
    \quad\text{and}\quad
    L^{Y^\epsilon}_t
    \defeq L^{Y^\epsilon}(0)
    = L^{\bar Y}_{\psi^{\bar Y,\epsilon}(t)} \,.
  \end{equation*}
  By definition of $\psi^{\bar X, \epsilon}$, we know
  \begin{equation*}
    t = \psi^{\bar X, \epsilon}(t)
      + \tau^{\bar Y} \paren[\Big]{
	  \frac{\alpha \epsilon}{2}
	  L^{\bar X}(\psi^{\bar X, \epsilon}(t))
	}\,.
  \end{equation*}
  Using~\eqref{e:psixPlusPsiy} this gives
  \begin{equation*}
    \psi^{\bar Y, \epsilon}(t)
      = \tau^{\bar Y} \paren[\Big]{
	  \frac{\alpha \epsilon}{2}
	  L^{\bar X}(\psi^{\bar X, \epsilon}(t))
	}\,,
  \end{equation*}
  and using the fact that $\tau^{\bar Y}$ is the inverse of $L^{\bar Y}$, we get~\eqref{e:claimLtratio} as desired.
\end{proof}

\begin{proof}[Proof of Lemma~\ref{l:jqvXY}]
 Fix $\delta > 0$, and define a sequence of stopping times $0 = \sigma_0 < \theta_1 < \sigma_1 < \theta_2 < \sigma_2 < \dots$ inductively, by
\begin{align}
\sigma_0 & = 0 \no \\
\theta_{k+1} & = \inf \left \{ t > \sigma_k \;|\; \text{either $Y^\epsilon_t = \delta$ or $d(X^\epsilon_t,\epsilon \mathbb{Z}) = \delta$} \right\}, \quad k = 0,1,2,3,\dots \no \\
\sigma_{k+1} & = \inf \left \{ t > \theta_k \;|\; \text{$Y_t = 0$ and $X^\epsilon_t \in \epsilon \mathbb{Z}$} \right\}, \quad k = 0,1,2,3,\dots \no
\end{align}

Then for $T > 0$, we decompose the joint quadratic variation over $[0,T]$ as
\[
\qv{X^\epsilon,Y^\epsilon}_{[0,T]} = \sum_{k \geq 0} \qv{X^\epsilon,Y^\epsilon}_{[\sigma_k \wedge T, \theta_{k+1} \wedge T]} + \qv{X^\epsilon,Y^\epsilon}_{[\theta_{k+1} \wedge T,\sigma_{k+1} \wedge T]}.
\]
We claim that for all $k$, 
\begin{equation}
\qv{X^\epsilon,Y^\epsilon}_{[\theta_{k+1} \wedge T,\sigma_{k+1} \wedge T]} = 0 \label{qvarzero}
\end{equation}
holds with probability one. Hence,
\begin{align}
\left|\qv{X^\epsilon,Y^\epsilon}_{[0,T]} \right| & \leq  \sum_{k \geq 0} \left| \qv{X^\epsilon,Y^\epsilon}_{[\sigma_k \wedge T, \theta_{k+1} \wedge T]} \right| \no \\
& \leq \sum_{k \geq 0} \frac{1}{2} \qv{X^\epsilon,X^\epsilon}_{[\sigma_k \wedge T, \theta_{k+1} \wedge T]} + \frac{1}{2} \qv{Y^\epsilon,Y^\epsilon}_{[\sigma_k \wedge T, \theta_{k+1} \wedge T]}  \no \\
& \leq \sum_{k \geq 0}|( \theta_{k+1} \wedge T )-  (\sigma_k \wedge T)| \no \\
& \leq \left| \left\{ t \in [0,T] \;|\; \;\; |\bar Y_t| \leq \delta, \;\;\text{and} \;\; d(\bar X_t, \epsilon \mathbb{Z}) \leq \delta \quad \right\} \right|.
\end{align}
As $\delta \to 0$, the latter converges to $0$ almost surely, which proves that $\qv{X^\epsilon,Y^\epsilon} = 0$.

To establish the claim \eqref{qvarzero}, we may assume $\theta_k < T$, for otherwise, the statement is trivial. At time $\theta_k$, we have either $X^\epsilon_{\theta_k} \notin  \epsilon \mathbb{Z}$ or $Y_{\theta_k} = \delta$.  In the former case, we must have $X_t \notin \epsilon \mathbb{Z}$ for all $t \in [\theta_k,\sigma_k)$.  Hence, $\psi^{\bar Y,\epsilon}(t)$ and $Y^\epsilon_t$ are constant for all $t \in [\theta_k,\sigma_k)$.  In the other case, $Y_t > 0$ for all $t \in [\theta_k,\sigma_k)$ while $X_t$ is constant on $[\theta_k,\sigma_k]$.  In either case, this implies that $\qv{X^\epsilon,Y^\epsilon}_{[\theta_k \wedge T, \sigma_k \wedge T]} = 0$ holds with probability one.
\end{proof}

\bibliographystyle{halpha-abbrv}
\bibliography{refs,preprints}
\end{document}